\documentclass[11pt,oneside,reqno]{amsart}

\usepackage[a4paper, total={450pt,675pt}]{geometry}
\usepackage[OT2,T1]{fontenc}
\usepackage[utf8]{inputenc}
\usepackage[american]{babel}
\usepackage{amssymb}
\usepackage{tocbasic}
\usepackage[dvipsnames]{xcolor}
\usepackage[
    colorlinks=true,
    linkcolor=Maroon,
    citecolor=JungleGreen,
    urlcolor=NavyBlue,
    pdfauthor=author]{hyperref}

\usepackage[capitalize,nameinlink]{cleveref}

\usepackage{enumitem}
\usepackage{dsfont}
\usepackage{tikz}
\usepackage{subcaption}
\usetikzlibrary{arrows}
\makeatletter
\def\@tocline#1#2#3#4#5#6#7{\relax
  \ifnum #1>\c@tocdepth % then omit
  \else
    \par \addpenalty\@secpenalty\addvspace{#2}%
    \begingroup \hyphenpenalty\@M
    \@ifempty{#4}{%
      \@tempdima\csname r@tocindent\number#1\endcsname\relax
    }{%
      \@tempdima#4\relax
    }%
    \parindent\z@ \leftskip#3\relax
    \advance\leftskip\@tempdima\relax
    \rightskip\@pnumwidth plus4em \parfillskip-\@pnumwidth
    #5\leavevmode\hskip-\@tempdima
      \ifcase #1
       \or\or \hskip 2em \or \hskip 2em \else \hskip 3em \fi%
      #6\nobreak\relax
    \dotfill\hbox to\@pnumwidth{\@tocpagenum{#7}}\par
    \nobreak
    \endgroup
  \fi}
\makeatother
\usepackage[
    backend=biber,
    style=alphabetic,
    giveninits=true,
    doi=false,
    maxbibnames=10,
    sorting=nyt]{biblatex}

\DeclareLabelalphaTemplate{
  \labelelement{
    \field[final]{shorthand}
    \field{label}
    \field[strwidth=1]{labelname}}
  \labelelement{
    \field[strwidth=2,strside=right]{year}}
}

\DeclareFieldFormat[article,incollection]{title}{#1}
\DeclareFieldFormat{date}{#1}
\DeclareFieldFormat{pages}{{#1}}
\DeclareFieldFormat{doi}{
\href{http://www.ams.org/mathscinet-getitem?mr=#1}{#1}}
\renewbibmacro{in:}{%
  \ifentrytype{article}{}{\printtext{\bibstring{in}\intitlepunct}}}

\newtheorem{theorem}{Theorem}[section]
\newtheorem{corollary}[theorem]{Corollary}

\newtheorem{lemma}[theorem]{Lemma}
\newtheorem{proposition}[theorem]{Proposition}
\newtheorem{conjecture}[theorem]{Conjecture}

\theoremstyle{definition}
\newtheorem{definition}[theorem]{Definition}
\newtheorem{example}[theorem]{Example}
\newtheorem{remark}[theorem]{Remark}

\crefname{section}{Sect.}{section}
\numberwithin{equation}{section}
\begin{document}
\begin{sloppypar}
%%%%%%%%%%%%%%%%%%%%%%%%%%%%%%%%%%%%%%%%%%%%%%%%%%%%%%%%%%%%
%%%%%%%%%%               TITLE PAGE               %%%%%%%%%%
%%%%%%%%%%%%%%%%%%%%%%%%%%%%%%%%%%%%%%%%%%%%%%%%%%%%%%%%%%%%
\title[Irrationalyy elliptic closed Reeb orbits]{Two irrationally elliptic closed orbits of Reeb flows on the boundary of star-shaped domain in $\mathbb{R}^{2n}$}

\author{Xiaorui Li, Hui Liu, Wei Wang}

\begin{abstract}
There are two long-standing conjectures in Hamiltonian dynamics concerning  Reeb flows on the boundaries of star-shaped domains in $\mathbb{R}^{2n}$ ($n \geq 2$). One conjecture states that such a Reeb flow  possesses either $n$ or infinitely many prime closed orbits; the other states that all the closed Reeb orbits  are irrationally elliptic when the domain is convex and the flow possesses finitely many prime closed orbits. In this paper, we prove that for dynamically convex Reeb flow on the boundary of a star-shaped domain in $\mathbb{R}^{2n}$ ($n \geq 2$) with exactly $n$ prime closed orbits, at least two of them must be irrationally elliptic.
\end{abstract}


%%%%%%%%%%% MSC2020
\makeatletter
\@namedef{subjclassname@2020}{\textnormal{2020}
    \it{Mathematics Subject Classification}}
\makeatother
%%%%%%%%%%% MSC2020

\maketitle
\tableofcontents

%%%%%%%%%%%%%%%%%%%%%%%%%%%%%%%%%%%%%%%%%%%%%%%%%%%%%%%%%%%%
%%%%%%%%%%                SECTION I               %%%%%%%%%%
%%%%%%%%%%%%%%%%%%%%%%%%%%%%%%%%%%%%%%%%%%%%%%%%%%%%%%%%%%%%
\section{Introduction}
\subsection{Set up and background}
In this paper, we focus on the stability problem for closed characteristics, interpreted as closed orbits of the Reeb flow on the boundary of a star-shaped domain in $\mathbb{R}^{2n}$. Specifically, let $\Sigma^{2n-1}$ bound a star-shaped domain in $\mathbb{R}^{2n}$. Throughout this paper, we assume $\Sigma$ is smooth. The restriction of the Liouville form 
\[
\lambda_{\mathrm{std}} = \frac{1}{2} \sum_{i=1}^{n} (y_i \mathrm{d}x_i - x_i \mathrm{d}y_i)
\] 
to $\Sigma$ yields a contact form $\alpha$, with $\xi = \ker \alpha$ defining the contact structure. The Reeb vector field $R$ of $\alpha$ on $\Sigma$ satisfies
\[
\alpha(R) = 1, \quad \mathrm{d}\alpha(R, \cdot) = 0.
\]
Denote its flow by $\phi_\alpha^t$. For a closed orbit $x$ of period $T$, we call $x$ \textit{prime} if $T$ is its minimal period. Two fundamental problems in Hamiltonian dynamics concern the multiplicity of prime closed Reeb orbits and their stability. We always identify the closed orbits $t\mapsto x(t)$ and $t\mapsto x(t+\theta)$ for all $\theta\in\mathbb{R}$.

The stability of a closed Reeb orbit $x$ is characterized by its \textit{Floquet multipliers}—the eigenvalues of the restricted linearized Poincaré return map 
\[
(\mathrm{d}\phi_\alpha^T)_{x(0)}|_{\xi}: \xi_{x(0)} \to \xi_{x(0)}.
\]
Let $\sigma(x)$ denote this eigenvalue set and $\mathbb{U} := \{ z \in \mathbb{C} : |z| = 1 \}$ the unit circle. We classify $x$ as:
\begin{itemize}

    \item \textit{hyperbolic} if $\sigma(x) \cap \mathbb{U} = \emptyset$;
    \item \textit{elliptic} if $\sigma(x) \subset \mathbb{U}$;
    \item \textit{irrationally elliptic} if $\mathrm{d}\phi_\alpha^T|_{\xi}$ decomposes into a direct sum of $n-1$ rotations
    \[
    \begin{pmatrix}
        \cos \theta_i & -\sin \theta_i \\
        \sin \theta_i & \cos \theta_i
    \end{pmatrix}
    \]
    on $2$-dimensional symplectic subspaces, with rotation angles $\theta_i$ being irrational multiples of $\pi$ for $1 \leq i \leq n-1$;
    \item \textit{non-degenerate} if $1\notin\sigma(x)$; in addition, the Reeb flow on $\Sigma$ is called non-degenerate if all closed orbits, prime or not, are non-degenerate.
\end{itemize}

Denote by $\mathcal{T}(\Sigma)$ the set of prime  closed Reeb orbits on $\Sigma$ (with respect to the contact form $\alpha=\lambda_{\rm std}|_{\Sigma}$). 
There is a long-standing conjecture concerning the 
multiplicity of closed Reeb orbits on compact star-shaped hypersurfaces in
$\mathbb{R}^{2n}$: 
\begin{conjecture}\label{conj:mul}
For each $\Sigma^{2n-1}$ bounding a star-shaped domain in $\mathbb{R}^{2n}$,  $^\#\mathcal{T}(\Sigma)\geq n$.
\end{conjecture}

Since the pioneering works of  Rabinowitz in \cite{Rab78} of 1978 (for star-shaped hypersurfaces) and Weinstein in \cite{Wei78} of 1978 (for convex hypersurfaces) showing $^\#\mathcal{T}(\Sigma)\geq 1$, the existence of multiple closed Reeb orbits has been deeply studied. For any convex hypersurface in $\mathbb{R}^{2n}$ with $n\geq 2$, Ekeland-Lassoued, Ekeland-Hofer and Szulkin in \cite{EL87} of 1987, \cite{EH87} of 1987 and \cite{Szu88} of 1988 proved $^{\#}\mathcal{T}(\Sigma)\ge 2$; Long-Zhu in \cite{LZ02} of 2002 further established the breakthrough result that $^{\#}\mathcal{T}(\Sigma) \ge\lfloor n/2\rfloor+1$ and $^{\#}\mathcal{T}(\Sigma)\geq n$ under non-degeneracy condition; Wang-Hu-Long in \cite{WHL07} of 2007 proved $\,^{\#}\mathcal{T}(\Sigma)\ge 3$ for $n=3$;  Wang in \cite{Wang16a} of 2016 showed $\,^{\#}\mathcal{T}(\Sigma)\ge \lfloor(n+1)/2\rfloor+1$ and in \cite{Wang16b} of 2016 demonstrated $\,^{\#}\mathcal{T}(\Sigma)\ge 4$ for $n=4$. Recent progress by \c{C}ineli-Ginzburg-G\"{u}rel in \cite{CGG24} of 2024 confirmed the conjecture  for $\Sigma$ bounding star-shaped domains with dynamically convex Reeb flows.

For star-shaped domains, we summarize some other key results: 
In \cite{Gir84} of 1984 and \cite{BLMR85} of 1985, $\;^{\#}\mathcal{T}(\Sigma)\ge n$ was established under some pinching conditions. 
Viterbo in \cite{Vit89} of 1989 proved generic existence of infinitely many closed Reeb orbits. 
Hu-Long in \cite{HL02} of 2002 showed $\;^{\#}\mathcal{T}(\Sigma)\ge 2$ under non-degeneracy condition. 
Cristofaro-Gardiner-Hutchings in \cite{CGH16} of 2016 proved $\;^{\#}\mathcal{T}(\Sigma)\ge 2$ for $n=2$ without pinching or non-degeneracy conditions, with alternative proofs in \cite{GHHM13} of 2013, \cite{GG15} of 2015 and \cite{LL16} of 2016. 
Duan-Liu-Long-Wang in\cite{DLLW18} and \cite{DLLW24}  established $\,^{\#}\mathcal{T}(\Sigma)\ge n$ under index  and non-degeneracy conditions, yielding at least $n$ ($n-1$) non-hyperbolic orbits for even (odd) $n$. 
Gutt-Kang in \cite{GK16} of 2016 proved $\,^{\#}\mathcal{T}(\Sigma)\ge n$ requiring non-degeneracy and Conley-Zehnder indices not less than $ n-1$. 
Duan-Liu in \cite{DL17} of 2017 and Ginzburg-Gürel in \cite{GG20} of 2020 showed $\,^{\#}\mathcal{T}(\Sigma)\ge \lfloor(n+1)/2\rfloor+1$ for dynamically convex flows independently. 
Additionally, Ginzburg-Gürel-Macarini in \cite{GGM18} of 2018 studied prequantization bundles, while Abreu-Macarini in \cite{AM17} of 2017 investigated dynamically convex contact forms, yielding various multiplicity results.

Note that in \cite{HWZ98} of 1998, Hofer-Wysocki-Zehnder
proved that $\,^{\#}\mathcal{T}(\Sigma)=2$ or $\infty$ holds for every $\Sigma$
bounding a convex domain in $\mathbb{R}^4$. Subsequently in \cite{HWZ03} of 2003, they proved the same result holds when $\Sigma$
bounds a star-shaped domain in $\mathbb{R}^4$ with non-degeneracy and transversality assumptions,
and  conjectured that:

{\it For each $\Sigma$ bounding a star-shaped domain in $\mathbb{R}^{4}$, there holds $^{\#}\mathcal{T}(\Sigma)\in \{2\} \cup \{\infty\}$.
}

 %In \cite{HWZ23} of 2003,
%H. Hofer, K. Wysocki and E. Zehnder proved that $\,^{\#}\mathcal{T}(\Sigma)=2$ or $\infty$ holds with non-degenerate assumption
 % (the same result was proved in \cite{HWZ98} for any $\Sigma\in \mathcal{H}_{cov}(4)$) provided that all the stable and unstable manifolds of the  closed %characteristics
%on $\Sigma$ intersect transversally. 
This conjecture was confirmed in \cite{CGHHL23b}, by proving a stronger theorem that the ``two or infinity'' result holds for  any closed contact 3-manifold whose contact structure possessing torsion first Chern class. For $n\geq3$, the analogous conjecture is still widely open:

\begin{conjecture}[cf. {\cite{WHL07, LL17, CGG24}}]
  For each $\Sigma^{2n-1}$ bounding a star-shaped domain in $\mathbb{R}^{2n}$, there holds \[^{\#}\mathcal{T}(\Sigma)\in\{n\} \cup \{\infty\}.\]
\end{conjecture}

 It seems interesting to study the property of closed Reeb orbits for the case $^{\#}\mathcal{T}(\Sigma)=n$.  Note that we have the following example of Harmonic Oscillator (cf. \cite[Section 1.7]{Eke90}):
\begin{example}\label{exam:HO}
Let $r=(r_1,\ldots, r_n)$ with $r_i>0$ for $1\le i\le n$.
Define 
\begin{equation*}\mathcal{E}_n(r)=\left\{z=(x_1, \ldots,x_n,
y_1,\ldots,y_n)\in\mathbb{R}^{2n}\left |\frac{}{}\right.
\frac{1}{2}\sum_{i=1}^n\frac{x_i^2+y_i^2}{r_i}=1\right\}.\end{equation*}
 In this case, the corresponding  system is linear and all the  solutions  
can be computed explicitly. Suppose $\frac{r_i}{r_j}\notin\mathbb{Q}$ whenever $i\neq j$, then
$^{\#}\mathcal{T}(\mathcal{E}_n(r)) = n$ and all the closed 
Reeb orbits
on $\mathcal{E}_n(r)$ are irrationally elliptic. On the other hand, if $\frac{r_i}{r_j}\in\mathbb{Q}$ for some $i\neq j$, then   $^{\#}\mathcal{T}(\mathcal{E}_n(r)) = \infty$.
\end{example}

In \cite{Long00} of 2000,  Long proved  that for
$\Sigma$ bounding a convex domain in $\mathbb{R}^4$, $\,^{\#}\mathcal{T}(\Sigma)=2$ implies that both of the
closed Reeb orbits must be elliptic. In \cite{WHL07}, the authors
proved further that  both of the two closed Reeb orbits must be irrationally
elliptic. (This result is generalized to any closed contact 3-manifold in \cite{CGHHL23} of 2023.)

Based on these facts,  Wang-Hu-Long in \cite{WHL07} made the following conjecture:

\begin{conjecture}\label{conj:ir elip}
For each $\Sigma^{2n-1}$ bounding a convex domain in $\mathbb{R}^{2n\geq4}$, assume $^\#\mathcal{T}(\Sigma)<\infty$, then all the closed Reeb orbits on $\Sigma$ are irrationally elliptic.
\end{conjecture}

\begin{remark}
     Note that in Example \ref{exam:HO}, if 
$r_1=r_2=\cdots=r_n=1$,  then every trajectory of the Reeb vector field is  closed  on
$\mathcal{E}_n(r)$, thus $^\#\mathcal{T}(\mathcal{E}_n(r))=\infty$. Moreover,   
$1$ is a Floquet multiplier  with  multiplicity
$2n-2$ for any closed Reeb orbit on $\mathcal{E}_n(r)$. 
Thus one can not hope to find irrationally elliptic closed Reeb orbit
on arbitrary $\Sigma$.
\end{remark}

In \cite{Wang09} of 2009, Wang proved that if $\Sigma$ bounds a convex domain in $\mathbb{R}^6$ with $^\#\mathcal{T}(\Sigma)=3$, then there exist at least two  elliptic closed Reeb orbits on $\Sigma$. This result was improved  in \cite{Wang14} of 2014, where it was shown that there exist at least two  irrationally elliptic closed Reeb orbits on $\Sigma$ under the same assumption. In \cite{HO17} of 2017, Hu-Ou established the existence of at least two elliptic closed Reeb orbits on $\Sigma$ under the assumption of Conjecture \ref{conj:ir elip}. In \cite{Wang22} of 2022, Wang demonstrated that under the condition ``$\rho_n(\Sigma)=n$'' (slightly weaker than non-degeneracy), if $\Sigma$ bounds a convex domain in $\mathbb{R}^{2n}$ with $^\#\mathcal{T}(\Sigma)=n$, then there are at least two irrationally elliptic closed Reeb orbits; this result also holds for star-shaped domains with dynamically convex Reeb flows. Our main result improves upon \cite{Wang22} by eliminating the technical condition ``$\rho_n(\Sigma)=n$'', as stated in Theorem \ref{thm:main}.

%{\bf Remark 1.5.}

There are some related results considering the stability problem. A long-standing conjecture concerning the stability of closed Reeb orbits on compact convex hypersurfaces in $\mathbb{R}^{2n}$ (cf. \cite[p. 235]{Eke90}) states:

\begin{conjecture}
For each $\Sigma^{2n-1}$ bounding a strictly convex domain in $\mathbb{R}^{2n}$, there exists at least one elliptic closed Reeb orbit on $\Sigma$.
\end{conjecture}

This conjecture has been verified under the following assumptions:
\begin{itemize}
    \item[-] Ekeland (1986) \cite{Eke86}: $\sqrt{2}$-pinched convex domains;
    \item[-] Dell'Antonio-D'Onofrio-Ekeland (1992) \cite{DDE92}: centrally symmetric domains;
    \item[-] Long-Zhu (2002) \cite{LZ02}: finite set of closed orbits ($^\#\mathcal{T}(\Sigma)<+\infty$);
    \item[-] Arnaud (1999) \cite{Arn99}, Liu-Wang-Zhang (2020) \cite{LWZ20}: domains invariant under specific rotation direct sums;
    \item[-] Abreu-Macarini (2022) \cite{AM22}: relaxing strict convexity to convexity in \cite{LWZ20}).
\end{itemize}
To the best of our knowledge, no lower bound for multiple  elliptic closed Reeb orbits is known under only the strictly convexity hypothesis.
%In Definition 1.2 of \cite{HWZ98},  H. Hofer-K. Wysocki-E. Zehnder
%introduced an interesting class of contact forms on $S^3$ which are
%called {\it dynamically convex contact forms}.

 %The concept of  dynamically convex was generalized to general cases
 %by several authors, cf. M. Abreu and L. Macarini in \cite{AM22}, H. Duan and H. Liu 
 %in \cite{DL},  J. Gutt and J. Kang in \cite{GK}, etc.  All these definitions are equivalent 
%when restricted to  the case that $\Sg$ is a compact star-shaped hypersurface
 %in $\R^{2n}$, since a  convex hypersurface is a special type of star-shaped 
 %hypersurface and the concept of dynamically convex generalize that of convex 
 %in the sense that all of them simply put the index condition satisfied by the convex case to the 
 %star-shaped case as its definition (e.g., the {\it Remark} in P. 222 pf \cite{HWZ98}).   
%Here we use  the definition in \cite{DL}. 
\subsection{Main results}
To state our results, we first briefly review some key terminology, which will be formally defined in Sections \ref{sec:pre} and \ref{sec:index}.

Let $(\Sigma^{2n-1},\alpha)$ be a contact manifold bounding a star-shaped domain in $\mathbb{R}^{2n}$, as introduced in the previous section. For a closed Reeb orbit $x$ on $\Sigma$ with period $T$, viewed as an immersion $S^1 = \mathbb{R}/T\mathbb{Z} \to \Sigma$, we define its action by
\begin{equation*}
    \mathcal{A}(x) := \int_{S^1} x^*\alpha.
\end{equation*}
By the definition of Reeb vector field, $\mathcal{A}(x)$ equals the period of $x$. For each $k\in\mathbb{N}$, the $k$-th iteration of $x$ is the closed Reeb orbit $y$ satisfying $\mathcal{A}(y) = k\mathcal{A}(x)$ and $y|_{[0,T]} = x$, denoted by $y = x^k$. The integer $k$ or iteration $x^k$ is  called \textit{admissible} if the algebraic multiplicity of eigenvalue 1 in $\sigma(x)$ and $\sigma(x^k)$ coincides. For instance, $x^k$ is admissible when the set $\{\lambda \in \sigma(x) : |\lambda|^k = 1\}$ is either $\{1\}$ or empty.

For each closed Reeb orbit $x$ on $\Sigma$, we associate two integer-valued Maslov-type indices $\mu_{\pm}(x)$, obtained via upper and lower semi-continuous extensions of the Conley-Zehnder index (see Section \ref{subsec:index}),  and a $\mathbb{R}$-valued index $\hat\mu(x)$, known as the {\it mean index}  (see Section \ref{subsec:index}). The Reeb flow on $\Sigma^{2n-1}$ is \textit{dynamically convex} if $\mu_-(x) \geq n+1$ for all closed Reeb orbits $x$. This notion was originally introduced in \cite{HWZ98} and subsequently generalized 
 by several authors, cf. Abreu-Macarini in \cite{AM22},  Duan-Liu 
 in \cite{DL17},  Gutt-Kang in \cite{GK16}, etc.  All these definitions are equivalent when  restricted to  the case where $\Sigma$ is the boundary of a star-shaped domain in $\mathbb{R}^{2n}$.

The following theorem is the main result of this paper:
\begin{theorem}\label{thm:main}
Let $\Sigma^{2n-1}\subset \mathbb{R}^{2n\geq4}$ be the boundary  of a star-shaped domain. Assume that the Reeb flow on $\Sigma$ is dynamically convex and has exactly $n$ prime closed orbits, then  at least two of them are irrationally elliptic.  
\end{theorem}

The proof is given in Section \ref{sec:two irra} and the main ingredients  include:
\begin{itemize}
    \item[-] Iteration theory of Maslov-type indices  (see Section \ref{subsec:ind iterate});
    \item[-] Commutative property of closed orbits in the common index jump intervals originally established in \cite{Wang16b} for $n=4$ and generalized to arbitrary $n$ in \cite{Wang22} (see Proposition \ref{prop:commut});
    \item[-] Key descriptions of local Floer-theoretic invariants recently given in \cite{CGG24} (see Section \ref{subsec:ind interval});
    \item[-] The following crucial observation (Theorem \ref{thm:local max}) proved in Section \ref{sec:local max}.
\end{itemize}

Fix a ground field $\mathbb{F}$. For each isolated closed Reeb orbit $x$, we define its \textit{equivariant local symplectic homology} $\mathrm{CH}_*(x;\mathbb{F})$, which serves as the Morse-theoretic counterpart of critical modules introduced in \cite{GM69a,GM69b}. The following theorem demonstrates that for certain closed Reeb orbit $x$, the support of the equivariant local symplectic homology concentrates at the Maslov-type index $\mu_+$.

\begin{theorem}\label{thm:local max}
Let $\Sigma^{2n-1} \subset \mathbb{R}^{2n\geq4}$ be the boundary of a star-shaped domain. Suppose $x$ is an isolated closed Reeb orbit on $\Sigma$ satisfying:
\begin{itemize}
    \item $\mathrm{CH}_{\mu_+(x)}(x;\mathbb{Q}) \neq 0$;
    \item $\mu_{+}(x) = d + n - 1$ for some integer $d$ with $|d - \hat{\mu}(x)| < \frac{1}{2}$.
\end{itemize}
Then 
\[
\mathrm{CH}_*(x;\mathbb{Q}) = 0 \quad \text{whenever} \quad * \neq \mu_{+}(x).
\]

Additionally, if $x = z^k$ is an admissible iteration of a closed Reeb orbit $z$, then for any admissible iteration $z^{k'}$,
\[
\mathrm{CH}_*(z^{k'};\mathbb{Q}) = 0 \quad \text{whenever} \quad * \neq \mu_{+}(z^{k'}).
\]
\end{theorem}

\begin{remark}
In fact, our proof in Section \ref{sec:local max} shows that Theorem \ref{thm:local max} holds whenever the absolute grading of $\mathrm{CH}_*(x;\mathbb{Q})$ is well-defined. This includes contact manifolds $(\Sigma^{2n-1},\xi)$ satisfying $c_1(\xi) = 0 \in H^2(\Sigma;\mathbb{Z})$ and $H_1(\Sigma;\mathbb{Q}) = 0$; see Remark \ref{rmk:frac grad}. The requirement of $\mathbb{Q}$-coefficients is essential and does not extend to fields of positive characteristic $p \geq 2$; see Remark \ref{rmk:coeffient field}.
\end{remark}

\begin{remark}
We conjecture that the second condition $\mu_{+}(x) = d + n - 1$ can be eliminated. This conjecture is motivated by an analogous result for equivariant critical modules of closed characteristics established in \cite[Proposition 2.6]{Wang09}. Some alternative attempt is discussed in Remark \ref{rmk:attempt}.
\end{remark}

\subsection{Organization and notation}
\subsubsection{Organization}
The paper is organized as follows. Sections \ref{sec:pre} and \ref{sec:index} present preliminary material, which is well-known to experts but included for the reader's convenience. In Sections \ref{subsec:local Floer} and \ref{subsec:local SH}, we review local Floer-theoretic invariants for closed orbits and their  properties, including local Floer homology, equivariant and non-equivariant local symplectic homology. Section \ref{subsec:GF} recalls the definition of generating functions in $\mathbb{R}^{2n}$.

The main focus of Section \ref{sec:index} are Maslov-type indices for symplectic paths,  extending from the Conley-Zehnder index, and their iteration theory.  This includes basic definitions and properties in Section \ref{subsec:index}, abstract precise iteration formulas in Section \ref{subsec:ind iterate}, and the Common Index Jump Theorem in Section \ref{subsec:CIJT}.

Section \ref{sec:local max} contains the proof of Theorem \ref{thm:local max} via an analogue for local Floer homology. In Section \ref{subsec:ind interval}, we summarize key arguments from \cite{CGG24}, and Section \ref{subsect:proof} presents the proof of Theorem \ref{thm:main}.

\subsubsection{Conventions and notation}\label{subsec:nontation}
 For the reader's convenience, we summarize key notations here.  Some will be restated in the context when they appear.

Let $\mathbb{N}$, $\mathbb{Z}$, $\mathbb{Q}$, $\mathbb{R}$, and $\mathbb{C}$ denote the sets of positive integers, integers, rational numbers, real numbers, and complex numbers respectively. Both $S^1$ and $\mathbb{U}$ denote the unit circle. We use $\mathbb{U}$ when it is viewed as a subset of $\mathbb{C}$. Let $\mathbb{F}$ denote an arbitrary field, and $\mathbb{F}_p$ denote a field with positive characteristic $p\geq 2$. Let $\mathbb{Z}_p$ denote the field given by $\mathbb{Z}/p\mathbb{Z}$ when integer $p\geq 2$ is prime.

For a linear transformation $M$, let $\sigma(M)$ denote its eigenvalue set.  Both $\mathrm{id}$ and $\mathrm{Id}$ denote the identity map. We consistently use $\mathrm{Id}$ when referring to the identity linear transformation. When the dimension of the underlying vector space on which the identity transformation acts needs to be specified, we use the notation $\mathrm{Id}_m$, where the subscript 
$m$ explicitly indicates this dimension.

The standard symplectic vector space $(\mathbb{R}^{2m}, \omega_{\mathrm{std}})$ has coordinates $(x_1,\ldots,x_m,y_1,\ldots,y_m)$ with symplectic form
\[
\omega_{\mathrm{std}} := \sum_{i=1}^m \mathrm{d}y_i \wedge \mathrm{d}x_i.
\]
Denote by $\mathrm{Sp}(2m)$ the symplectic group  equipped with the subspace topology from $\mathbb{R}^{4m^2}$, i.e.,
\begin{equation*}
  \mathrm{Sp}(2m)=\{M\in\mathrm{GL}(2m;\mathbb{R}): M^TJM=J\},
\end{equation*} 
where $\mathrm{GL}(2m;\mathbb{R})$ denotes the set of all the linear transformations on $\mathbb{R}^{2m}$ and $J:=\begin{pmatrix}
            0 & -\mathrm{Id}_m \\
            \mathrm{Id}_m & 0 
          \end{pmatrix}$.
For two symplectic matrices in block form:
\[
M_k = \begin{pmatrix}
A_k & B_k \\
C_k & D_k
\end{pmatrix} \in \mathrm{Sp}(2m_k), \quad k=1,2,
\]
their $\diamond$-product $M_1 \diamond M_2 \in \mathrm{Sp}(2m_1 + 2m_2)$ is defined as:
\[
M_1 \diamond M_2 = \begin{pmatrix}
A_1 & 0 & B_1 & 0 \\
0 & A_2 & 0 & B_2 \\
C_1 & 0 & D_1 & 0 \\
0 & C_2 & 0 & D_2
\end{pmatrix},
\]
representing the {\it symplectic direct sum} in the canonical coordinates of $\mathbb{R}^{2(m_1+m_2)}$.

%For a loop $\gamma(t)$ in $\mathrm{Sp}(2m)$, the Maslov index $\mu_{\mathrm{Maslov}}(\gamma)$ corresponds to its homotopy class in $\pi_1(\mathrm{Sp}(2m)) \simeq \mathbb{Z}$; see \cite[Section 2.2]{MS17} for details.

We use the floor and ceiling functions defined by:
\begin{equation*}\lfloor a\rfloor = \max{\{k \in \mathbb{Z}
\mid k \leq a\}},~~\lceil a\rceil = \min{\{k \in  \mathbb{Z} \mid k \geq
a\}}.
\end{equation*}

\section{Preliminary}\label{sec:pre}

%\subsection{Local Morse homology}

\subsection{Local Floer homology }\label{subsec:local Floer}

\subsubsection{Basic definition and properties}
Fix a ground field $\mathbb{F}$. Let $(W^{2n},\omega)$ be a symplectic manifold, and $H:S^1\times W^{2n}\to \mathbb{R}$ be a smooth 1-periodic  Hamiltonian on $W^{2n}$. Let $x$ be a 1-periodic orbit of $H$, i.e.,  a 1-periodic solution of the system
\begin{equation*}
    \dot{x}(t)=X_{H_t}(x(t)),
\end{equation*}
where $X_{H_t}$ denotes the Hamiltonian vector field defined by $\omega(X_{H_t},\cdot)=-\mathrm{d}H_t$. Suppose $x$ is isolated.  Let $\varphi_H^t$ be the Hamiltonian flow of $H$.
Choose a sufficiently small tubular neighborhood $U$ of $x$. As in \cite[Theorem 9.1]{SZ92} we can take a $C^2$-small perturbation $\tilde{H}$ of $H$ supported in $U$ such that all the 1-periodic orbits $\tilde{x}$ of $\tilde{H}$ entering $U$ are {\it non-degenerate}, that is, $\mathrm{d}\varphi_{\tilde{H}}^1(\tilde{x}(0)):T_{\tilde{x}(0)}W\to T_{\tilde{x}(0)}W$ has no eigenvalue equal to 1.  When $\tilde{H}$ is sufficiently $C^2$-close to $H$ and $U$ is sufficiently small, every Floer trajectory $u$ connecting two 1-periodic orbits of $\tilde{H}$ lying in $U$, whether broken or not, is also contained in $U$. ( This  follows from the energy separation argument induced by Gromov's compactness,  see \cite[Section 1.5]{Sal99}.) 

Choose a generic almost complex structure $J$ such that the regularity conditions are satisfied. Then we can construct a Floer complex over $\mathbb{F}$ on $U$ in the standard way: the generators are given by the 1-periodic orbits of $\tilde{H}$ contained in $U$ (hence close to $x$), graded by the Conley-Zehnder index, and the differential is given by counting Floer trajectories. The standard continuation argument in Floer theory shows that the homology of Floer complex is independent of the choice of $\tilde{H}$,  the almost complex structure, and the sufficiently small neighborhood $U$; see \cite{SZ92, Sal99, Gi10} for more details.
We refer to the resulting homology group $\mathrm{HF}_*(H,x)$ as the {\it local Floer homology} of $H$ at $x$.  We omit the coefficient filed $\mathbb{F}$ unless there is a specific notation.  Sometime we also denote the local Floer homology by $\mathrm{HF}_*(\varphi,x)$, where $\varphi=\varphi_H^1$ is the time-1 map.

The absolute grading of the local Floer homology is determined by fixing a symplectic trivialization of the tangent bundle $TW$ along the orbit $x$. This trivialization induces corresponding trivializations along all non-degenerate closed orbits $\tilde{x}$ that $x$ splits into, and yields the associated linearized symplectic paths used to define the Conley-Zehnder index (will be explained in Section \ref{subsec:index}). 

Furthermore, when the first Chern class $c_1(TW) = 0$, there exists a canonical construction of symplectic trivializations for contractible orbits, see Subsection \ref{subsect:index for orbits}. This ensures the Conley-Zehnder index is well-defined for such orbits. Under this condition, we always assume the grading is given by these Conley-Zehnder indices.

\begin{example}[{cf. \cite[Example 3.4]{Gi10}}]\label{exam:non degen HF}
    Assume $c_1(TW)=0$, $x$ is non-degenerate and contractible. Then the local Floer homology of $H$ at $x$ is given by
\begin{equation*}
    \mathrm{HF}_*(H,x)=\left\{
    \begin{aligned}
        &\mathbb{F},&~*=\mu(x),\\
        &0&~\text{otherwise,}
    \end{aligned}
    \right.
\end{equation*}
where $\mu(x)$ denotes the Conley-Zehnder index defined in Subsection \ref{subsect:index for orbits}. In fact, due to the non-degeneracy of $x$, we can find a  neighborhood $U$ of $x$ such that the Floer complex on $U$ has exactly one generator $x$. 
\end{example}

In general,  the single isolated 1-periodic orbit $x$ in this construction can be replaced by an isolated connected compact set of 1-periodic orbits of $H$; see \cite{Mcl12}.

We list some basic properties of local Floer homology we need in this paper. For a more detailed account, we refer the reader to \cite[Section 3.2]{Gi10}, \cite[Section 3.2]{GG10} and their reference.

\begin{itemize}
\item [\hypertarget{list:LF1}{(\text{LF1})}]
(Homotopy invariance) Let $H^s$, $s\in[0,1]$, be a family of Hamiltionians such that $x$ is a uniformly isolated one-periodic orbit for all $H^s$, i.e., $x$ is the only periodic orbit of $H^s$ for all $s$, in some open set independent of $s$. Then $\mathrm{HF}_*(H^s,x)$ is constant throughout the family. In particular, $\mathrm{HF}_*(H^0,x)=\mathrm{HF}_*(H^1,x)$.

\quad

\item [\hypertarget{list:LF2}{\text(LF2)}](K\"{u}nneth formula) Let $x_1$ and $x_2$ be 1-periodic orbits of the Hamiltonians $H_1$ and $H_2$, respectively, on the sympletic manifolds $M_1$ and $M_2$. Then 
\begin{equation*}
    \mathrm{HF}_*(H_1+H_2,(x_1,x_2))= \bigoplus_{*'+*''=*}\mathrm{HF}_{*'}(H_1,x_1)\otimes \mathrm{HF}_{*''}(H_2,x_2),
\end{equation*}
where $H_1+H_2$ is the naturally defined Hamiltonian on $M_1\times M_2$.
\end{itemize}

\quad

When $G$ and $H$ are two 1-periodic Hamiltonians, the composition $G\#H$ is defined by 
\[
(G\#H)_t=G_t+H_t\circ(\varphi_G^t)^{-1}.
\]
The flow of $G\#H$ is $\varphi_G^t\circ\varphi_H^t$. In general, $G\#H$ is not 1-periodic. However, when $\varphi_G^1=\mathrm{id}$, $G\#H$ is  1-periodic automatically. Consider $\overline{H}$  given by  $\overline{H}_t:=-H_t\circ\varphi_H^t$, then we have $\varphi_{\overline{H}}^t=(\varphi_H^t)^{-1}$ and $\overline{H}\#H=0=H\#\overline{H}$.
\begin{itemize}
\item [\hypertarget{list:LF3}{\text(LF3)}] Let $x(t)\equiv p$ be a constant orbit of the Hamiltonian $H$. Let $\varphi_G^t$ with $\varphi_G^0=\varphi_G^1=\mathrm{id}$ be a loop of Hamiltonian diffeomorphisms defined on a neighborhood of $p$ and fixing $p$. Then
\begin{equation*}
    \mathrm{HF}_{*+2\mu_{\rm Maslov}}(G\#H,p)=\mathrm{HF}_*(H,p),
\end{equation*}
where $\mu_{\rm Maslov}$ denotes the Maslov index of the loop $t\mapsto \mathrm{d}(\varphi_G^t)_p\in \mathrm{Sp}(2n)$.
\end{itemize}

 (For a loop $\gamma(t)$ in $\mathrm{Sp}(2n)$, the {\it Maslov index} $\mu_{\mathrm{Maslov}}(\gamma)$ corresponds to its homotopy class in $\pi_1(\mathrm{Sp}(2n)) \simeq \mathbb{Z}$; see \cite[Section 2.2]{MS17} for details.)

\begin{itemize}
    \item[\hypertarget{list:LF4}{(LF4)}]Define the support of $\mathrm{HF}_*(H,x)$ by $\mathrm{supp}\mathrm{HF}_*(H,x):=\{m\in\mathbb{Z}:\mathrm{HF}_m(H,x)\neq0\}$. Then we have
         \[\mathrm{supp}\mathrm{HF}_*(H,x)\subset[\mu_-(x),\mu_+(x)]\subset[\hat\mu(x)-n,\hat\mu(x)+n],\]
where $\mu_{\pm}$ denote the upper and lower Conley-Zehnder index and $\hat \mu$ denotes the mean index, which we will introduce in Section \ref{sec:index}.
\end{itemize}

\subsubsection{Local Floer homology via local Morse homology}
Let $f: W^m \to \mathbb{R}$ be a smooth function on a manifold $W$ (not necessarily symplectic) and $p \in W$ be an isolated critical point of $f$. The \textit{local Morse homology} $\mathrm{HM}_*(f, p)$ is defined by choosing an isolating neighborhood $U$ of $p$ and a small Morse perturbation $\tilde{f}$ of $f$. Specifically, the chain complex is generated by critical points of $\tilde{f}$ in $U$, with the differential defined via counting trajectories of the negative gradient flow between critical points; see \cite{Sch93,Gi10,AD14} for details. 

Compared to Floer homology, local Morse homology is more accessible due to its foundation in classical Morse theory, which provides an intuitive geometric interpretation of the homology groups through gradient flow dynamics. A fundamental property of Floer homology is that, for a $C^2$-small autonomous Hamiltonian $H$,  its Floer homology is equal to its Morse homology. In \cite{Gi10}, Ginzburg proved a slightly more general version of this fact, which allows us to understand the local Floer homology  of a ``relatively autonomous'' Hamiltonian (this concept was introduced in \cite[Lemma 4]{Hin09}) by the Morse homology of a reference function. This argument will play a crucial role in our poof of Theorem \ref{thm:local max}.

\begin{lemma}[{\cite[Lemma 3.6]{Gi10}}]\label{lem:KF}
Let $F$ be a smooth function and let $K$ be a $T$-periodic Hamiltonian, both defined on a neighborhood of a point $p$ in $W^{2n}$. Assume that $p$ is  an isolated critical point of $F$, and the following conditions are satisfied:
\begin{itemize}
    \item The inequalities $\Vert X_{K_t}(z)-X_F(z)\Vert\leq \epsilon \Vert X_F(z) \Vert$ and $\Vert \dot{X}_{K_t}(z)\Vert\leq\epsilon\Vert X_{F}(z)\Vert$ hold point-wise for $z$ near $p$ for all $t\in S^1$. (The dot stands for the derivative with respect to time.)
    \item The Hessian $\mathrm{d}^2(K_t)_p$ and $\mathrm{d}^2F_p$ and the constant $\epsilon\geq 0$ are sufficiently small. Namely, $\epsilon<1$ and 
    \begin{equation*}
        T \cdot(\epsilon(1-\epsilon)^{-1}+\max_t\Vert \mathrm{d}^2(K_t)_p\Vert +\Vert\mathrm{d}^2 F_p\Vert)<2\pi. 
    \end{equation*}   
\end{itemize}
Then $p$ is an isolated $T$-periodic orbit of $K$. Futhermore,
\begin{itemize}
    \item[(a)] $\mathrm{HF}_*(K,p)=\mathrm{HM}_{*+n}(F,p)$;
    \item[(b)] if $\mathrm{HF}_n(K,p)\neq0$, the functions $K_t$ for all $t$ and $F$ have a strict local maximum at $p$.
\end{itemize}
\end{lemma}

\begin{remark}\label{rmk:local maximal} We can infer from (a) and (b) immediately that, \[\mathrm{HF}_n(K,p)\neq0 \Rightarrow \mathrm{HF}_*(K,p)=0~\text{whenever}~ *\neq n,\] since in $W^{2n}$
 the Morse homology of a smooth function $F$ at a strict local maximum point $p$ is
 \begin{equation*}
     \mathrm{HM}_*(F,p)=\left\{
     \begin{aligned}
     &\mathbb{F},& *=2n;\\
     &0,&\text{otherwise;}
     \end{aligned}
     \right.
 \end{equation*}
 see \cite[Example 3.2]{Gi10}.
\end{remark}

\subsection{Local symplectic homology}\label{subsec:local SH}
\subsubsection{Basic definitions and facts}

Let $(\Sigma,\alpha)$ be a contact manifold with contact structure $\xi=\ker\alpha$, and $x$ be an isolated closed Reeb orbit, not necessarily prime, on $\Sigma$.  Fix an isolating neighborhood $U$ of $x(\mathbb{R})$ in $\Sigma$, which means that $U$ contains no other closed orbit of period close to $\mathcal{A}(x)$. Consider the symplectization $W=(1-\delta,1+\delta)\times U$ with coordinates $(r,z)$ and the symplectic form $\omega=\mathrm{d}(r\alpha)$. Take a Hamiltonian $h(r):W\to\mathbb{R}$ only depending on the $r$-coordinate with $h'(1)=\mathcal{A}(x)$ and $h''>0$. The Hamiltonian vector field of $h$ on $W$ is given by
\[
X_h(r,z)=(0,h'(r)R_\alpha).
\]
Then $x$ gives rise to an isolated set $S$ of 1-periodic orbits of $h$ with initial conditions on $x(\mathbb{R})$, which admits an $S^1$-action given by translation on time. Fix a ground field $\mathbb{F}$, the  {\it local symplectic homology} $\mathrm{SH}_*(x;\mathbb{F})$ is defined by the local Floer homology $\mathrm{HF}_*(h;\mathbb{F})$ of $h$ at $S$.

 More explicitly, we can describe $\mathrm{SH}_*(x;\mathbb{F})$ as follows. Consider a small non-degenerate perturbation of $\alpha$. Then $x$ splits into a finite collection of non-degenerate  closed Reeb orbits $x_i$.  By a Morse-Bott construction as in \cite{BO09mb}, $\mathrm{SH}_*(x;\mathbb{F})$ is the homology of a certain complex with generators $\check{x}_i$ and $\hat{x}_i$ which are born from $x_i$ and are of degrees $\mu(x_i)$ and $\mu(x_i)+1$ respectively, and the differential is given by counting the Morse-Bott broken trajectories; see also \cite{GG20,CGGM23}.

 The absolute grading of the local symplectic homology is determined by fixing a symplecic trivialization of the contact structure $\xi=\ker\alpha$ along $x$. (Note that such a trivialization gives rise to trivializations along all closed orbits $x_i$ that $x$ splits into.) Furthermore, when $c_1(\xi) = 0$ and $H^1(\Sigma;\mathbb{Q})=0$, there exists a canonical construction of symplectic trivializations for each closed orbits, which ensures the Conley-Zehnder index is well-defined; see Subsection \ref{subsect:index for orbits} and Remark \ref{rmk:frac grad}.  Under this condition, we always assume the grading is given by these Conley-Zehnder indices.

 The local symplectic homology has an $S^1$-equivariant counterpart, which we denoted by $\mathrm{CH}_*(x;\mathbb{F})$. We refer the reader to  \cite{BO13gysin,GG20} for the details of the definition. When $\mathbb{F}=\mathbb{Q}$,  consider a small perturbation of $\alpha$ such that $x$ splits into  non-degenerate  closed  orbits $x_i$ again. Then $\mathrm{CH}_*(x;\mathbb{Q})$ can be described as the homology of a complex generated by the  orbits $x_i$ and graded by the Conley-Zehnder index; see \cite[Section 2.3]{GG20}.

\begin{example}\label{exam:non-deg SH}
    Assume $x$ is  non-degenerate, then we have
    \begin{equation*}
        \mathrm{SH_*}(x;\mathbb{Z}_2)=\left\{
        \begin{aligned}
          \mathbb{Z}_2,~~~~~&*=\mu(x)~\text{or}~\mu(x)+1;\\
          0,~~~~~&\text{otherwise};
        \end{aligned}
        \right.
    \end{equation*}
see \cite[Proposition 2.2]{CFHW96}. In addition, assume $x=y^k$, $k\in\mathbb{N}$, where $y$ is the underlying prime closed orbit.  Non-degenerate orbit $x$ is called {\it bad} if  $\mu(x)$ and $\mu(y)$ have different parity, or equivalently,  $y$ has an odd number of Floquet multipliers (counted with algebraic multiplicity) in $(-1,0)$ and $k$ is even; otherwise, $x$ is called {\it good}. Then for a good orbit $x$, we have
    \begin{equation*}
        \mathrm{CH_*}(x;\mathbb{Q})=\left\{
        \begin{aligned}
          \mathbb{Q},~~~~~&*=\mu(x);\\
          0,~~~~~&\text{otherwise};
        \end{aligned}
        \right.
    \end{equation*}
while for a bad orbit $x$, $\mathrm{CH}_*(x;\mathbb{Q})=0$ in all degrees; see \cite[Example 2.19]{GG20}.    
\end{example}

\begin{remark}\label{rmk:coeffient field} As mentioned in \cite[Remark 2.6]{CGG24}, the description for $\mathrm{CH}_*(x;\mathbb{Q})$ does not carry over to a field of positive characteristic $p\geq2$. For instance, for a non-degenerate orbit $x=y^p$, $y$ is prime, $\mathrm{CH}_*(x;\mathbb{F}_p)$ is isomorphic to the homology of a infinite-dimensional lens space $S^\infty/\mathbb{Z}_p$ over $\mathbb{F}_p$.    
\end{remark}

As in \hyperlink{list:LF4}{(LF4)} of Section \ref{subsec:local Floer}, the supports of $\mathrm{SH}_*(x;\mathbb{F})$ or $\mathrm{CH}_*(x;\mathbb{Q})$ are also bounded by the Maslov-type indices. Namely, we have
\begin{equation}\label{eq:supp SH}
    \mathrm{supp}~\mathrm{SH_*}(x;\mathbb{F}) \subset [\mu_-(x),\mu_+(x)+1]\subset[\hat\mu(x)-n+1,\hat\mu(x)+n]
\end{equation}
and
\begin{equation}\label{eq:supp CH}
    \mathrm{supp}~\mathrm{CH_*}(x;\mathbb{Q}) \subset [\mu_-(x),\mu_+(x)]\subset[\hat\mu(x)-n+1,\hat\mu(x)+n-1].
\end{equation}
 Furthermore, it is shown in \cite{GG20} that equivariant and non-equivariant local symplectic homology fit into the Gysin exact sequence 
\begin{equation*}
  \dots\rightarrow \mathrm{SH}_*(x;\mathbb{F})\rightarrow \mathrm{CH}_*(x;\mathbb{F})
  \stackrel{\mathcal{D}}\rightarrow \mathrm{CH}_{*-2}(x;\mathbb{F})
  \rightarrow \mathrm{SH}_{*-1}(x;\mathbb{F})\rightarrow\cdots.
\end{equation*}
And by \cite[Proposition 2.21]{GG20}, the shift operator $\mathcal{D}$ vanishes when $\mathbb{F}=\mathbb{Q}$. Consequently, we have 
\begin{equation}\label{eq:SH split}
    \mathrm{SH}_*(x;\mathbb{Q})=\mathrm{CH}_*(x;\mathbb{Q})\oplus\mathrm{CH}_{*-1}(x;\mathbb{Q})
\end{equation}
since the Gysin sequence splits into short exact sequences.

 A closed Reeb orbit $x$ is called {\it$\mathbb{F}$-visible} if $\mathrm{SH}(x;\mathbb{F})\neq0$, is called {\it equivariantly $\mathbb{F}$-visible } if $\mathrm{CH}(x;\mathbb{F})\neq0$. By the Gysin exact sequence,  a orbit is $\mathbb{F}$-visible implies that it is equivariantly $\mathbb{F}$-visible. And $\mathbb{Q}$-visiblity is equivalent to equivariantly $\mathbb{Q}$-visiblity due to \eqref{eq:SH split}. 
 
 The {\it equivariant Euler characteristic} of $x$ is defined as
 \begin{equation*}
   \chi^{\rm eq}(x)=\sum_m(-1)^m\dim \mathrm{CH}_m(x;\mathbb{Q}).
 \end{equation*}
 For instance, $\chi^{\rm eq}(x)=0$ if $x$ is not $\mathbb{Q}$-visible. According to Example \ref{exam:non-deg SH}, if $x$ is non-degenerate, then $\chi^{\rm eq}(x)$ is either $(-1)^{\mu(x)}$ or 0 depending on whether $x$ is good or bad. 
 
 The following result from \cite{CGG24} will help us to determine $\mathbb{Q}$-visibility for certain orbit iterations in Section \ref{subsec:ind interval}. 
 \begin{lemma}[{\cite[Lemma 2.8]{CGG24}}]\label{lem:euler}
   Let $x$ be an isolated (not necessarily prime) closed Reeb orbit. Assume that $k$ is odd and iteration $x^k$ is admissible. Then
   \begin{equation*}
     \chi^{\rm eq}(x^k)=\chi^{\rm eq} (x).
   \end{equation*}
   In particular, $x^k$ is $\mathbb{Q}$-visible when $\chi^{\rm eq}(x)\neq 0$.
   
   Moreover, the condition $k$ is odd can be dropped provided $x$ has an even number of Floquet multipliers (counted with algebraic multiplicity) in $(-1,0)$. 
 \end{lemma}

\subsubsection{Specific local model}\label{subsect:local model}
 We introduce  a local model for what a contact manifold looks like along a closed Reeb orbit, which is given in \cite{HM15}: 

\begin{lemma}[{\cite[Lemma 5.2]{HM15}}]\label{lem:HM15}
    Let $x$ be a prime  $T$-period closed Reeb orbit of a contact manifold $(Y,\alpha)$ of dimension $2n-1$. Then there exists a tubular neighborhood $U=S^1\times B^{2n-2}$ of $x(\mathbb{R})$, where $B^{2n-2}\subset\mathbb{R}^{2n-2}$ is a small open ball centered at the origin, with coordinates $(t,q_1,\dots,q_{n-1},p_1,\dots,p_{n-1})$, such that $x(\mathbb{R})=S^1\times\{0\}$, $\alpha\simeq H\mathrm{d}t+\lambda_{0}$, where $H:U\to\mathbb{R}$ satisfies $H_t(0)\equiv T$, $\mathrm{d}H_t(0)\equiv0$, and $\lambda_0=\frac{1}{2}\sum_1^{n-1}q_i\mathrm{d}p_i-p_i\mathrm{d}q_i$.
\end{lemma}

 Such a specific local model allows us to give an explicit construction for {\it the associated Poincar\'{e} return map} of the Reeb flow near $x(0)$, with respect to a local cross section given by  the slice $\{0\}\times B^{2n-2}$:
 Firstly we look at the Reeb vector field on our local model. Note that $\alpha=H\mathrm{d}t+\lambda_0$ and that $\mathrm{d}\alpha=\mathrm{d}H\wedge\mathrm{d}t+\omega_0$, where $\omega_0=\mathrm{d}\lambda_0$ is the standard symplectic form on $B^{2n-2}$. Then one can  verify that the Reeb vector field of $\alpha$ is given by
 \begin{equation*}
     R=\frac{1}{F}\left(\frac{\partial}{\partial t}-X_H\right),
 \end{equation*}
 where $\iota_{X_H}\omega_0=-\mathrm{d}H$ and $F=H-\lambda_0(X_H)$ (which is nonzero at least in a small neighborhood of $x$). Then we see that $R$ has the same integral curve as the vector field $\frac{\partial}{\partial t}-X_H$ on $U=S^1\times B^{2n-2}$. This implies that 
 \begin{equation*}
  (t,\varphi_{-H}^t(z))=\phi_R^{\tau(t,z)}(0,z),~~\forall t\in(0,1],~ z\in B^{2n-2},
 \end{equation*}
 where $\varphi_{-H}^t$ denotes the Hamiltonian flow of $-H$ on $B^{2n-2}$, and $\tau(t,z)$ is defined by
 \begin{equation*}
     \tau(t,z)=\min\{s>0:\phi_R^s(0,z)\in\{t\}\times B^{2n-2}\subset U\}.
 \end{equation*}
We  conclude that the time-1 map $\varphi=\varphi_{-H}^1$ is exactly the associated Poincar\'e return map of $x(0)$,  following the Reeb flow from the slice $\{0\}\times B^{2n-2}$ back to itself.

It is shown in \cite{Fe20} that the local  symplectic homology of an isolated closed Reeb orbit can be computed by the local Floer homology of the associated return map, provided the closed Reeb orbit is prime. Note that  $0\in B^{2n-2}$ is a constant periodic orbit of the Hamiltonian flow of $-H$ since $\mathrm{d}H_t(0)\equiv0$.   For each isolated and prime closed Reeb orbit $x$,
\begin{equation}\label{eq:SH prime}
    \mathrm{SH}_*(x,\mathbb{F})\cong \mathrm{HF}_*(\varphi, 0;\mathbb{F})\oplus\mathrm{HF}_{*-1}(\varphi,0;\mathbb{F}),
\end{equation}
see \cite[Theorem 1.1]{Fe20}.

\begin{remark}\label{rmk:same ind}
The tubular neighborhood $U = S^1 \times B^{2n-2}$ in Lemma \ref{lem:HM15} corresponds to a symplectic trivialization of the contact structure $\xi = \ker\alpha$. (Note that $\xi$ can be identified with the normal bundle of $x$ in $Y$.) Under this trivialization, the Reeb flow along $x$ can be linearized and restricted to $\xi$, yielding a symplectic path used to define the Maslov-type indices of $x$; see Subsection  \ref{subsect:index for orbits}.

On the other hand, consider the Hamiltonian flow $\varphi_{-H}^t$ on the extended phase space $S^1 \times B^{2n-2}$, viewing $\varphi_{-H}^t$ as a map $\{0\} \times B^{2n-2} \to \{t\} \times B^{2n-2}$. Since $0$ is a constant orbit of $\varphi_{-H}$ in $B^{2n-2}$, we can linearize $\varphi_{-H}^t$ along $(t,0)$ and restrict it to $B^{2n-2}$, obtaining a symplectic path that defines the Maslov-type indices of $0$. As the Reeb vector field $R$ shares integral curves with the vector field $\frac{\partial}{\partial t} - X_H$, these two symplectic paths—associated with $x$ and $0$ respectively—coincide up to reparameterization and consequently have identical Maslov-type indices.
\end{remark}
 For the k-th iteration $x^k$, consider the $k$-fold covering 
\[\pi:U_k=S^1\times B^{2n-2}\to U,~~~(e^{\sqrt{-1}\theta},z)\mapsto(e^{\sqrt{-1}k\theta},z),\] 
and the pull-back of the contact form from $U$ to $U_k$. This pullback and the Reeb vector field are invariant under the action of $\mathbb{Z}_k$ on $U_k$ given by deck transformations with generator $(e^{\sqrt{-1}\theta},z)\mapsto(e^{\sqrt{-1}(\theta+\frac{2\pi}{k})},z)$. Provided $x^k$ is isolated, then $x^k$ can be lifted to an isolated closed Reeb orbit on $U_k$, with up to $k$ distinct choices of initial values (which correspond to the same closed Reeb orbit under different initial conditions), denoted by $kx$.  The associated return map of $kx$ is $\varphi^k$. More precisely, assume $\mathcal{A}(x)=T$ and the minimal period of $x$ is $T/l$ for some $l\in\mathbb{N}$, we have
\begin{equation*}
  x^k:[0,kT]\to U=S^1\times B^{2n-2},~~t\mapsto(e^{\frac{2\pi\sqrt{-1} lt}{T}},0).
\end{equation*}
And the  lift $kx$ is given by 
\begin{equation}\label{eq:kx}
  kx:[0,kT]\to U_k=S^1\times B^{2n-2},~~t\mapsto(e^{\frac{2\pi\sqrt{-1} (lt+mT)}{kT}},0),
\end{equation}
for some $m\in\{0,\dots,k-1\}$.  

\begin{lemma}[{\cite[Lemma 2.9]{CGG24}}]
There is a natural $\mathbb{Z}_k$-action on $\mathrm{SH}(kx;\mathbb{F})$, and 
\begin{equation}\label{eq:SH kx}
    \mathrm{SH}_*(x^k;\mathbb{Q})=\mathrm{SH}_*(kx;\mathbb{Q})^{\mathbb{Z}_k},
\end{equation}
where the superscript refers to the $\mathbb{Z}_k$-invariant part of the homology.
\end{lemma}

According to \eqref{eq:kx}, $kx$ is prime whenever $x$ is prime. Then combining \eqref{eq:SH prime} and \eqref{eq:SH kx} we get
\begin{equation}\label{eq:SH leq HF}
 \dim\mathrm{SH}_*(x^k; \mathbb{Q})\leq \dim \mathrm{HF}_*(\varphi^k, 0; \mathbb{Q})+\dim\mathrm{HF}_{*-1}(\varphi^k, 0; \mathbb{Q}).
\end{equation}

\subsection{Generating functions}\label{subsec:GF}

In this subsection we recall the definition of the generating function on $\mathbb{R}^{2n}$, which will be used  to construct the reference function for applying Lemma \ref{lem:KF}  in our proof of Theorem \ref{thm:local max}. We closely follow the description in \cite[Section 6.1]{Gi10} and the references therein. 

Consider the symplectic vector space $\mathbb{R}^{2n}\times \bar{\mathbb{R}}^{2n}:=(\mathbb{R}^{2n}\oplus \mathbb{R}^{2n},\omega_{\rm std}\oplus-\omega_{\rm std})$, and identify $\mathbb{R}^{2n}$ with the Lagrangian diagonal $\triangle\subset\mathbb{R}^{2n}\times \bar{\mathbb{R}}^{2n}$.  Fix a Lagrangian complement $N$ to $\triangle$. Then  $\mathbb{R}^{2n}\times \bar{\mathbb{R}}^{2n}$ can be treated as $T^*\triangle$ via a symplectomorphisms sending $\triangle$ to the zero section and identifying $N$ with each cotangent fiber by a canonical isomorphism $\eta\mapsto (\omega_{\rm std}\oplus-\omega_{\rm std})(\eta,\cdot)$. (This identification is  a special case of  Weinstein's Lagrangian Neighborhood Theorem; see \cite{Wei71,MS17} for  details.)

Let $\varphi$ be a Hamiltonian diffeomorphism defined on a neighborhood of the origin $p$ in $\mathbb{R}^{2n}$ such that $\Vert\varphi-\mathrm{id}\Vert_{C^1}$ is sufficiently small. Then the graph $\Gamma$ of $\varphi$ is $C^1$-close to $\triangle$. Consequently, $\Gamma$ is transverse to every fiber and can be viewed as the graph in $T^*\triangle$ of an exact 1-form $\mathrm{d}F$ near $p\in\triangle=\mathbb{R}^{2n}$. The function $F$, normalized by $F(p)=0$, is called the {\it generating function} of $\varphi$ and  possesses the following properties:
\begin{itemize}
    \item $p$ is an isolated critical point of $F$ if and only if $p$ is an isolated fixed point of $\varphi$,
    \item $\Vert F\Vert_{C^2}=O(\Vert\varphi-\mathrm{id}\Vert_{C^1})$ and $\Vert\mathrm{d}^2F_p\Vert=\Vert\mathrm{d}\varphi_p-\mathrm{Id}\Vert$.
\end{itemize}

In general the function $F$ depends on the choice of the Lagrangian complement $N$ to $\triangle$. In this paper we  take as $N$ the specific linear subspace
\[
N_0:=\{(x,0,0,y)\}\subset\mathbb{R}^{2n}\times \bar{\mathbb{R}}^{2n},
\]
where $x=(x_1,\dots,x_n)$ and $y=(y_1,\dots,y_n)$ are the standard canonical coordinates on $\mathbb{R}^{2n}$, i.e., $\omega_{\rm std}=\sum\mathrm{d}y_i\wedge\mathrm{d}x_i$.  Then the   symplectomorphisms between  $\mathbb{R}^{2n}\times \bar{\mathbb{R}}^{2n}$ and $T^*\triangle$ can be explicitly expressed as $\Psi_{N_0}:\mathbb{R}^{2n}\times \bar{\mathbb{R}}^{2n}\to T^*\triangle$,
\begin{equation*}
   \Psi_{N_0}(x,y,x',y'):=(x',y ,y-y',x'-x)=(q,q',p,p'), 
\end{equation*}
where $(q,q',p,p')$ are the canonical coordinates in $T^*\triangle=T^*\mathbb{R}^{2n}$ such that the symplectic form $\omega_{\rm can}=\sum (\mathrm{d}q_i\wedge\mathrm{d}p_i+\mathrm{d}q'_i\wedge\mathrm{d}p'_i).$ 
Then  $F$ is determined by the equation
\begin{equation}\label{eq:graph F}
\Psi_{N_0}((x,y),\varphi(x,y))=((q,q'),\mathrm{d}F(q,q')),   
\end{equation}
where the left-hand side represents the graph of $\varphi$ mapped by the symplectomorphism $\Psi_{N_0}$, and the right-hand side is the graph of a 1-form in $T^*\mathbb{R}^{2n}$.

Set $z=(x,y)$, then \eqref{eq:graph F} can be rewritten as
\begin{equation*}
    \varphi(z)-z=X_F(\psi(z)),
\end{equation*}
where  $X_F$ is the Hamiltonian vector field of $F$ and 
$\psi(z):=(\text{$x$-component of $\varphi(z)$},y)$.

\section{Maslov-type indices: definition and iteration properties}\label{sec:index}

\subsection{Maslov-type indices and dynamically convexity}\label{subsec:index}
Let $\mathrm{Sp}(2m)$ denote the group of symplectic linear transformations on $\mathbb{R}^{2m}$ (equipped with the standard symplectic form). In this section we recall the definition and some basic properties of the various Maslov-type indices for sympletcic paths, i.e. paths in  $\mathrm{Sp}(2m)$.
The index theory was introduced by Conley-Zehnder in \cite{CZ84} for the non-degenerate case with $m\ge 2$,
Long-Zehnder in \cite{LZ90} for the non-degenerate case with $m=1$,
and Long in \cite{Long90} and Viterbo in \cite{Viterbo90} independently for
the degenerate case. For a more thorough treatment, we refer the reader to \cite{Long02}, \cite[Section 3]{SZ92} or \cite[Section 4]{GG20}.

\subsubsection{The mean index}
Let $\widetilde{\rm Sp}(2m)$ denote the universal covering of $\mathrm{Sp}(2m)$.  An element $\Phi\in\widetilde{\rm Sp}(2m)$ can be viewed as a path $\Phi(t):[0,1]\to\mathrm{Sp}(2m)$ with $\Phi(0)=\mathrm{Id}$, up to homotopy with fixed end-points. 

\begin{remark}\label{rmk:multiplicate to concatenation}
   For $\Phi$ and $\Psi\in\widetilde{\rm Sp}(2m)$, consider the matrix product path $\Psi(t)\Phi(t):[0,1]\to \mathrm{Sp}(2m)$. One can construct a homotopy with fixed end-points, such that 
   \begin{equation}\label{eq:multiplicate to concatenation}
    \Psi(t)\Phi(t)\sim \Phi(t)*\Psi(t)\Phi(1):=\left\{\begin{aligned}
         &\Phi(2t), & 0\le t\le \frac{1}{2}, \cr
            &\Psi(2t-1)\Phi(1), & \frac{1}{2}\le t\le 1. 
    \end{aligned}\right.
   \end{equation}
   More precisely, the homotopy is given by $\Gamma:[0,1]\times[0,1]\to\mathrm{Sp}(2m)$,
   \begin{equation*}
\Gamma(s,t):=\left\{
\begin{aligned}
    &\Phi\left(\frac{2t}{2-s}\right), &0\leq t\leq\frac{s}{2},\\
    &\Psi\left(\frac{2t-s}{2-s}\right)\Phi\left(\frac{2t}{2-s}\right), &\frac{s}{2}\leq t \leq \frac{2-s}{2},\\
    &\Psi\left(\frac{2t-s}{2-s} \right)\Phi(1), &\frac{2-s}{2}\leq t\leq 1.
\end{aligned}
\right.
\end{equation*}   
   Therefore, the multiplication in $\widetilde{\rm Sp}(2m)$ can be viewed as the concatenation of paths. From this opinion, for each $k\in\mathbb{N}$, $\Phi^k\in\widetilde{\rm Sp}(2m)$ is identified with  the $k$-th iteration of symplectic path $\Phi(t)$ defined in, for example, \cite[Page 322]{LZ02}.
\end{remark}

For each $\Phi\in\widetilde{\rm Sp}(2m)$,  one can define its {\it  mean index} $\hat\mu(\Phi)$, which measures the total rotation angle of certain unit eigenvalues of $\Phi(t)$, as follows. Following \cite{SZ92}, one can construct a continuous function $\rho:\mathrm{Sp}(2m)\to S^1$, uniquely characterized by the following properties:
\begin{itemize}
    \item (Naturality) If $A$ and $B\in\mathrm{Sp}(2m)$, then $\rho(BAB^{-1})=\rho(A)$.  
    \item (Product) If $A\in\mathrm{Sp}(2m_1)$ and $B\in\mathrm{Sp}(2m_2)$ with $m_1+m_2=m$, then \[\rho(A\diamond B)=\rho(A)\rho(B).\]
    \item (Determinant) If $A\in \mathrm{U}(m)=\mathrm{Sp}(2m)\cap\mathrm{O}(2m)$, then 
    \[
     \rho(A)=\mathrm{det}_{\mathbb{C}}(X+\sqrt{-1}Y), \text{~where $A=\begin{pmatrix}
         X&-Y\\
         Y& X
     \end{pmatrix}$,  $X$, $Y\in\mathbb{R}^{m\times m}$.}
    \]
  
    In particular, $\rho$ induces an isomorphism $\rho_*:\pi_1(\mathrm{Sp}(2m))\to\pi_1(S^1)\simeq\mathbb{Z}$.
    \item (Normalization) If $\sigma(A)\cap S^1\subset\mathbb{R}$, then \[\rho(A)=(-1)^{m_0/2},\] where $m_0$ is the total  multiplicity of the negative real eigenvalues.
\end{itemize}
Consider the path $\rho(\Phi(t)):[0,1]\to S^1$ and lift it to a path $\theta:\mathbb[0,1]\to\mathbb{R}$ such that $\rho(\Phi(t))=\exp(\sqrt{-1}\theta(t))$. We set 
\[
\hat\mu(\Phi):=\frac{\theta(1)-\theta(0)}{\pi}.
\]
The properties of $\rho$  imply the following properties of mean index:
\begin{itemize}
    \item[\hypertarget{list:MI1}{\text{(MI1)}}] (Additivity) $\hat\mu(\Phi\diamond\Psi)=\hat\mu(\Phi)+\hat\mu(\Psi)$, where $\Phi\in\widetilde{\rm Sp}(2m_1)$ and $\Psi\in\widetilde{\rm Sp}(2m_2)$.   
    \item[\hypertarget{list:MI2}{\text{(MI2)}}]  (Maslov index) $\hat\mu$ restricts to an isomorphism $\pi_1(\mathrm{Sp}(2m))\to 2\mathbb{Z}$. Indeed, \[\hat\mu(\Phi)=2\mu_{\rm Maslov}(\Phi)\] provided $\Phi$ is a loop. Moreover, for each $\Psi\in\widetilde{\rm Sp}(2m)$, if $\psi$ is a loop in $\mathrm{Sp}(2m)$ with $\psi(0)=\Psi(1)$, then we have
    \[
    \hat\mu(\Psi*\psi)=\hat\mu(\Psi)+2\mu_{\rm Maslov}(\psi),
    \]
    where $*$ denotes the  concatenation of paths.
     \item[\hypertarget{list:MI3}{\text{(MI3)}}] (Integer case) $\hat\mu(\Phi)\in\mathbb{Z}$ provided $\sigma(\Phi(1))\cap S^1\subset\mathbb{R}$. In particular, $\hat\mu(\Phi)\in 2\mathbb{Z}$ provided $\sigma(\Phi(1))=\{1\}.$
     \item[\hypertarget{list:MI4}{\text{(MI4)}}] (Homogeneity) $\hat\mu(\Phi^k)=k\hat\mu(\Phi)$, for each $k\in\mathbb{N}$. (cf. \cite[Lemma 3.4]{SZ92}.)
\end{itemize} 

Furthermore, by the explicit construction of $\rho$ in \cite[Page 1316]{SZ92} (see also \cite[Page 60]{Long02}), $\rho(A)$ is totally determined by $\sigma(A)$ and the Krein type of eigenvalues in $\sigma(A)\cap S^1$, which, as mentioned in \cite[Section 4.1.1]{GG20}, implies that: 
\begin{itemize}
    \item[\hypertarget{list:MI5}{\text{(MI5)}}]$\hat\mu(\Phi_s)\equiv const$ for a family of elements $\Phi_s$ in $\widetilde{\rm Sp}(2m)$ as long as $\sigma(\Phi_s(1))\equiv const$ for all $s$.
    \item[\hypertarget{list:MI6}{\text{(MI6)}}] $\hat \mu(\Phi)=0$, provided $\sigma(\Phi(t))\equiv\{1\}$ for $t\in[0,1]$.
\end{itemize}

\begin{remark}
    In fact, $\hat\mu:\widetilde{\rm Sp}(2m)\to\mathbb{R}$ is a quasimorphism between Lie groups, that is,
    \[
    |\hat\mu(\Phi\Psi)-\hat\mu(\Phi)-\hat\mu(\Psi)|<const,
    \]
  where the constant is independent of $\Phi$ and $\Psi$.  And one can prove that $\hat\mu$ is the unique quasimorphism $\widetilde{\rm Sp}(2m)\to\mathbb{R}$ which is continuous, homogeneous (see \hyperlink{list:MI4}{(MI4)}), and normalized by 
  \[
  \hat\mu(\Phi_0)=2~ \text{~for $\Phi_0(t)=\exp(2\pi\sqrt{-1}t)\diamond \mathrm{Id_{2m-2}};$}
  \]
  see \cite{BG92}.
  Although $\hat\mu$ fails to be a homomorphism in general, according to \eqref{eq:multiplicate to concatenation}, we still have
  \begin{equation}\label{eq:homo loop}
  \hat\mu(\Psi\Phi)=\hat\mu(\Psi)+\hat\mu(\Phi), ~~\text{provided $\Psi\in\widetilde{\rm Sp}(2m)$ is a loop.}
  \end{equation}
\end{remark}

\subsubsection{Maslov-type indices}

Recall that $\Phi\in\widetilde{\rm Sp}(2m)$ is called {\it non-degenerate} if $1\notin\sigma(\Phi(1))$;  {\it totally degenerate} if $\sigma(\Phi(1))=\{1\}$;   {\it strongly non-degenerate} if $\Phi^k$ is non-degenerate for each $k\in\mathbb{N}$, i.e., none of the eigenvalues of $\Phi(1)$ is a root of unity. Let $\mathrm{Sp}(2m)^*$ denote the set of $A\in\mathrm{Sp}(2m)$ such that $1\notin\sigma(A)$, and $\widetilde{\rm Sp}(2m)^*$ denote the set of $\Phi\in\widetilde{\rm Sp}(2m)$ such that $\Phi$ is non-degenerate. %The set $\mathrm{Sp}(2m)^*$ consists of exactly two path-connected components, differed by the sign of $\det(A-\mathrm{Id})$. 
For each $\Phi\in\widetilde{\rm Sp}(2m)$, one can connect $\Phi(1)$ to some symplectic transformations $M$  with $\rho(M)\in\{\pm1\}$ by a path lying entirely  in $\mathrm{Sp(2m)^*}$. Concatenating this path with $\Phi$ to get a new path $\Phi'\in\widetilde{\rm Sp}(2m)$ such that $\hat\mu(\Phi')\in\mathbb{Z}$. Then the {\it Conley-Zehnder index} of $\Phi$ is defined by $\mu(\Phi):=\hat\mu(\Phi')$. One can show that $\mu(\Phi)$ is well-defined and $\mu:\widetilde{\rm Sp}(2m)^*\to\mathbb{Z}$ is constant on connected components of $\widetilde{\rm Sp}(2m)^*$, namely, it is locally constant. (cf. \cite[Theorem 3.3]{SZ92}) Also we have
\begin{equation}\label{eq:parity}
    (-1)^{\mu(\Phi)-m}=\mathrm{sign} \det(\mathrm{Id}-\Phi(1)),
\end{equation}
which implies that the parity of $\mu(\Phi)$ is determined by  the number of eigenvalues of $\Phi(1)$ in $(-1,0)$, counted with algebraic multiplicity.
And we list here a specific example that for the path $\Phi(t)=\exp(2\pi\sqrt{-1}\lambda t)$, $t\in[0,1]$, in $\mathrm{Sp}(2)$ we have
\begin{equation*}
    \mu(\Phi)=\mathrm{sign(\lambda)}(2\lfloor |\lambda| \rfloor+1)~~  \text{when}~\lambda\notin\mathbb{Z}.
\end{equation*}

In general, for a not necessarily non-degenerate $\Phi\in\widetilde{\rm Sp}(2m)$, one can define its {\it upper and lower Conley-Zehnder indices} as
\begin{equation*}
    \mu_+(\Phi):=\limsup_{\Psi\to\Phi}\mu(\Psi)~~\text{ and }~~\mu_-:=\liminf_{\Psi\to\Phi}\mu(\Psi),
\end{equation*}
where in both cases the limit is taken over $\Psi\in\widetilde{\rm Sp}(2m)^*$ converging to $\Phi$, see \cite[Pages 63-64]{Long97}. The indices $\mu_{\pm}$ are the upper semi-continuous and lower semi-continuous extensions of $\mu$ from $\widetilde{\rm Sp}(2m)^*$ to $\widetilde{\rm Sp}(2m)$ respectively. Clearly, $\mu_{\pm}(\Phi)=\mu(\Phi)$ provided $\Phi$ is non-degenerate. Moreover,  we have
\begin{equation}\label{eq:+-nullity}
    \mu_+(\Phi)-\mu_-(\Phi)=\nu(\Phi):=\dim\ker(\Phi(1)-\mathrm{Id}),~ \forall\Phi\in\widetilde{\rm Sp}(2m),
\end{equation}
see \cite[Page 142]{Long02}. 

The indices $\mu_{\pm}$ and $\mu$ possess the same additivity as $\hat\mu$ in \hyperlink{list:MI1}{(MI1)}, see \cite[Theorem 1.4]{Long97} or \cite[Lemma 4.3]{GG20}. For each $\Phi\in\widetilde{\rm Sp}(2m)$, the mean index $\hat\mu$ and the indices $\mu_{\pm}$ are related by the inequalities
\begin{equation}\label{eq:bounded by mean}
    \hat\mu(\Phi)-m\leq \mu_-(\Phi)\leq\mu_+(\Phi)\le\hat\mu(\Phi)+m.
\end{equation}
(cf. \cite[Page 213]{Long02}.) Consequently, 
\begin{equation}\label{eq:mean ind}
    \lim_{k\to\infty}\frac{\mu_{\pm}(\Phi^k)}{k}=\hat\mu(\Phi).
\end{equation}
Furthermore, suppose $\psi$ is a loop in $\mathrm{Sp}(2m)$  with $\psi(0)=\Phi(1)$, then one can prove that
\begin{equation}\label{eq:differ by maslov}
    \mu_{\pm}(\Phi*\psi)-\mu_{\pm}(\Phi)=\hat\mu(\Phi*\psi)-\hat\mu(\Phi)=2\mu_{\rm Maslov}(\psi),
\end{equation}
where the second equality is from \hyperlink{list:MI2}{(MI2)}. (Indeed, the first equality can be derived from Corollary \ref{coro:precise mean ind} of Section \ref{subsec:ind iterate}.)

\subsubsection{Dynamical convexity}

As mentioned before, the notion of dynamical convexity, originally introduced in \cite{HWZ98} for the Reeb flow on the standard contact sphere $S^3$, is a condition arguing some lower bound of the indices $\mu_-$ for closed Reeb orbits. In this section, essentially focus on the symplectic path aspect, we adopt the following definition.

\begin{definition}
    A path $\Phi\in\widetilde{\rm Sp}(2m)$ is said to be {\it dynamically convex} if $\mu_-(\Phi)\geq m+2$.
\end{definition}

The following lemma, as a consequence of the abstract precise iteration formula \eqref{eq:precise} we will introduce in Section \ref{subsec:ind iterate}, asserts that  the indices $\mu_-(\Phi^k)$ are increasing provided $\Phi$ is dynamically convex.
\begin{lemma}[{cf. \cite[Corollary 3.1]{LZ02} or \cite[Lemma 4.8]{GG20}}]\label{lem:ind mono}
For each $\Phi\in\widetilde{\rm Sp}(2m)$, we have 
\begin{equation*}
    \mu_-(\Phi^{k+1})\geq \mu_-(\Phi^k)+(\mu_-(\Phi)-m)
\end{equation*}
for all $k\in\mathbb{N}$. In particular,
\[
\mu_-(\Phi^k)\geq(\mu_-(\Phi)-m)k+m.
\]
Assume furthermore, that $\Phi$ is dynamically convex. Then the function $\mu_-(\Phi^k)$ of $k\in\mathbb{N}$ is strictly increasing,
\[
\mu_-(\Phi^k)\geq 2k+m,
\]
and $\hat\mu(\Phi)\geq\mu_-(\Phi)-m\geq 2$. Thus $\mu_-(\Phi^k)\geq m+2$ and all the iterations $\Phi^k$ are also dynamically convex.
\end{lemma}

\subsubsection{Maslov-type indices for closed orbits}\label{subsect:index for orbits}
Now we briefly recall how  to translate the definition of the Maslov-type indices introduced above for elements of $\widetilde{\rm Sp}(2m)$ to closed orbits of Hamiltonian flow or Reeb flow. 

We begin with the case of Hamiltonian flow. Consider a symplectic manifold $(W^{2n},\omega)$. Following \cite[Section 5]{SZ92}, let $x$ be a contractible 1-periodic orbit of a Hamiltonian $H: S^1 \times W^{2n} \to \mathbb{R}$. Choose a {\it capping disk} for $x$, i.e., a map $v: D^2 \to W^{2n}$ satisfying $v(e^{2\pi i t}) = x(t)$ for $t \in [0,1]$, where $D^2$ denotes the unit disk in $\mathbb{R}^2 = \mathbb{C}$. Fix a symplectic trivialization of the pullback bundle $v^*TW$:
\[
\psi_v: (v^*TW, \omega) \to D^2 \times (\mathbb{C}^n, \omega_{\mathrm{std}}).
\]
Linearizing the Hamiltonian flow along $x$ with respect to $\psi_v$ gives rise to a symplectic path
\[
\Phi_v(t) = \psi_v(x(t)) \circ \mathrm{d}\varphi_H^t(x(0)) \circ \psi_v^{-1}(x(0)), \quad t \in [0,1],
\]
with $\Phi_v(0) = \mathrm{Id}$, so $\Phi_v \in \widetilde{\mathrm{Sp}}(2n)$. For a fixed capping disk $v$, any two such trivializations yield the same symplectic path up to homotopy with fixed endpoints. We then define
\[
\mu_{\pm}(x; v) := \mu_{\pm}(\Phi_v) \quad \text{and} \quad \hat{\mu}(x; v) := \hat{\mu}(\Phi_v).
\]
If $x$ is non-degenerate, set $\mu(x; v) := \mu(\Phi_v)$. Moreover, assuming $c_1(TW) = 0 \in H^2(W; \mathbb{Z})$, any two capping disks for $x$ produce the same homotopy class $\Phi \in \widetilde{\mathrm{Sp}}(2n)$ (see \cite[Lemma 5.2]{SZ92}). Thus, in this case, we define
\[
\mu_{\pm}(x) := \mu_{\pm}(\Phi), \quad \hat{\mu}(x) := \hat{\mu}(\Phi),
\]
and $\mu(x) := \mu(\Phi)$ for non-degenerate $x$. We call $\Phi$ the {\it associated symplectic path} of $x$. We sometimes use notations $\mu_{\pm}(x,H)$, $\hat\mu(x,H)$, $\mu(x,H)$ or $\mu_{\pm}(x,\varphi_H)$,  $\hat\mu(x,\varphi_H)$, $\mu(x,\varphi_H)$ to emphasize the Hamiltonian $H$, depending on the context.

Now consider the case of Reeb flow. Let $x$ be a closed Reeb orbit on a contact manifold $(\Sigma^{2n-1}, \alpha)$ with contact structure $\xi = \ker \alpha$. Assume $x$ is contractible with action $\mathcal{A}(x) = T$. Rescale $\alpha$ to $\frac{1}{T}\alpha$ (still denoted by $\alpha$) so that $\mathcal{A}(x) = 1$. The  definitions for Hamiltonian flow then extend analogously to the Reeb flow. Specifically, take a capping disk $w: D^2 \to \Sigma^{2n-1}$ satisfying $w(e^{2\pi i t}) = x(t)$. Fix a symplectic trivialization $\psi_w$ of $w^*\xi$, linearize the Reeb flow $\phi_\alpha^t$ along $x$ and restrict to $\xi$, obtaining a path $\Phi_w \in \widetilde{\mathrm{Sp}}(2n-2)$:
\[
\Phi_w(t) = \psi_w(x(t)) \circ \mathrm{d}\phi_\alpha^t(x(0))|_{\xi} \circ \psi_w^{-1}(x(0)), \quad t \in [0,1].
\]
Define $\mu_{\pm}(x; w) := \mu_{\pm}(\Phi_w)$ and $\hat{\mu}(x; w) := \hat{\mu}(\Phi_w)$. When $x$ is non-degenerate, set $\mu(x; w) := \mu(\Phi_w)$. If $c_1(\xi) = 0$, any two capping disks yield identical $\Phi \in \widetilde{\mathrm{Sp}}(2n-2)$, so we define
\[
\mu_{\pm}(x) := \mu_{\pm}(\Phi), \quad \hat{\mu}(x) := \hat{\mu}(\Phi),
\]
and $\mu(x) := \mu(\Phi)$ for non-degenerate $x$. As in the case of Hamiltonian flow, $\Phi$ is called the {\it associated symplectic path} of $x$.

For iterated orbits, define $w^k: D^2 \to \Sigma^{2n-1}$ by $w^k(z) = w(z^k / |z|^{k-1})$ for $z \neq 0$ and $w^k(0) = w(0)$. This provides a capping disk for $x^k$. Any symplectic trivialization $\psi_w$ of $w^*\xi$ extends to $(w^k)^*\xi$ analogously, satisfying $\Phi_{w^k} = \Phi_w^k$. Consequently, if $c_1(\xi) = 0$,
\[
\mu_{\pm}(x^k) = \mu_{\pm}(\Phi^k), \quad \hat{\mu}(x^k) = \hat{\mu}(\Phi^k),
\]
and $\mu(x^k) = \mu(\Phi^k)$ for non-degenerate $x$.

These definitions apply directly to our setting where  $\Sigma^{2n-1}$ bounds a star-shaped domain in $\mathbb{R}^{2n\geq4}$, as this implies that $\Sigma$ is simply connected and $H^2(\Sigma;\mathbb{Z})=0$.

\begin{remark}\label{rmk:frac grad}
Although the above definition suffices for our purposes, we note that \cite{Mcl16} introduced an alternative definition of Maslov-type indices for closed Reeb orbits $x$ that eliminates the contractibility assumption while coinciding with the standard definition in the contractible case. This approach has been adopted in recent works such as \cite{AM22,CGG24}.

The construction proceeds as follows. Assuming $c_1(\xi) = 0$, the top complex wedge power $L_\xi := \bigwedge_{\mathbb{C}}^{n-1} \xi$ forms a trivial complex line bundle. Choose either a trivialization of this bundle or equivalently, a nonvanishing section $\mathfrak{s}$ of $L_\xi$. Let $\psi: x^*\xi \to S^1 \times \mathbb{C}^{n-1}$ be a Hermitian trivialization whose induced trivialization on $\bigwedge_{\mathbb{C}}^{n-1} x^*\xi$ matches $\mathfrak{s}$. This condition uniquely determines the homotopy class of $\psi$. Furthermore, the induced trivialization on $x^k$ coincides with the $k$-th iteration of $\psi$ up to homotopy. Using this trivialization, we convert the linearized Reeb flow along $x$ into a path $\Phi \in \widetilde{\mathrm{Sp}}(2n-2)$. When $H_1(\Sigma; \mathbb{Q}) = 0$, the homotopy class $\Phi \in \widetilde{\mathrm{Sp}}(2n-2)$ is independent of the section $\mathfrak{s}$, resulting in well-defined Maslov-type indices for $x$. Complete details can be found in \cite[Section 4]{Mcl16}.
\end{remark}

\subsection{Splitting number and abstract precise iteration formula}\label{subsec:ind iterate}
In \cite{Long99}, to develop the Bott-type formula of the Maslov-type indices,  Long generalized the definition of $\mu_{\pm}$ to each $\omega\in\mathbb{U}$ (where $\mu_{\pm}$ corresponds to the case $\omega=1$) and introduced the splitting number theory for the symplectic matrix, which induces  a precise description of  
$\mu_-(\Phi^k)$ for arbitrary $k\in\mathbb{N}$.

\subsubsection{Splitting number}

Consider the following real function  introduced in \cite{Long99}:
\[D_{\omega}(M) =(-1)^{n-1}\bar{\omega}^n \det(M-\omega \mathrm{Id}),
~\forall\omega \in\mathbb{U},~ M \in \mathrm{Sp}(2m).\]
Then for each $\omega\in \mathbb{U}$ the following codimension $1$ hypersurface in $\mathrm{Sp}(2m)$ is defined:\[\mathrm{Sp}(2m)^0_\omega = \{M \in \mathrm{Sp}(2m) \mid D_\omega(M) =
0\}.\] %For each $M \in \mathrm{Sp}(2m)^0_\omega$, we define a co-orientation
%of $\mathrm{Sp}(2m)^0_\omega$ at $M$ by the positive direction $\frac{d}{dt}M
%e^{t\epsilon J} |_{t=0}$ of the path $M e^{t\epsilon J}$ with $0\leq
%t \leq 1$ and $\epsilon > 0$ being sufficiently small.
Let
\begin{align*}
&\mathrm{Sp}(2m)_\omega^*:= \mathrm{Sp}(2m)\setminus\mathrm{Sp}(2m)^0_{\omega},\\
&\widetilde{\mathrm{Sp}}(2m)^*_{\omega}:=\{\Phi \in \widetilde{\mathrm{Sp}} (2m): \Phi(1)\in\mathrm{Sp}(2m)^*_{\omega}\}.%\\
%&\widetilde{\mathrm{Sp}}(2m)^0_{\omega}:=\widetilde{\mathrm{Sp}}(2m)\setminus \widetilde{\mathrm{Sp}}(2m)^*_{\omega}
\end{align*} 
Following \cite{Long99}, the $\omega$-{\it index} $i_{\omega}$ is defined as follows. Firstly, for $\Phi \in \widetilde{\mathrm{Sp}}(2m)^*_{\omega}$, $i_\omega(\Phi)$ is defined analogously as the Conley-Zehnder index introduced in the previous subsection via replacing $\mathrm{Sp}(2m)^*$ and $\widetilde{\rm Sp}(2m)^*$ by $\mathrm{Sp}(2m)^*_\omega$ and  $\widetilde{\rm Sp}(2m)^*_\omega$ respectively. Then, for an arbitrary $\Phi\in \widetilde{\rm Sp}(2m)$, define
\begin{equation*}
    i_{\omega}(\Phi):=\liminf_{\Psi\to\Phi}i_{\omega}(\Psi),
\end{equation*}
where  the limit is taken over $\Psi\in\widetilde{\rm Sp}(2m)^*_\omega$ converging to $\Phi$. Like \eqref{eq:+-nullity}, there also holds
\begin{eqnarray*}
    i_{\omega}(\Phi)+\nu_{\omega}(\Phi)=\limsup_{\Psi\to\Phi}i_{\omega}(\Psi),
\end{eqnarray*} 
where $\nu_{\omega}(\Phi) := \dim_{\mathbb{C}} \ker_{\mathbb{C}}(\Phi(1) - \omega \mathrm{Id})$, and is called the $\omega$-{\it nullity} of $\Phi$.
%\begin{equation}\label{eq:intersection number}
%i_{\omega}(\Phi) =[\mathrm{Sp}(2m)^0_{\omega}:  \gamma\ast \xi_n],
%\end{equation}
%where the right hand side of \eqref{eq:intersection number} is the usual homotopy intersection number,
%and the orientation of  $\gamma\ast \xi_n$ is its positive time direction under homotopy with fixed end points.

The splitting number is defined by the ``jump'' of $\omega$-indices.
\begin{definition}
For each $M \in \mathrm{Sp}(2m)$ and $\omega \in \mathbb{U}$, {\it the splitting
numbers} $S^{\pm}_M (\omega)$ of $M$ at $\omega$ are defined by
\begin{equation}\label{eq:split num}
S^{\pm}_M (\omega)=\lim_{\epsilon\rightarrow 0^+}{i_{\omega
\exp{(\pm\sqrt{-1}\epsilon)}}(\Phi)-i_{\omega}(\Phi)},
\end{equation}
for any  $\Phi\in\widetilde{\mathrm{Sp}} (2m)$ satisfying
$\Phi(1) = M$.
\end{definition}

\begin{remark}\label{rmk: nullity for orbit}
The  $\omega$-nullity and the splitting number can be defined for a periodic orbit of Hamiltonian flow or Reeb flow by its (restricted) linearized Poincar\'e return map. For instance, let $x$ be a closed orbit  of a Hamiltonian (or Reeb) flow and $\Phi\in\widetilde{\mathrm{Sp}}(2m)$ be its associated symplectic path defined in Subsection \ref{subsect:index for orbits}, then we define 
\begin{equation*}
  \nu_\omega(x):=\nu_\omega(\Phi(1)),~~ ~~~~S_x^\pm(\omega):=S_{\Phi(1)}^\pm(\omega).
\end{equation*}
And we always omit the subscript $\omega$ of $\nu_\omega(x)$ when $\omega=1$.
\end{remark}

\begin{lemma}[\cite{Long99}] \label{lem:prop  split}
For each $M\in \mathrm{Sp}(2m)$ and $\omega\in \mathbb{U}$, the
splitting numbers $S^{\pm}_M (\omega)$ are well defined, i.e., they are independent of the choice of the path
$\Phi \in \widetilde{\rm Sp}(2m)$ with $\Phi(1)=M$ appeared in \eqref{eq:split num}. And
 the following properties hold:\begin{itemize}

\item[\hypertarget{list:S1}{\rm (S1)}]~$S^{+}_M (\omega)= S^{-}_M (\bar{\omega})$;

%\item[(ii)] $S^{\pm}_M (\omega)= S^{\pm}_N (\omega)$ if $P^{-1}N\in
%\Omega^0(P^{-1}M)$.

\item[\hypertarget{list:S2}{\rm (S2)}]~$S^{\pm}_{M_1\diamond M_2}(\omega)= S^{\pm}_{M_1}(\omega)+ S^{\pm}_{M_2}(\omega)$;

\item[\hypertarget{list:S3}{\rm (S3)}]~$S^{\pm}_M (\omega)=0$ whenever $\omega\notin \sigma(M)$;

\item[\hypertarget{list:S4}{\rm (S4)}]~$0\leq S^{\pm}_M (\omega)\leq \nu_{\omega}(M)$, where $\nu_{\omega}(M) := \dim_{\mathbb{C}} \ker_{\mathbb{C}}(M - \omega \mathrm{Id})$.
\end{itemize}
\end{lemma}

\subsubsection{Abstract precise iteration formula}

For each $\Phi\in\widetilde{\mathrm{Sp}}(2m)$, the iteration behavior of indices $\mu_-(\Phi^k)$, $k\in\mathbb{N}$, is completely determined by $\mu_-(\Phi)$ and the splitting numbers of the end point $\Phi(1)$. Indeed, we have the following abstract precise iteration formula proved in \cite{LZ02}: 

\begin{proposition}[{\cite[Theorem 2.1]{LZ02}}]\label{prop:precise}
For each $\Phi\in\widetilde{\mathrm{Sp}}(2m)$ with $M=\Phi(1)$, and $k\in\mathbb{N}$, we have
\begin{equation}\label{eq:precise}
\begin{aligned}
\mu_-(\Phi^k)
&=k(\mu_-(\Phi)+S^{+}_M(1)-C(M))\\
&+2\sum_{\theta\in(0,2\pi)}\lceil\frac{k\theta}{2\pi}\rceil S^-_M(e^{\sqrt{-1}\theta})-(S_M^+(1)+C(M)),
\end{aligned}
\end{equation}
where $C(M):=\sum_{\theta\in(0,2\pi)}S_M^-(e^{\sqrt{-1}\theta})$.
\end{proposition}

As a consequence of Proposition \ref{prop:precise} and \eqref{eq:mean ind}, we have
\begin{corollary}[{\cite[Corollary 2.1]{LZ02}}]\label{coro:precise mean ind}
For each $\Phi\in\widetilde{\mathrm{Sp}}(2m)$ with $M=\Phi(1)$, we have
\begin{equation}\label{eq:precise mean ind}
\begin{aligned}
\hat\mu(\Phi)
&=\mu_-(\Phi)+S^{+}_M(1)-C(M)+\sum_{\theta\in(0,2\pi)}\frac{\theta}{\pi} S^-_M(e^{\sqrt{-1}\theta}).
\end{aligned}
\end{equation}
\end{corollary}

\subsubsection{Basic normal forms}

For an arbitrary $M\in \mathrm{Sp}(2m)$, the splitting numbers is subtle to compute. To overcome this difficulty,  Long  introduced 
in \cite{Long99} the {\it homotopy component}  $\Omega^0(M)$,  which is defined as the path connected component
containing $M$ of the set
\begin{equation*}
\begin{aligned}
  \Omega(M)=\{N\in{\rm Sp}(2m):\sigma(N)\cap\mathbb{U}=\sigma(M)\cap\mathbb {U}\;
{\rm and}~
\nu_{\omega}(N)=\nu_{\omega}(M),\;\forall
\omega\in\sigma(M)\cap \mathbb{U}\}, 
\end{aligned}
\end{equation*}
and the following {\it basic normal forms} :
\begin{eqnarray}
D(\lambda):=\begin{pmatrix}\lambda & 0 \\
         0  & \lambda^{-1}\end{pmatrix}, &\quad& \lambda=\pm 2,\nonumber\\
N_1(\lambda,b): = \begin{pmatrix}
    \lambda & a\\
         0  & \lambda
\end{pmatrix}, &\quad& \lambda=\pm 1, a=\pm1, 0, \nonumber\\
R(\theta):=\begin{pmatrix}
    \cos\theta & -\sin\theta\\
        \sin\theta  & \cos\theta
\end{pmatrix}, &\quad& \theta\in (0,\pi)\cup(\pi,2\pi),\nonumber\\
N_2(\omega,b):= \begin{pmatrix}R(\theta) & b\\
              0 & R(\theta)
              \end{pmatrix}, &\quad& \omega=e^{\sqrt{-1}\theta},~\theta\in (0,\pi)\cup(\pi,2\pi),\label{eq:4x4}
                    \end{eqnarray}
where $b=\begin{pmatrix}
    b_1 & b_2\\
    b_3 & b_4
\end{pmatrix}$ with  $b_i\in\mathbb{R}$ and  $b_2\not=b_3$. In addition, $N_2(\omega,b)$ is {\it non-trivial} if $(b_2-b_3)\sin\theta<0$, and  {\it trivial} if  $(b_2-b_3)\sin\theta>0$. 

With these preparations, the computation for the splitting numbers  can be  reduced to the cases of  basic normal forms. Indeed, we have the following lemma, the proof of the assertions can be found in \cite{Long99}.  
\begin{lemma}\label{lem:splt normal}
Let $\omega \in \mathbb{U}$ and $M\in \mathrm{Sp}(2m)$.  The following statements hold:
\begin{itemize}
    \item[(i)] there exists $N^*\in\Omega^0(M)$ such that $N^*$ is a symplectic direct sum (see Subsection \ref{subsec:nontation}) of some basic normal forms.
    \item[(ii)] the splitting numbers $S^{\pm}_N(\omega)$ are constant for all $N \in \Omega^0(M)$;
    \item[(iii)] the splitting numbers of the basic normal forms is given by :
    \begin{eqnarray*}
(S_{N_1(1,a)}^+(1),S_{N_1(1,a)}^-(1)) &=& \left\{\begin{aligned}(1,1), &\quad {\rm if}\;\; a\geq 0, \\
(0,0), &\quad {\rm if}\;\; a< 0. \end{aligned}\right. \\
(S_{N_1(-1,a)}^+(-1), S_{N_1(-1,a)}^-(-1))&=& \left\{\begin{aligned}(1,1), &\quad {\rm if}\;\; a\leq 0, \\
(0,0), &\quad {\rm if}\;\; a> 0. \end{aligned}\right. \\
(S_{R(\theta)}^+(e^{\sqrt{-1}\theta}), S_{R(\theta)}^-(e^{\sqrt{-1}\theta}))&=&(0,1),~ \theta \in (0, \pi) \cup (\pi, 2\pi).\\
(S_{N_2(\omega,b)}^+(\omega),S_{N_2(\omega,b)}^-(\omega)) &=& \left\{\begin{aligned}(1,1), &\quad \text{\rm if $N_2(\omega,b)$ is non-trivial}, \\
(0,0), &\quad\text{\rm if $N_2(\omega,b)$ is trivial}. \end{aligned}\right. \\
\end{eqnarray*}
\end{itemize}
\end{lemma}

\subsection{Common index jump theorem}\label{subsec:CIJT}
We state the following Common Index Jump Theorem, originally from \cite[Theorem 4.3]{LZ02}, with an enhanced version in \cite[Theorem 3.5]{DLW16}. To maintain notational consistency throughout this article, our formulation here  follows \cite[Theorem 5.2]{GG20}.
 \begin{theorem}\label{thm:CIJT}  
 Let $\Phi_1,\ldots, \Phi_r$ be a finite collection of
elements in $\widetilde{\rm Sp}(2m)$  satisfying $\hat\mu(\Phi_i)>0$ for $1\le i\le r$. Then for any
$\eta>0$, %and any $l_0\in\mathbb{N}$,
there exists an integer sequence $d_j\rightarrow\infty$ and $r$ integer sequences
$k_{ij}\rightarrow \infty$ %for $1\le i\le r$ 
such that for $1\le i\le r$, $j\in\mathbb{N}$, %and $1\le l\le l_0$,
we have
\begin{align} &|\hat\mu(\Phi_i^{k_{ij}})-d_j|<\eta,\label{eq:IR1}\\
&\mu_\pm(\Phi_i^{k_{ij}+1})=d_j+\mu_\pm(\Phi_i),\label{eq:IR2}\\
&\mu_+(\Phi_i^{k_{ij}-1})=d_j-(\mu_-(\Phi_i)+2S^+_{\Phi_i(1)}(1)-\nu(\Phi_i)),\label{eq:IR3}
\end{align}
%where \[Q_{ij}(l):=\sum_{\frac{k_{ij}\theta}{\pi}\in\mathbb{Z}} S_{\Phi_i(1)}^-(e^{\sqrt{-1}\theta}).\]
Moreover, for any $N\in\mathbb{N}$, we can make all $d_j$ and $k_{ij}$ divisible by $N$.

\end{theorem}

In what follows, we will  refer to the collection $\{k_{ij},d_{j}\}$ for a fixed $j$ as a {\it common index jump event}. 

%\begin{remark}
    %The right side of \eqref{eq:IR3} can also be expressed by the splitting numbers and nullity. Indeed we can rewrite \eqref{eq:IR3} as
    %\begin{equation*}
        %\mu_+(\Phi_i^{k_{ij}-l})=d_j-\mu_-(\Phi_i^{l})-2S_{\Phi_i^l(1)}^+(1)+\nu(\Phi^l_i),
    %\end{equation*}
%\end{remark}

Under the assumption of dynamically convexity, we have the following proposition as a consequence of Lemma \ref{lem:ind mono} and Theorem \ref{thm:CIJT}.

\begin{proposition}\label{prop:CIJT}
 Suppose the paths $\Phi_1,\ldots,\Phi_r$ in Theorem \ref{thm:CIJT}
are dynamically convex. Then for  all $l\in\mathbb{N}$, we have
\begin{align} &\mu_-(\Phi_i^{k_{ij}+l})\ge d_j+2+m,\label{eq:CIJT low}\\
&\mu_+(\Phi_i^{k_{ij}-l})\le d_j-2.\label{eq:CIJT up}
\end{align}
\end{proposition} 

\begin{proof}
By the index increasing property stated in Lemma \ref{lem:ind mono}, it is sufficient to consider the case $l=1$. Then we conclude
\[\mu_-(\Phi_i^{k_{ij}+1})\ge d_j+2+m,\]
from \eqref{eq:IR2} and the dynamically convexity, and \eqref{eq:CIJT low} follows. 

 Due to \eqref{eq:IR3}, to prove \eqref{eq:CIJT up}, it remains to show that

\begin{equation}\label{eq:geq 2}
    \mu_-(\Phi_i)+2S^+_{\Phi_i(1)}(1)-\nu(\Phi_i)\geq2.
\end{equation}
 Fix $i\in\{1.\dots,r\}$. According to (i) of Lemma \ref{lem:splt normal}, we  can find $M'\in \Omega^0(\Phi_i(1))$ such that
\[M'=N(1,1)^{\diamond p_+}\diamond\mathrm{Id}_2^{\diamond p_0}\diamond N(1,-1)^{\diamond p_-}\diamond G_,\]
where $p_\pm$, $p_0$ are nonnegative integers with $p_-+p_0+p_{+}\leq m$ and $G$ is a symplectic matrix with $1\notin \sigma(G)$. Then we have 
\[\nu(\Phi_i)=p_++2p_0+p_-~~~~ \text{and}~~~~ S^+_{\Phi_i(1)}(1)=p_++p_0.\]
Consequently,
\begin{equation}\label{eq:geq -m}
   2S^+_{\Phi_i(1)}(1)-\nu(\Phi_i)=p_+-p_- \geq- p_-\geq-m.
\end{equation}
Then \eqref{eq:geq 2} follows from \eqref{eq:geq -m} and the dynamically convexity.
\end{proof}

\section{Proof of Theorem \ref{thm:local max}}\label{sec:local max}

\subsection{Analogue for local Floer homology}
Our strategy to prove Theorem \ref{thm:local max} involves translating the equivariant local symplectic homology of a closed Reeb orbit to the local Floer homology via its Poincar\'{e} return map and then establishing an analogous argument. Based on the discussion of the local model in Subsection \ref{subsect:local model}, it suffices to consider the local Floer homology in the case where $W^{2n}=\mathbb{R}^{2n}$ and $x(t)\equiv p=0$ is a constant one-periodic orbit of a Hamiltonian $H$ defined in a neighborhood of $p$.

\begin{theorem}\label{thm:vanish}
Let $p$ be an isolated fixed point  of the Hamiltonian diffeomorphism $\varphi=\varphi_H^1$. Assume that
\begin{itemize}
    \item  $\mathrm{HF}_{\mu_+}(H,p)\neq0$, where $\mu_+$ is abberiviation of $\mu_+(p,H)$;
     \item $\mu_{+}(p,H)=d+n,$ where $d$ is some integer satisfying $|d-\hat{\mu}(p,H)|<\frac{1}{2}$;   
\end{itemize}
then we have 
\begin{equation*}
    \mathrm{HF}_*(H,p)=0, \ \ \text{whenever $*\neq \mu_{+}(p,H)$.}
\end{equation*}
\end{theorem}

\subsubsection{Particular case 1: \texorpdfstring{$p$}{Lg} is non-degenerate}
In this case, the assertion holds obviously. Indeed, since  $p$ is non-degenerate, we have $\mu_+(p,H)=\mu(p,H)$ and 
\begin{equation*}
    \mathrm{HF}_*(H,p)=\left\{ \begin{aligned}
     \mathbb{F}, & \ \ \ \ *=\mu(p,H),\\
     0,          & \ \ \ \ \text{otherwise,}
    \end{aligned}
    \right .
\end{equation*}
by Example \ref{exam:non degen HF}.
\subsubsection{Particular case 2: \texorpdfstring{$p$}{Lg} is totally degenerate} 
In this case, all the eigenvalues of $\mathrm{d}\varphi_p$ are equal to one, then $\hat{\mu}(p,H)$ is an integer by \hyperlink{list:MI3}{(MI3)} of Section \ref{subsec:index}.  Hence we have 
\begin{equation}\label{eq:mean d}
\hat{\mu}(p,H)=d
\end{equation}
by the second hypothesis.
Let $\Phi= \mathrm{d}\varphi_p$ and $V=T_pM\simeq \mathbb{R}^{2n}$. By \cite[Lemma 5.5]{Gi10}, we can find a decomposition of $V=L\oplus L'$, where both $L$ and $L'$ are Lagrangian subspaces of $V$ and $\Phi(L)=L$. Moreover, we can choose a linear canonical coordinate system $(x,y)$ on $V$, which is compatible with the decomposition, i.e., such that the $x$-coordinates span $L$ and the $y$-coordinates span $L'$. \cite[Lemma 5.5]{Gi10} asserts that we can make $\| \Phi-{\rm Id} \|$ arbitrarily small by a suitable choice of such coordinate system. As a consequence, the generating function $F$ of $\varphi$ can be defined in a neighborhood $U$ of $p$ as in Section \ref{subsec:GF}. Set $z=(x,y)\in\mathbb{R}^{2n}$, then $F$ is determined by the equation
\begin{equation}\label{eq:F def}
    \varphi(z)-z=X_F(\psi(z)),
\end{equation}
where $X_F$ is the Hamiltonian vector field of $F$, and $\psi:U\to U$ is a diffeomorphism defined by 
\begin{equation*}
\psi(z)=(\text{$x$-component of $\varphi(z)$},y).
\end{equation*}

As in \cite[Section 6]{Gi10}, starting with $F$, one can construct near $p$ a 1-periodic Hamiltonian $K_t$ with time-1 map $\varphi$, such that the pair $(K,F)$ satisfies the hypotheses of Lemma \ref{lem:KF}. Then by Remark \ref{rmk:local maximal} we can conclude that
\begin{equation}\label{eq:KF}
\left\{
  \begin{aligned}
      &\mathrm{HF}_*(K,p)=\mathrm{HM}_{*+n}(F, p),\\
      &\mathrm{HF}_n(K,p)\neq0 \Rightarrow \mathrm{HF}_*(K,p)=0, \forall*\neq n.
  \end{aligned}
\right.    
\end{equation}

On the other hand, consider the composition $H\# \overline{K}$ introduced in Section \ref{subsec:local Floer}, whose Hamiltonian flow $\varphi_H^t\circ(\varphi_K^t)^{-1}$ fix $p$. Since both the Hamiltonians $K$ and $H$ generate the same time-1 map $\varphi$,  such flow is a loop of Hamiltonian diffeomorphisms starting at $\mathrm{id}$ in a neighborhood of $p$, fixing $p$. Note that $H=(H\#\overline{K})\#K$. According to \hyperlink{list:LF3}{(\text{LF3})} of Section \ref{subsec:local Floer},  we have
\begin{equation}\label{eq:HK}
    \mathrm{HF}_{*+m}(H,p)=\mathrm{HF}_*(K,p),~~~m=2\mu_{\rm maslov}(\psi), 
\end{equation}
where $\psi$ is a loop in $\mathrm{Sp}(2n)$ given by $\psi(t):=(\mathrm{d}\varphi_H^t)_p\circ(\mathrm{d}\varphi_K^t)_p^{-1}$. By Remark \ref{rmk:multiplicate to concatenation}, $(\mathrm{d}\varphi_H^t)_p=\psi(t)\circ(\mathrm{d}\varphi_K^t)_p$ can be viewed as concatenation of $(\mathrm{d}\varphi_K^t)_p$ and $\psi\circ(\mathrm{d}\varphi_K^1)_p$ up to a homotopy with fixed end. Then by \eqref{eq:differ by maslov}, we have
\begin{equation}\label{eq: shift m}
\begin{aligned}
  m=2\mu_{\rm maslov}(\psi)&=2\mu_{\rm maslov}(\psi\circ(\mathrm{d}\varphi_K^1)_p)\\
  &=\hat{\mu}(\psi(t)\circ(\mathrm{d}\varphi_K^t)_p)-\hat\mu((\mathrm{d}\varphi_K^t)_p)\\
  &=\hat{\mu}(p,H)-\hat\mu(p,K),
\end{aligned}
\end{equation}
where the second equality holds by the homotopy invariance of the Maslov index and the fact $\mathrm{Sp}(2n)$ is connected, see \cite[Theorem 2.2.12]{MS17}.

Assume we have verified that $\mu_+(p,H)=n+\hat{\mu}(p,H)-\hat\mu(p,K)$, then by \eqref{eq:HK} and \eqref{eq: shift m}, $ \mathrm{HF}_{\mu+}(H,p)\neq0$ implies that $ \mathrm{HF}_n(K,p)\neq0$. Thus, Theorem \ref{thm:vanish} is established for the totally degenerate case via \eqref{eq:KF}.

To complete the proof, we verify $\mu_+(p,H)=n+\hat{\mu}(p,H)-\hat\mu(p,K)$. In view of \eqref{eq:mean d}, this reduces to showing $\hat{\mu}(p,K)=0$.

Now we recall the construction of the  Hamiltonian $K$. (For more details, see \cite[Section 6]{Gi10}.)
Let $\widetilde{K}$ be a Hamiltonian defined in some neighborhood of $p$, satisfying 
\begin{equation}\label{eq:K~flow}
 \varphi_{\widetilde{K}}^t(z)-z=tX_{F}(\psi^t(z)),   
\end{equation}
where  \begin{equation*}
\psi^t(z)=(\text{$x$-component of $\varphi_{\widetilde{K}}^t(z)$},y).
\end{equation*}
Note that $\varphi_{\widetilde{K}}^1=\varphi$  and $\psi^1=\psi$.
 
Consider a monotone increasing smooth function  $\lambda:[0,1]\to[0,1]$ with $\lambda(t)\equiv0$ when $t$ is near 0 and $\lambda(t)\equiv1$ when $t$ is near 1. The Hamiltonian $K$ is given by
\begin{equation}\label{eq:K}
 K_t(z)=(1-\lambda'(t))F(z)+\lambda'(t)\widetilde{K}_{\lambda(t)}(\varphi_F^{\lambda(t)-t}(z)),
\end{equation}
where  $\varphi_F$ denotes the Hamiltonian flow of $F$. Consequently,
\begin{equation}\label{eq:K flow}
    \varphi_K^t=\varphi_F^{t-\lambda(t)}\varphi_{\widetilde{K}}^{\lambda(t)}.
\end{equation}
It is clearly that $\varphi_K^1=\varphi$.

Let $\tilde\Phi_t, \Phi_t\in \widetilde{\rm Sp}(2n)$ be the linearizations of $\varphi_{\widetilde{K}}^t$ and $\varphi_K^t$ at $p$ respectively. Denote by $Q$ the Hessian of $F$ at $p$ and let $X_Q$ be the linear Hamiltonian vector field of $Q$ on $V$. Linearizing \eqref{eq:F def}, \eqref{eq:K~flow} and \eqref{eq:K flow}, we see that
\begin{align}
    &\Phi-\mathrm{Id}=X_QP(\Phi),\label{eq:Phi}\\
    &\tilde\Phi_t=\mathrm{Id}+tX_QP(\tilde{\Phi}_t),\label{eq:Phi~}\\
    &\Phi_t=\exp((t-\lambda(t))X_Q)\tilde{\Phi}_{\lambda(t)},\label{eq:Phi t}
\end{align}
where $P(\Phi):V\to V $ is defined by 
\begin{equation*}
    \Phi:(x,y)\mapsto(\bar x,\bar y),~~ P(\Phi)(x,y):=(\bar x,y),
\end{equation*}
with respect to the decomposition $V=L\oplus L'$. ($P(\tilde\Phi_t)$ is defined analogously.)

%Let $\sigma(\cdot)$ denote the set of all  eigenvalues of a linear transformation $\cdot$. 

\quad

\hypertarget{Claim 1}{\textbf{Claim 1:}} $\sigma(X_Q)=\{0\}$.

\begin{proof}[Proof of Claim 1] Representing \eqref{eq:Phi} as a block matrix relative to the decomposition  $V=L\oplus L'$, and noting that $\Phi(L)=L$, we have $\Phi=\begin{pmatrix}
    A & B\\
    0 & D
\end{pmatrix}$. Then \eqref{eq:Phi} is equivalent to
\begin{equation*}
\begin{pmatrix}
    A-\mathrm{Id} & B\\
    0       & D-\mathrm{Id}
\end{pmatrix}
=X_Q\begin{pmatrix}
    A & B\\
    0 & \mathrm{Id}
\end{pmatrix}.
\end{equation*}
This yields
\begin{equation}\label{eq:XQ}
    X_Q=\begin{pmatrix}
        \mathrm{Id}-A^{-1} & A^{-1}B\\
        0 & D-\mathrm{Id}
    \end{pmatrix},    
\end{equation}
where the invertibility of $A$ is guaranteed by the fact $\sigma(\Phi)=\{1\}$.  Therefore, 
\[\sigma(X_Q)=\sigma(\mathrm{Id}-A^{-1})\cup\sigma(D-\mathrm{Id})=\{0\},
\]
since $\sigma(\Phi)=\{1\}$.
\end{proof}

\hypertarget{Claim 2}{\textbf{Claim 2:}} $\hat{\mu}(\tilde{\Phi}_t)=0$.

\begin{proof}[Proof of Claim 2]
Expressing \eqref{eq:Phi~} in  block matrix form relative to $V=L\oplus L'$, let $\tilde\Phi_t=\begin{pmatrix}
    A(t) & B(t)\\
    C(t) & D(t)
\end{pmatrix}$. Then by \eqref{eq:XQ} and a direct computation, \eqref{eq:Phi~} becomes
\begin{align}
    \begin{pmatrix}
        A(t) & B(t)\\
        C(t) & D(t)
    \end{pmatrix}
    &=\begin{pmatrix}
        \mathrm{Id}& 0\\
        0 & \mathrm{Id}
    \end{pmatrix}+ t \begin{pmatrix}
        \mathrm{Id}-A^{-1} & A^{-1}B\\
        0 & D-\mathrm{Id}
    \end{pmatrix}
    \begin{pmatrix}
    A(t) & B(t)\\
    0 & \mathrm{Id}
\end{pmatrix}.\nonumber\\
&=\begin{pmatrix}
        \mathrm{Id}+t(\mathrm{Id}-A^{-1})A(t) & t(B(t)-A^{-1}B(t)+A^{-1}B)\\
        0 & tD+(1-t)\mathrm{Id}
    \end{pmatrix}.
\end{align}
This implies that
\begin{equation}\label{eq:ADC}
    \left\{\begin{aligned}
    &C(t)\equiv0,\\
    &D(t)=tD+(1-t)\mathrm{Id},\\
    &A(t)=[(1-t)\mathrm{Id}+tA^{-1}]^{-1}.
    \end{aligned}
    \right.
\end{equation}
Consequently,
\begin{equation}\label{eq:eigenvalue=1}
    \sigma(\tilde{\Phi}_t)=\sigma(A(t))\cup\sigma(D(t))\equiv\{1\},~~ \forall t\in[0,1],
\end{equation}
where the second equality uses \eqref{eq:ADC} and the fact $\sigma(D)=\{1\}=\sigma(A^{-1})$. It follows immediately that $\hat\mu(\tilde\Phi_t)=0$ from \eqref{eq:eigenvalue=1} and \hyperlink{list:MI6}{\text{(MI6)}} of Section \ref{subsec:index}.
\end{proof}

Combining \hyperlink{Claim 1}{Claim 1}, \hyperlink{Claim 2}{Claim 2}, \eqref{eq:homo loop} and \eqref{eq:Phi t}, we conclude $\hat\mu(p,K)=0$ and therefore Theorem \ref{thm:vanish} holds for the totally degenerate case.

\begin{remark}\label{rmk:attempt} Conjecturally, Theorem \ref{thm:vanish} holds without the condition $\mu_+(p,H)=d+n$.  Following the idea in this section, it sufficiently to prove this conjecture for totally degenerate case. By our discussion above, this reduces to find the pair $(K,F)$ in Lemma \ref{lem:KF} such that \[\hat\mu(p,K)=n+\hat{\mu}(p,H)-\mu_+(p,H),\]
which is equivalent to
\[\hat\mu(p,K)=n+S_{\Phi}^+(1)-\nu_1(\Phi), ~\text{where}~ \Phi=\mathrm{d}\varphi_p,\]
due to Corollary \ref{coro:precise mean ind}.
\end{remark}

\subsubsection{Particular case 3: split diffeomorphisms}\label{subsec:split}
Assume that $V=\mathbb{R}^{2n}$ is decomposed as a product of two symplectic vector spaces $V_0$ and $V_1$, $p=(p_0,p_1)$ in this decomposition, and the Hamiltonian $H$ is split, i,e., $H=H_0+H_1$ where $H_0$ and $H_1$ are Hamiltonians on $V_0$ and $V_1$ respectively, whose flows fix $p_0$ and $p_1$, correspondingly. Futhermore, assume that the time-1 map $\varphi_{H_0}$ of $H_0$ is non-degenerate at $p_0$ and the time-1 map $\varphi_{H_1}$ of $H_1$ is totally degenerate at $p_1$. 
 By \hyperlink{list:LF2}{\text{(LF2)}} of Section \ref{subsec:local Floer} we have
 \begin{equation}\label{eq:tensor}
     \mathrm{HF}_*(H,p)=\bigoplus_{*'+*''=*}\mathrm{HF}_{*'}(H_0,p_0)\otimes \mathrm{HF}_{*''}(H_1,p_1).
 \end{equation}
Since $\varphi_{H_0}$ is nondegenerate at $p_0$, From our particular case 1  we obtain
\begin{equation}\label{eq:H0p0}
    \mathrm{HF}_*(H_0,p_0)=\left\{ \begin{aligned}
     \mathbb{F}, & \ \ \ \ *=\mu(p_0,H_0),\\
     0,          & \ \ \ \ \text{otherwise.}
    \end{aligned}
    \right .
\end{equation} 
Thus $\mathrm{HF}_{\mu_+(p,H)}(H,p)\neq0$ implies that $\mathrm{HF}_{\mu_+(p_1,H_1)}(H_1,p_1)\neq0$, where \[\mu_+(p,H)=\mu(p_0,H_0)+\mu_+(p_1,H_1).\]  

To apply our particular case 2 to $\mathrm{HF}_*(H_1,p_1)$, we need to verify that 
\begin{equation}\label{eq:1/2dim V1}
    \mu_+(p_1,H_1)-\hat{\mu}(p_1,H_1)=\frac{1}{2}\dim V_1.
\end{equation}
By  \eqref{eq:bounded by mean}, 
\begin{align}
    \mu(p_0,H_0)-\hat{\mu}(p_0,H_0)&\leq\frac{1}{2}\dim V_0,\label{eq:<1/2dim V0}\\
    \mu_+(p_1,H_1)-\hat{\mu}(p_1,H_1)&\leq\frac{1}{2} \dim V_1.\nonumber
\end{align}
Observe that both $\mu_+(p_1,H_1)$ and $\hat{\mu}(p_1,H_1)$ are integers. Assuming for contradiction that
\begin{equation}\label{eq:1/2dim V1 -1}
    \mu_+(p_1,H_1)-\hat{\mu}(p_1,H_1)\leq\frac{1}{2} \dim V_1-1.
\end{equation}
 and combining \eqref{eq:1/2dim V1 -1}  with \eqref{eq:<1/2dim V0}, we see that
\begin{align*}
\mu_+(p,H)-\hat\mu(p,H)
&=\mu(p_0,H_0)+\mu_+(p_1,H_1)-(\hat{\mu}(p_0,H_0)+\hat{\mu}(p_1,H_1))\\
&\leq\frac{1}{2}(\dim V_0+\dim V_1)-1\\
&=n-1.
\end{align*}
which contradicts  the hypothesis $\mu_+(p,H)-d=n$ since $|d-\hat\mu(p,H)|<\frac{1}{2}$. This establishes \eqref{eq:1/2dim V1}.

According to our particular case 2,
\begin{equation}\label{eq:H1p1}
    \mathrm{HF}_*(H_1,p_1)=\left\{ \begin{aligned}
     \mathbb{F}, & \ \ \ \ *=\mu_+(p_1,H_1),\\
     0,          & \ \ \ \ \text{otherwise.}
    \end{aligned}
    \right .
\end{equation} 
Therefore Theorem \ref{thm:vanish} holds for split Hamiltonian $H$ by \eqref{eq:tensor}, \eqref{eq:H0p0} and \eqref{eq:H1p1}.

More generally, assume that the time-1 map $\varphi=\varphi_H$, but not necessarily $H$, is split, i.e., $\varphi=(\varphi_0,\varphi_1)$, where $\varphi_0$ and $\varphi_1$ are the germs of symplectomorphisms fixing $p_0$ in $V_0$ and, respectively, fixing $p_1$ in $V_1$. In fact both $\varphi_0$ and $\varphi_1$ are Hamiltonian, see \cite[Remark 4.1]{GG10}. Denote by $H_0$ and $H_1$ some Hamiltonians generating $\varphi_0$ and $\varphi_1$ respectively. As above, in addition we assume that the time-1 map $\varphi_0$ of $H_0$ is non-degenerate at $p_0$ and the time-1 map $\varphi_1$ of $H_1$ is totally degenerate at $p_1$.

Note that we do not necessarily have $H=H_0+H_1$, but both $H$ and $H_0+H_1$ generate the same time-1 map $\varphi$. Therefore, similarly with \eqref{eq:HK} and \eqref{eq: shift m}, by \hyperlink{list:LF3}{\text{(LF3)}} of Section \ref{subsec:local Floer} and \eqref{eq:differ by maslov}, we have
\begin{equation}\label{eq:H0+H1}
\left\{
\begin{aligned}
    &\mathrm{HF}_{*+m}(H,p)=\mathrm{HF}_{*}(H_0+H_1,p), \\ 
    &m=\hat\mu(p,H)-\hat\mu(p,H_0+H_1),\\
    &\mu_+(p,H)=\mu_{+}(p,H_0+H_1)+m,
\end{aligned}
\right.
\end{equation}
where $m$ is an even integer.
Consequently, $\mathrm{HF}_{\mu_+(p,H)}(H,p)\neq0$ implies
\begin{equation}\label{eq:HF H0+H1}
\mathrm{HF}_{\mu_+}(H_0+H_1,p)\neq 0,~\text{where $\mu_+=\mu_+(p,H_0+H_1)$.}  
\end{equation}
Moreover, $\mu_+(p,H)-d=n$ and $|d-\hat\mu(p,H)|<\frac{1}{2}$ combined with \eqref{eq:H0+H1}  yields
\begin{equation}\label{eq:ind H0+H1}
\mu_+(p,H_0+H_1)-(d-m)=n,~|d-m-\hat\mu(p,H_0+H_1)|<\frac{1}{2}.     
\end{equation}
Note that \eqref{eq:HF H0+H1} and \eqref{eq:ind H0+H1} precisely match the hypotheses in Theorem \ref{thm:vanish} for the split Hamiltonian $H_0+H_1$. From our previous discussion, we conclude that \[\mathrm{HF}_*(H_0+H_1,p)=0~~ \text{whenever $*\neq \mu_+(p,H_0+H_1)$}.\] Finally, by \eqref{eq:H0+H1}, we deduce that Theorem \ref{thm:vanish} remains valid for $H$ even when only its time-1 map is assumed to be split.

%\begin{remark}\label{rmk:m=0}
 %In \eqref{eq:H0+H1} we may assume  $m=0$  and $\hat\mu(p,H_0+H_1)=\hat\mu(p_0,H_0)$  through appropriate modification of $H_0$ and $H_1$. More precisely, let $G_1$ be a  Hamiltonian defined near $p_1$ in $V_1$, generating a loop of Hamiltonian diffeomorphisms $\varphi_{G_1}^t$ fixing $p_1$, with  \[2\mu_{\rm maslov}(\mathrm{d}(\varphi_{G_1}^t)_{p_1})=-\hat\mu(p_1,H_1).\] ($\hat\mu(p_1,H_1)$ is even by \hyperlink{list:MI3}{(MI3)}.) Let $G_0$ be a  Hamiltonian defined near $p_0$, generating a loop of Hamiltonian diffeomorphisms $\varphi_{G_0}^t$ fixing $p_0$, with  \[2\mu_{\rm maslov}(\mathrm{d}(\varphi_{G_0}^t)_{p_0})=\hat\mu(p_1,H_1)+m.\] Such Hamiltonians $G_0$ and $G_1$ can be explicitly constructed since $p_0$ and $p_1$ are origins in standard symplectic vector space.
 %By \hyperlink{list:LF3}{(LF3)} we have the isomorhphism
%\begin{equation*}
   % \mathrm{HF}_{*+m}(G_0\#H_0+G_1\#H_1,p)=\mathrm{HF}_*(H_0+H_1,p).
%\end{equation*}
%The desired result follows by replacing $H_0$  with $G_0\#H_0$ and $H_1$ with $G_1\#H_1$.
%\end{remark}

\subsubsection{The general case}\label{subsect:general case}
Let $\varphi=\varphi_H$ be the Hamiltonian diffeomorphsim fixing $p$ in $\mathbb{R}^{2n}$ and $\mathrm{d}\varphi_p$ denote the linearization of $\varphi$ at $p$. Choose a symplectic decomposition of $T_p\mathbb{R}^{2n}\simeq\mathbb{R}^{2n}=V_0\times V_1$ such that all eigenvalues of $\mathrm{d}\varphi_p|_{V_0}$ are different from 1 and all eigenvalues of $\mathrm{d}\varphi_p|_{V_1}$ are equal to 1. As in \cite[Page 343]{GG10}, one can construct a homotopy of Hamiltonian diffeomorphisms $\varphi_s$, $s\in[0,1]$ such that
\begin{itemize}
    \item $\varphi_0=\varphi$, and $\varphi_1$ is split with respect to the decomposition $\mathbb{R}^{2n}=V_0\times V_1$ as in Subsection \ref{subsec:split};
    \item $p$ is a uniformly isolated fixed point of $\varphi_s$, and $(\mathrm{d}\varphi_s)_p\equiv \mathrm{d}\varphi_p$, $\forall s\in[0,1]$.
\end{itemize}

Let $K_s$ be the Hamiltonian generating $\varphi_s$ as its time-1 map, obtained by concatenating the flow $\varphi_H^t$, $t\in[0,1]$ with the homotopy $\varphi_\zeta$, $\zeta\in[0,s]$, up to an obvious reparametrization. Then $K_0=H$ and $p$ is a uniformly isolated fixed point of $\varphi_{K_s}$ for all $s\in[0,1]$. Thus, by \hyperlink{list:LF1}{\text{(LF1)}} of Section \ref{subsec:local Floer}, we have
\begin{equation}\label{eq:H K1}
    \mathrm{HF}_*(H,p)=\mathrm{HF}_*(K_1,p).
\end{equation}
Moreover, since $(\mathrm{d}\varphi_s)_p$ is constant,  the homotopy invariance of the Maslov-type index yields
\begin{equation}\label{eq:ind H K1}
    \hat{\mu}(p,H)=\hat\mu(p,K_1),~~ \mu_+(p,H)=\mu_+(p,K_1).
\end{equation}
As $\varphi_1=\varphi_{K_1}$ is split, Theorem \ref{thm:vanish} applies to $K_1$ by Subsection \ref{subsec:split}. Through \eqref{eq:H K1} and \eqref{eq:ind H K1}, we conclude that Theorem \ref{thm:vanish} holds for $H$ as well.

  This concludes the proof of Theorem \ref{thm:vanish}.
\qed
\subsection{From local Floer homology to equivariant local symplectic homology}

\begin{proof}[Proof of Theorem \ref{thm:local max}]

 Assume $x=y^l$ is the $l$-th iteration of a prime closed Reeb orbit $y$. By \eqref{eq:SH split} and \eqref{eq:SH leq HF}, the condition $\mathrm{CH}_{\mu_+(x)}(x;\mathbb{Q})\neq0$ implies:
\begin{equation}\label{eq:SH mu+1}
\begin{aligned}
  1 &\leq \dim\mathrm{SH}_{\mu_+(x)+1}(y^l;\mathbb{Q}) \\
  &\leq \dim\mathrm{HF}_{\mu_+(x)+1}(\varphi_H^l,0;\mathbb{Q})+ \dim\mathrm{HF}_{\mu_+(x)}(\varphi_H^l,0;\mathbb{Q}),
  \end{aligned}
\end{equation}
where $\varphi_H$ is the Poincaré return map of $y$ defined on $B^{2n-2}$. 
From Remark \ref{rmk:same ind}, we have:
\[
\mu_+(0,\varphi_H^l) = \mu_+(x) \quad \text{and} \quad \hat\mu(0,\varphi_H^l) = \hat\mu(x).
\]

By \hyperlink{list:LF4}{(LF4)} of Section \ref{subsec:local Floer}, we obtain
\[
\dim\mathrm{HF}_{\mu_+(x)+1}(\varphi_H^l,0;\mathbb{Q}) = 0,
\]
which combined with \eqref{eq:SH mu+1} leads to
\begin{equation}\label{eq:HF geq1}
\dim\mathrm{HF}_{\mu_+(x)}(\varphi_H^l,0;\mathbb{Q}) \geq 1.
\end{equation}
Under the assumptions of Theorem \ref{thm:local max}, we have:
\begin{equation}\label{eq:mu+ 0}
\mu_+(0,\varphi_H^l) = d + n - 1, \quad \text{where} \quad |d - \hat{\mu}(0,\varphi_H^l)| < \tfrac{1}{2}.
\end{equation}
Since \eqref{eq:HF geq1} and \eqref{eq:mu+ 0} satisfy the conditions of Theorem \ref{thm:vanish} on $B^{2n-2}$, we conclude
\begin{equation}\label{eq:l vanish}
\mathrm{HF}_*(\varphi_H^l,0;\mathbb{Q}) = 0 \quad \text{for all $*\neq\mu_+(x)$}.
\end{equation}

Furthermore, from \eqref{eq:SH split} and \eqref{eq:SH leq HF}, we obtain:
\begin{align*}
0 &\leq \dim\mathrm{CH}_{*}(y^l;\mathbb{Q}) + \dim\mathrm{CH}_{*-1}(y^l;\mathbb{Q}) \\
&\leq \dim\mathrm{HF}_{*}(\varphi_H^l,0;\mathbb{Q}) + \dim\mathrm{HF}_{*-1}(\varphi_H^l,0;\mathbb{Q}),
\end{align*}
which implies
\[
\mathrm{CH}_{*}(x;\mathbb{Q}) = 0 \quad \text{for all $*\neq\mu_+(x)$}.
\]

For the additional part, assume both $x=z^k$ and $z^{k'}$ are admissible iterations of $z = y^m$, $l=mk$.    By \cite[Theorem 1.1]{GG10}, we have isomorphism
\begin{equation}\label{eq:GG10}
\mathrm{HF}_{*+s}(\varphi_H^{l},0;\mathbb{Q}) = \mathrm{HF}_*(\varphi^{l'}_H,0;\mathbb{Q}),
\end{equation}
where $l' = mk'$ and $s=\mu_+(x)-\mu_+(z^{k'})$.

Combining \eqref{eq:l vanish} with \eqref{eq:GG10}  yields
\[
\mathrm{HF}_*(\varphi^{l'}_H,0;\mathbb{Q}) = 0 \quad \text{for all $*\neq\mu_+(z^{k'})$}.
\]

Finally, applying \eqref{eq:SH split} and \eqref{eq:SH leq HF} again we get
\[
\mathrm{CH}_{*}(z^{k'};\mathbb{Q}) = 0 \quad \text{for all $*\neq\mu_+(z^{k'})$}.
\]
\end{proof}

\begin{remark}\label{rmk:explicit shift}
 The explicit expression of the shift $s$ in \eqref{eq:GG10}, while not stated in the original formulation of \cite[Theorem 1.1]{GG10}, can be extracted from the proof in \cite[Section 4]{GG10}. For clarity, we present the detailed explanation below. Our goal is to show that:

\vspace{1mm}
{\it Assume $z$ is an isolated closed Reeb orbit and  $z^k$ is an admissible iteration. Let  $\varphi_H$ be the Poincar\'e return map of $z$, then there holds isomorphism
\begin{equation*}
\mathrm{HF}_{*+s}(\varphi_H^{k},0;\mathbb{Q}) = \mathrm{HF}_*(\varphi_H,0;\mathbb{Q}),~~\text{where $s=\mu_+(z^k)-\mu_+(z)$.} 
\end{equation*}
}

\begin{proof}
Recall the Bott-type formula for nullity (cf.\cite[Theorem 9.2.1]{Long02}):
\begin{equation*}
\nu(z^k)=\sum_{\omega^k=1}\nu_\omega(z),
\end{equation*}
where the $\omega$-nullity $\nu_\omega(\cdot)$ of a orbit is defined in Remark \ref{rmk: nullity for orbit}.
Since $z^k$ is an admissbile iteration, we have $\nu(z^k)=\nu(z)$, and consequently by \eqref{eq:+-nullity},
\begin{equation}\label{eq:+- diff equal}
  \mu_+(z^k)-\mu_+(z)=\mu_-(z^k)-\mu_-(z).
\end{equation}

Now we divide the proof into three cases:

  {\textbf{Case 1:}} {\it $z$ is non-degenerate.}
 In this case, both $z^k$ and $z$ are non-degenerate, hence we apply Example \ref{exam:non degen HF} to get 
\begin{equation*}
  s=\mu(0,\varphi_H^k)-\mu(0,\varphi_H)=\mu(z^k)-\mu(z),
\end{equation*} 
where the second equality holds by Remark \ref{rmk:same ind}.

  {\textbf{Case 2:}} {\it $z$ is totally degenerate.}
According to \cite[Page 339]{GG10}, in this case
\begin{align*}
  s=\hat\mu(0,\varphi_H^k)-\hat\mu(0,\varphi_H)
  =\mu_-(0,\varphi_H^k)-\mu_-(0,\varphi_H)
  =\mu_-(z^k)-\mu_-(z),
\end{align*} 
where the second equality follows from Proposition \ref{prop:precise} and Corollary \ref{coro:precise mean ind}. 

{\textbf{Case 3:}} {\it general case} As in \cite[Section 4.5]{GG10},  choose a symplectic decomposition of $T_0B^{2n-2}\simeq\mathbb{R}^{2n-2}=V_0\times V_1$ such that all eigenvalues of $(\mathrm{d}\varphi_H)_0|_{V_0}$ are different from 1 and all eigenvalues of $(\mathrm{d}\varphi_H)_0|_{V_1}$ are equal to 1, then one can find a Hamiltonian $K$ with time-1 map $\varphi_K$ defined in a neighborhood of origin in $\mathbb{R}^{2n-2}$ such that:

\begin{itemize} 
  \item There exists a homotopy of Hamiltonian diffeomorphisms $\varphi_s$, $s\in[0,1]$, such that $\varphi_0=\varphi_H$, $\varphi_1=\varphi_K$ and $(\mathrm{d}\varphi_s)_0\equiv (\mathrm{d}\varphi_H)_0$;
  \item  $\varphi_K$ is split as $\varphi_K=(\varphi_{K_0},\varphi_{K_1})$, where $\varphi_{K_0}$ ($\varphi_{K_1}$ resp.) denotes the time-1 map of the Hamiltonian $K_0$ on $V_0$ ($K_1$ on $V_1$ resp.) defined in a neighborhood of the origin; 
  \item  $\hat\mu(0,\varphi_H)=\hat\mu(0,\varphi_K)$, and for each admissible iteration $z^k$, there holds isomorphism
  \[\mathrm{HF}_*(\varphi_H^k,0;\mathbb{Q})=\mathrm{HF}_*(\varphi_K^k,0;\mathbb{Q}).\]
\end{itemize}  
Following Subsection \ref{subsec:split}, by \eqref{eq:H0+H1} and \hyperlink{list:MI4}{(MI4)} of Section \ref{subsec:index}, we have  
\begin{align}
  &\mathrm{HF}_{*+m}(\varphi_K,0;\mathbb{Q})=\mathrm{HF}_*(\varphi_{K_0+K_1},0;\mathbb{Q}),~~~~m=\hat\mu(0,\varphi_K)-\hat\mu(0,\varphi_{K_0+K_1}),\label{eq:shift m}\\ 
  &\mathrm{HF}_{*+{m_k}}(\varphi_K^k,0;\mathbb{Q})=\mathrm{HF}_*(\varphi^k_{K_0+K_1},0;\mathbb{Q}),~~~~m_k=km, \label{eq:shift mk}
\end{align}
where $m$ is even.  
Note that the flow $\varphi^t_{(K_0+K_1)^{\#k}}=(\varphi^t_{K_0^{\#k}},\varphi^t_{K_1^{\#k}})$ is split, according to \hyperlink{list:LF2}{(LF2)} of Section \ref{subsec:local Floer}, we have
\begin{equation}\label{eq:split iterate}
\begin{aligned}
  \mathrm{HF}_*(\varphi_{K_0+K_1},0;\mathbb{Q})
  &=\bigoplus_{*'+*''=*}\mathrm{HF}_{*'}(\varphi_{K_0},0;\mathbb{Q})\otimes\mathrm{HF}_{*''}(\varphi_{K_1},0;\mathbb{Q})\\
  &=\bigoplus_{*'+*''=*}\mathrm{HF}_{*'+s_0}(\varphi_{K_0}^k,0;\mathbb{Q})\otimes\mathrm{HF}_{*''+s_1}(\varphi_{K_1}^k,0;\mathbb{Q})\\
  &=\mathrm{HF}_{*+s_0+s_1}(\varphi_{K_0+K_1}^k,0;\mathbb{Q}),
  \end{aligned}
\end{equation}
where
\begin{equation}\label{eq:shift s0 s1}
  s_0=\mu(0,\varphi_{K_0}^k)-\mu(0,\varphi_{K_0})~~\text{and $s_1=\hat\mu(0,\varphi_{K_1}^k)-\hat\mu(0,\varphi_{K_1})$}
\end{equation}
follow from Case 1 and Case 2. Summarizing \eqref{eq:shift m}-\eqref{eq:shift s0 s1} yields
\begin{equation}\label{eq:s 1}
\begin{aligned}
s&=m_k+s_0+s_1-m\\
&=(k-1)[(\hat\mu(0,\varphi_K)-\hat\mu(0,\varphi_{K_0+K_1})+\hat\mu(0,\varphi_{K_1})]
+\mu(0,\varphi_{K_0}^k)-\mu(0,\varphi_{K_0}).
\end{aligned}
\end{equation}

Let $\Phi\in\widetilde{\mathrm{Sp}}(2n-2)$ denote the associated symplectic path of $z$ as in Subsection \ref{subsect:index for orbits}. By Remark \ref{rmk:same ind} and the fact that $\varphi_K$ is homotopic to $\varphi_H$ by $\varphi_s$ with fixed $(\mathrm{d}\varphi_s)_0$, we conclude that
\begin{equation*}
 \hat\mu(0,\varphi_K)=\hat\mu(\Phi).
\end{equation*}
Let $n_0:=\dim V_0$, $n_1:=\dim V_1$, $\Psi_0\in\widetilde{\mathrm{Sp}}(n_0)$ and $\Psi_1\in\widetilde{\mathrm{Sp}}(n_1)$ denote the symplectic paths given by  $(\mathrm{d}\varphi_{K_0}^t)_0$ and $(\mathrm{d}\varphi_{K_1}^t)_0$, $t\in[0,1]$ respectively.  Note that $\Psi_0$ is non-degenerate while $\Psi_1$ is totally degenerate, and 
\begin{equation*}
  \hat\mu(0,\varphi_{K_0+K_1})=\hat\mu(\Psi_0)+\hat\mu(\Psi_1),~~\mu(0,\varphi_{K_0})=\mu(\Psi_0),~~ \hat\mu(0,\varphi_{K_1})=\hat\mu(\Psi_1).
\end{equation*}
Then \eqref{eq:s 1} can be rewritten as
\begin{equation}\label{eq:s 2}
  s=(k-1)[(\hat\mu(\Phi)-\hat\mu(\Psi_0)]+\mu(\Psi_0^k)-\mu(\Psi_0).
\end{equation}

 According to \eqref{eq:differ by maslov}, we can modify $\Psi_0$  by concatenating at the end point a symplectic loop $\psi$ with Maslov index $(\hat\mu(\Phi)-\hat\mu(\Psi_0))/2$ (such a loop can be constructed directly by rotation) so that
the resulting path $\Psi$ satisfying
\begin{equation}\label{eq:nondegenerate part}
\left\{
\begin{aligned}
&\hat\mu(\Psi)=\hat\mu(\Phi),\\
&\Phi(1)=\Psi(1)\diamond \Psi_1(1),
\end{aligned}
\right.
\end{equation}
Moreover, 
\begin{align}
&\mu(\Psi)-\mu(\Psi_0)=\hat\mu(\Psi)-\hat\mu(\Psi_0),\label{eq:Psi Psi0}\\
&\mu(\Psi^k)-\mu(\Psi_0^k)=k (\hat\mu(\Psi)-\hat\mu(\Psi_0)),\label{eq:k Psi Psi0}            
\end{align}
note that path $\Psi^k$ can be viewed as $\Psi_0^k$ concatenated with the symplectic loop $\psi_i:=\Psi_0^{i-1}(1)\psi$  at the point $\Psi^i_0(1)$ for $i=1,\dots,k$, and all these $k$ loops have the same Maslov index with $\psi$.  Substituting \eqref{eq:nondegenerate part}-\eqref{eq:k Psi Psi0} into \eqref{eq:s 2} yields
\begin{equation}
  s=\mu(\Psi^k)-\mu(\Psi).
\end{equation}

According to Lemma \ref{lem:prop  split} and Corollary \ref{coro:precise mean ind}, \eqref{eq:nondegenerate part} implies 
\begin{equation*}
  \mu(\Psi)=\mu_-(\Phi)+S_{\Phi(1)}^+(1),
\end{equation*}
which yields 
\begin{equation*}
  s=\mu_-(\Phi^k) - \mu_-(\Phi)=\mu_-(z^k) - \mu_-(z)
\end{equation*}
by Proposition \ref{prop:precise}.

Summarizing these cases and \eqref{eq:+- diff equal}, we conclude that 
\begin{equation}\label{eq:degree shift}
  s=\mu_-(z^k) - \mu_-(z)=\mu_+(z^k)-\mu_+(z).
\end{equation}
\end{proof}
\end{remark}

\begin{remark}\label{rmk:nondeg part}
  Let $\Phi\in\widetilde{\mathrm{Sp}}(2n)$. In \cite[Section 3.1.1]{CGG24}, a non-degenerate $\Psi\in \widetilde{\mathrm{Sp}}(2m)$ with $0\leq m \leq n$ satisfying \eqref{eq:nondegenerate part} is called the {\it non-degenerate part} of $\Phi$,  i.e., 
 \begin{equation*}
\left\{
\begin{aligned}
&\hat\mu(\Psi)=\hat\mu(\Phi),\\
&\Psi(1)\diamond A=\Phi(1), ~\text{for some $A\in\mathrm{Sp}(2n-2m)$ with $\sigma(A)=\{1\}$}.
\end{aligned}
\right.
\end{equation*}
Such $\Psi$ is uniquely determined as elements in $\Psi\in \widetilde{\mathrm{Sp}}(2m)$ since its Conley-Zehnder index  and end point are determined.
For instance, the non-degenerate part of $\Phi$ is given by $\Psi=\Phi$ if $\Phi$ is non-degenerate, and is trivial if $\Phi$ is totally degenerate. As already mentioned in the above remark, Proposition \ref{prop:precise} together with Corollary \ref{coro:precise mean ind} yields
\begin{equation*}
  \mu(\Psi^k)-\mu(\Psi)=\mu_-(\Phi^k)-\mu_-(\Phi),~~\forall k\in\mathbb{N}. 
\end{equation*}
whenever $\Psi$ is the non-degenerate part of $\Phi$.
\end{remark}

\section{Two irrationally elliptic closed Reeb orbits}\label{sec:two irra}

\subsection{Common index jump intervals}\label{subsec:ind interval}
Let $\Sigma^{2n-1}$ be the boundary of a star-shaped domain $W\subset\mathbb{R}^{2n}$. Assume that

\begin{description}
    \item[\hypertarget{list:DF}{\textbf{(DF)}}] {\it the Reeb flow on $\Sigma^{2n-1}$ is dynamically convex and has finitely many prime closed orbits.} 
\end{description}                                                                                                                                       
Under this assumption,  \c{C}ineli-Ginzburg-G\"urel proved in \cite[Theorem D]{CGG24} that for each isolated closed Reeb orbit $x$ (not necessarily prime) on $\Sigma$, its equivariant local symplectic homology $\mathrm{CH}_*(x;\mathbb{Q})$ satisfies
\begin{equation*}
    \dim \mathrm{CH}_*(x;\mathbb{Q})\leq 1.
\end{equation*}
Denote by $\mathcal{P}$ the set of all $\mathbb{Q}$-visible isolated closed Reeb orbits. Then for each $x\in\mathcal{P}$, there exists some $h\in\mathbb{Z}$ such that 
\begin{equation}\label{eq:deg}
\mathrm{CH}_*(x;\mathbb{Q})=\left\{
\begin{aligned}
  \mathbb{Q},~~~~&*=h,\\
 0,~~~~&\text{otherwise}.
\end{aligned}
\right.    
\end{equation}

We call $h$ from \eqref{eq:deg} {\it the degree of $x$}, denoted by $\deg(x)$. By \eqref{eq:supp CH}, the degree is bounded by the Maslov-type indices, namely,
\begin{equation}\label{eq:deg bound}
    \deg(x)\in[\mu_-(x),\mu_+(x)]\subset[\hat\mu(x)-n+1,\hat\mu(x)+n-1].
\end{equation}
Furthermore, the degree map has the following bijective property:
\begin{theorem}[ {\cite[Theorem E]{CGG24}}]\label{thm:bijection}
 Assume \hyperlink{list:DF}{\textbf{(DF)}} holds, then $\mathcal{P}$ has the following properties:
 \begin{itemize}
     \item[\hypertarget{list:O1}{(O1)}] The map $\deg:\mathcal{P}\to\mathbb{Z}$ is a bijection onto $n-1+2\mathbb{N}=\{n+1,n+3,\dots\}$, and orbits with larger degree have larger period, i.e. $\deg(x)>\deg(y)$ implies $\mathcal{A}(x)>\mathcal{A}(y)$ for each $x$, $y\in\mathcal{P}$. 
     \item[\hypertarget{list:O2}{(O2)}] All closed orbits $x\in\mathcal{P}$ have the same ration $\mathcal{A}(x)/\hat\mu(x)$.
 \end{itemize}
\end{theorem}

Let $x_0\in\mathcal{P}$ be the orbit with the least action.  \hyperlink{list:O1}{(O1)} implies $\deg(x_0)=n+1$. If $x_0=z^k$ with $k\geq 2$, by Lemma \ref{lem:ind mono} and dynamical convexity we have
\[
\mu_-(x_0)=\mu_-(z^k)>n+1=\deg(x_0),
\]
which contradicts \eqref{eq:deg bound}. Thus,  $x_0$ must be prime and $\mu_-(x_0)=n+1=\deg(x_0)$.

Let  $\{x_0,\dots,x_r\}$ be all the prime closed orbits which are {\it eventually $\mathbb{Q}$-visible}, i.e., $x_i^k\in\mathcal{P}$ for some $k\in\mathbb{N}$. 
Note that the orbits $x_{i\geq1}$ might be $\mathbb{Q}$-invisible and/or have action below $\mathcal{A}(x_0)$. For each $i\geq1$ we define $s_i$ by 
\begin{equation}\label{eq:si}
    s_i:=\left\{
    \begin{aligned}
        &\max\{s\in\mathbb{N}:\mathcal{A}(x_i^s)\leq\mathcal{A}(x_0)\},~~~~ &\text{if $\mathcal{A}(x_i)\leq\mathcal{A}(x_0)$;}\\
        &0, &\text{if $\mathcal{A}(x_i)>\mathcal{A}(x_0).$}
    \end{aligned}
    \right.
\end{equation}

Consider common index jump events $\{k_{ij},d_j\}_{j=1}^{\infty}$ as in Theorem \ref{thm:CIJT}, where we require $\eta>0$ to be sufficiently small, all $d_j$ and $k_{ij}$ are divisible by $N\in\mathbb{N}$ specified as follows. Let \[
p:=\mathrm{lcm}\{k\in\mathbb{N}: \text{$\exists~0\leq i\leq r$, $\lambda\in\mathbb{C}$, s.t. $\lambda\in\sigma(x_i)$, $\lambda^k=1$  }\},
\] 
where $\sigma(x_i)$ denote the set of all eigenvalues of the linearized Poincar\'e map of $x_i$, and $\mathrm{lcm}\{\cdot\}$ denotes the least common multiple of the numbers in the set. Then we let
\begin{equation}\label{eq:N}
    N=2p\prod_{i=1}^{r} s_i!
\end{equation}
where, as usual, we set $0!=1$.

Concentrating on one event, we drop $j$ from the notation to get $d=d_j$ and $k_i=k_{ij}$. The $\mathbb{Q}$-visible closed orbits $x_j^k$ are divided into three collections:
\begin{itemize}
    \item $\Gamma_+$ formed by the $\mathbb{Q}$-visible orbits $x_i^{k>k_i}$. Then $\mu_-(x_i^k)\geq d+n+1$ by \eqref{eq:CIJT low} with $m=n-1$, and hence $\deg(x_i^k)\geq d+n+1$ by \eqref{eq:deg bound}.
    \item $\Gamma_0$ comprising the $r+1$ $\mathbb{Q}$-visible orbits $x_i^{k_i}$. Then by  \eqref{eq:bounded by mean} with $m=n-1$ and \eqref{eq:IR1},
    \[
    d-n+1\leq \mu_-(x_i^{k_i})\leq\mu_+(x_i^{k_i})\leq d+n-1,
    \]
    and hence 
    \begin{equation}\label{eq:ind interval}
        \deg(x_i^{k_i})\in\mathcal{I}:=[d-n+1,d+n-1]\subset \mathbb{N}
    \end{equation}
    by \eqref{eq:deg bound}.
    \item $\Gamma_-$ formed by the $\mathbb{Q}$-visible orbits $x_i^{k<k_i}$. Then $\mu_+(x_i^k)\leq d-2$ by \eqref{eq:CIJT up}, and hence $\deg(x_i^k)\leq d-2$.
\end{itemize}
 For each common index jump event,  the interval $\mathcal{I}$ in \eqref{eq:ind interval} is called {\it the associated  common index jump interval}. 
 
In fact, it is shown in \cite[Section 4.1.1]{CGG24}  that the orbits from $\Gamma_-$ must have a degree less than $d-n-1$, thus cannot have a degree in $\mathcal{I}$.  (For the reader's convenience we will briefly recall the proof for this argument later. The construction of $N$ in \eqref{eq:N} is essentially used there.) Consequently, the only closed orbits with degree in $\mathcal{I}$ must be from $\Gamma_0$. Note that there are exactly $n$ integer spots in $\mathcal{I}$ of parity $n+1$, which must all be filled by $\deg(x_i^{k_i})$ when $d$ is sufficiently large; see \hyperlink{list:O1}{(O1)} in Theorem \ref{thm:bijection}.  This implies that $r+1\geq n$.

Now we recall the sketch of the proof for
\begin{equation}\label{eq:deg leq d-n}
\deg(x_i^k)\leq d-n-1 ~~ \text{\it when $k<k_i$ and $x_i^k\in\mathcal{P}$}.
\end{equation}
(The terminology has been adjusted to maintain notational consistency throughout this article,  which leads to a slight simplification of the original proof in \cite{CGG24}.)
\begin{proof}
Start with $x=x_0$ and $k=k_0-1$. Recall that  $x$ is prime with $\mu_-(x)=\deg(x)=n+1$. By construction of $N$ we see $x^k$ is an admissible iteration and $k$ is odd, which implies $x^k\in \mathcal{P}$ by 
\begin{equation*}
  \chi^{\rm eq}(x^k)=\chi^{\rm eq}(x)=(-1)^{n+1}\neq 0,
\end{equation*}
which follows from Lemma \ref{lem:euler}.  
Since $x$ is prime, by \eqref{eq:SH split} and \eqref{eq:SH prime} one can compute $\deg(x^k)$ through the local Floer homology of its Poincar\'e return map to get
\begin{equation}\label{eq:deg xk}
\deg(x^k)=\deg(x)+\mu_+(x^k)-\mu_+(x),
\end{equation}
 see \cite[Lemma 4.2]{CGG24}, Remark \ref{rmk:explicit shift} and Remark \ref{rmk:nondeg part}. Applying \eqref{eq:IR3} of Theorem \ref{thm:CIJT} to \eqref{eq:deg xk} yields
 \begin{equation}
   \begin{aligned}
   \deg(x^k)
   &=\deg(x)+d-\mu_-(x)-2S^+_x(1)+\nu(x)-\mu_+(x)\\
   &=d-n-1-2S^+_x(1)\\
   &\leq d-n-1,
   \end{aligned}
 \end{equation}
where the inequality holds by Lemma \ref{lem:prop  split}.

For $k'<k$ we have $\mathcal{A}(x^{k'})<\mathcal{A}(x^k)$, which leads to $\deg(x^{k'})<\deg(x^k)\leq d-n-1$ by Theorem \ref{thm:bijection}, provided $x^{k'}\in\mathcal{P}$.

Finally we turn to the case $x_{i}^{k_i-s}$ with $i\geq1$ and $s>0$. Fix $i$. Note that \eqref{eq:deg leq d-n} is an immediately consequence of the following claim:

\vspace{1mm}
\textbf{Claim}: {\it $\deg(x_i^{k_i-s})>d-n$ for some $s>0$ implies  $x_i^{k_i-s}\notin \mathcal{P}$.}
\vspace{1mm}

Indeed, by \hyperlink{list:O1}{(O1)} of Theorem \ref{thm:bijection} and the fact $\deg(x_i^{k_i-s})> d-n-1\geq\deg{x_0^{k_0-1}}$, we have 
\begin{equation}\label{eq:action>}
  \mathcal{A}(x_i^{k_i-s})>\mathcal{A}(x_0^{k_0-1}).
\end{equation}
Combining \eqref{eq:IR1} of Theorem \ref{thm:CIJT} with \hyperlink{list:O2}{(O2)} of Theorem \ref{thm:bijection} yields
\begin{equation}\label{eq:action close}
  |\mathcal{A}(x_i^{k_i})-\mathcal{A}(x_0^{k_0})|<\mathrm{const} \cdot \eta. 
\end{equation}
Provied $\eta>0$ sufficiently small, \eqref{eq:action>} together with \eqref{eq:action close} implies $\mathcal{A}(x_i^s)<\mathcal{A}(x_0)$. Then  by the definition of $x_0$, $x_i^s\notin\mathcal{P}$. Consequently $\chi^{eq}(x_i^s)=0$. Moreover, $s\leq s_i$ with $s_i$ defined in \eqref{eq:si}, hence $k_i$ is divisible by $s$ by the construction of $N$. Set $k_i'=k_i/s$. Then $x_i^{k_i-s}=(x_i^s)^{k_i'-1}$ is an admissible odd iterate of $x_i^s$. Then by Lemma \ref{lem:euler} we have
\begin{equation*}
  \chi^{\rm eq}(x_i^{k_i-s})=\chi^{\rm eq}(x_i^s)=0,
\end{equation*} 
which implies $x_i^{k_i-s}\notin \mathcal{P}$ by \eqref{eq:deg}.
\end{proof}

\subsection{Proof of two irrationally elliptic closed Reeb orbits}\label{subsect:proof}
From now on, we assume that 

\begin{description}
    \item[\hypertarget{list:DF'}{\textbf{(DF')}}] {\it the Reeb flow on $\Sigma^{2n-1}$ is dynamically convex and has exactly $n$ prime closed orbits.} 
\end{description}
Denote by $\{y_1,\dots,y_n\}$ these $n$ prime closed Reeb orbits. By our discussion in  the previous subsection, for each common index jump event $\{k_i,d\}$ with $d$ sufficiently large, the $n$ integer spots of parity $n+1$ in the common index jump interval $\mathcal{I}$, must be filled by distinct values of $\deg(y_i^{k_i})$ for $1\leq i \leq n$. (This implies that all the $n$ prime closed Reeb orbits are eventually $\mathbb{Q}$-visible! )

Furthermore,  by the same proof of  \cite[Lemma 4.4]{Wang22},  the mean indices of closed Reeb orbits in the common index jump interval have the following commutative property:
\begin{proposition} \label{prop:commut}
There exists two common index jump events $\{k_{i},d\}$ and $\{k'_{i},d'\}$ such that
for any  integers $1\leq i,j\leq n$ and $i\neq j$,
we have 
\[
\hat\mu(y_i^{k_i})<\hat\mu (y_j^{k_j}), ~ \text{while $\hat\mu(y_i^{k'_i})>\hat\mu(y_j^{k'_j})$},
\]
i.e., the ordering of mean indices of any two closed Reeb orbits with degree in the common index jump intervals gets
interchanged. 
\end{proposition}

Theorem \ref{thm:bijection} implies that  the ordering of mean indices is exactly the ordering of degrees for any two closed Reeb orbits in $\mathcal{P}$. Combining this with Proposition \ref{prop:commut}, after reordering  $\{y_1,\dots,y_n\}$
if necessary, we can find  common index jump events $\{k_{i},d\}$ and $\{k'_{i},d'\}$ with $d<d'$
such that the following diagrams hold:
\begin{eqnarray}\nonumber
 \begin{tabular}{lccccc}
&$d-n+1$&$d-n+3$& $\cdots$&$d+n-3$& $d+n-1$\\
&$\Vert$&$\Vert$& $\qquad$&$\Vert$& $\Vert$\\
&$\deg(y_1^{k_1})$&$\deg(y_2^{k_2})$& $\qquad$&$\deg(y_{n-1}^{k_{n-1}})$&$\deg(y_n^{k_n})$\\
\end{tabular}   
\end{eqnarray}
\begin{eqnarray}\nonumber
 \begin{tabular}{lccccc}
&$d'-n+1$&$d'-n+3$& $\cdots$&$d'+n-3$& $d'+n-1$\\
&$\Vert$&$\Vert$& $\qquad$&$\Vert$& $\Vert$\\
&$\deg(y_n^{k'_n})$&$\deg(y_{n-1}^{k'_{n-1}})$& $\qquad$&$\deg(y_2^{k'_2})$&$\deg(y_1^{k'_1})$\\
\end{tabular}   
\end{eqnarray}

\quad

\begin{proof}[Proof of Theorem \ref{thm:main}]
    
In the following we will show that both $y_1$ and $y_n$ are irrational elliptic. By the symmetric positions of $y_1$ and $y_n$, we only state the proof for $y_1$ here, which is divided into several steps:

\quad

{\bf Step 1.} {\it We show that $\deg(y_1^{k'_1})=\mu_{+}(y_1^{k'_1})$ and $\deg(y_1^{k_1})=\mu_{+}(y_1^{k_1})$}.

 Abbreviate  $\mu_+(y_1^{k'_1})$ and $\mu_+(y_1^{k_1})$ by $\mu_+'$ and $\mu_+$ respectively. By \eqref{eq:deg bound}, 
\begin{equation}\label{eq:mu+ bound}
d'+n-1=\deg(y_1^{k'_1})\leq \mu_+'\leq\hat\mu(y_1^{k'_1})+n-1.   
\end{equation}
In Theorem \ref{thm:CIJT} we have $|d'-\hat\mu(y_1^{k'_1})|<\eta$ and we have assumed that $\eta$ is sufficiently small (for instance, $0<\eta<\frac{1}{2}$). Since $\mu_+'$ is an integer, \eqref{eq:mu+ bound} implies that
\begin{equation}\label{eq:mu+'}
    \mu_+'=d'+n-1=\deg(y_1^{k'_1}).
\end{equation}
%Moreover, by \eqref{eq:SH kx}, \eqref{eq:SH split} and the definition of $\deg(\cdot)$ below \eqref{eq:deg}, $\deg(y_1^{k'_1})= \mu_+'$ implies 
%\begin{equation}\label{eq:HF geq1}
%\dim \mathrm{HF}_{\mu_+'}(\varphi^{k_1'},0;\mathbb{Q})\geq 1,
%\end{equation}
%where $\varphi$ is the associated Poincar\'e return map of $y_1$, defined on a ball $B^{2n-2}$. Note that \eqref{eq:mu+'} and \eqref{eq:HF geq1} meet the conditions in Theorem \ref{thm:vanish}, then we have
%\begin{equation}\label{eq:k1' vanish}
    %\mathrm{HF}_*(\varphi^{k_1'},0;\mathbb{Q})=0~~\text{whenever $*\neq\mu_+'$}.
%\end{equation}
By the construction of $N$ in \eqref{eq:N}, we conclude that  both $y_1^{k_1'}$ and $y_1^{k_1}$ are  admissible iteration of $y_1^N$. %from \cite[Theorem 1.1]{GG10} we see that
%\begin{equation}\label{eq:GG10}
    %\mathrm{HF}_{*+s}(\varphi^{k_1'},0;\mathbb{Q})=\mathrm{HF}_*(\varphi^{k_1},0;\mathbb{Q}),~~\text{where $s=\mu_+'-\mu_+$}.
%\end{equation}
Then by \eqref{eq:mu+'},  we can apply Theorem \ref{thm:local max} to get

\begin{equation*}
    \mathrm{CH}_{*}(y_1^{k_1};\mathbb{Q})=0 ~~\text{whenever $*\neq\mu_+$},  
\end{equation*}
which implies \begin{equation}\label{eq:mu+ k1}
    \mu_+=\deg(y_1^{k_1})=d-n+1.
\end{equation}

 Step $1$ is finished by \eqref{eq:mu+'} and \eqref{eq:mu+ k1}.

 \quad

{\bf Step 2.} {\it We show that $y_1$ is strongly non-degenerate, i.e. any iteration of $y_1$ is non-degenerate.} 

Indeed, by \eqref{eq:deg bound} and \eqref{eq:mu+ k1}, 
\begin{equation*}%\label{eq:mu_ bound}
\hat\mu(y_1^{k_1})-n+1\leq\mu_-(y_1^{k_1})\leq\deg(y_1^{k_1})=d-n+1,   
\end{equation*}
which implies $\mu_-(y_1^{k_1})=d-n+1$ since $|d-\hat\mu(y_1^{k_1})|<\frac{1}{2}$. Combining this fact with \eqref{eq:mu+ k1} we have $\mu_-(y_1^{k_1})=\mu_+(y_1^{k_1})$, which implies $y_1^{k_1}$ is non-degenerate by \eqref{eq:+-nullity}. Since $k_1$ is divisible by $N$, we conclude that $y_1^N$ is also non-degenerate. Thus, by the construction of $N$ in \eqref{eq:N}, we know that there is no root of unity in $\sigma(y_1)$, i.e.,  for each eigenvalue $\lambda=e^{\sqrt{-1}\theta}\in\sigma(y)\cap\mathbb{U}$, $\theta/\pi$ is irrational. Consequently, $y_1$ is strongly non-degenerate.

\quad

{\bf Step 3.} {\it We show that $y_1$ is  elliptic.} 

Denoted by $\Phi\in\mathrm{Sp}(2n-2)$ the restricted linearized Poincar\'e return map of $y_1$. In Step 2 we have shown that $y_1$ is strongly non-degenerate, then by the abstract precise iteration formula \eqref{eq:precise}, we have
\begin{align}
    \mu(y_1^{{k_1}+1})-\mu(y_1^{k_1})
    &=\mu(y_1)+\sum_{\theta\in(0,2\pi),\frac{\theta}{\pi}\notin\mathbb{Q}}(2\lceil\frac{(k_1+1)\theta}{2\pi}\rceil-2\lceil\frac{k_1\theta}{2\pi}\rceil-1)S_{\Phi}^-(e^{\sqrt{-1}\theta})\nonumber\\
    &=\mu(y_1)
    +\sum_{\theta\in(0,\pi),\frac{\theta}{\pi}\notin\mathbb{Q}}(2\lceil\frac{(k_1+1)\theta}{2\pi}\rceil-2\lceil\frac{k_1\theta}{2\pi}\rceil-1)(S_{\Phi}^-(e^{\sqrt{-1}\theta})-S_{\Phi}^+(e^{\sqrt{-1}\theta})),\label{eq:precise k1}
\end{align}
where the second equality holds by the fact that $ S_{\Phi}^-(e^{\sqrt{-1}(2\pi-\theta)})=S_{\Phi}^+(e^{\sqrt{-1}\theta})$ and
\begin{equation}\label{eq:theta 2pi-theta}
 \lceil\frac{(k_1+1)(2\pi-\theta)}{2\pi}\rceil-\lceil\frac{k_1(2\pi-\theta)}{2\pi}\rceil=1-\left(\lceil\frac{(k_1+1)\theta}{2\pi}\rceil-\lceil\frac{k_1\theta}{2\pi}\rceil\right).   
\end{equation}
By \eqref{eq:IR2} and \eqref{eq:mu+ k1}, \eqref{eq:precise k1} is equivalent to 
\begin{equation}\label{eq:n-1}
    n-1=\sum_{\theta\in(0,\pi),\frac{\theta}{\pi}\notin\mathbb{Q}}(2\lceil\frac{(k_1+1)\theta}{2\pi}\rceil-2\lceil\frac{k_1\theta}{2\pi}\rceil-1)(S_{\Phi}^-(e^{\sqrt{-1}\theta})-S_{\Phi}^+(e^{\sqrt{-1}\theta})).
\end{equation}
From \hyperlink{list:S4}{(S4)} of Lemma \ref{lem:prop  split}, we  have 
\begin{equation}\label{eq:split bound}
    \left|S_{\Phi}^-(e^{\sqrt{-1}\theta})-S_{\Phi}^+(e^{\sqrt{-1}\theta})\right|\leq\nu_{e^{\sqrt{-1}\theta}}(\Phi),
\end{equation}
Note that 
\begin{equation}\label{eq:leq 1}
   \lceil\frac{k_1\theta}{2\pi}\rceil\leq \lceil\frac{(k_1+1)\theta}{2\pi}\rceil=\lceil\frac{k_1\theta}{2\pi}+\frac{\theta}{2\pi}\rceil\leq\lceil\frac{k_1\theta}{2\pi}\rceil+1.
\end{equation}
Substituting  \eqref{eq:split bound} and \eqref{eq:leq 1} into \eqref{eq:n-1}, we see that
\begin{equation*}
    n-1\leq \sum_{\theta\in(0,\pi),\frac{\theta}{\pi}\notin\mathbb{Q}} \nu_{e^{\sqrt{-1}\theta}}(\Phi) \leq n-1,
\end{equation*}
which implies $y_1$ is elliptic.

\quad

{\bf Step 4.} {\it We show that $\Phi$ is symplectic similar to $R(\theta_1)\diamond R(\theta_2)\diamond\cdots\diamond R(\theta_{n-1})$ with $\frac{\theta_i}{\pi}\notin\mathbb{Q}$, which finishes the proof.} 

By (i) of Lemma \ref{lem:splt normal}, we  assume that  $\Phi$ can be connected within $\Omega^0(\Phi)$ to 
\begin{equation}\label{eq:basic decom}\begin{aligned}
 &R(\vartheta_1)\diamond\cdots\diamond R(\vartheta_r)
\diamond N_2(e^{\sqrt{-1}\rho_1}, u_1)\diamond\cdots\diamond N_2(e^{\sqrt{-1}\rho_{r_*}}, u_{r_*})\\
&\diamond N_2(e^{\sqrt{-1}\lambda_1}, v_1)\diamond\cdots\diamond N_2(e^{\sqrt{-1}\lambda_{r_0}}.v_{r_0}),   
\end{aligned}
\end{equation}
where $N_2(e^{\sqrt{-1}\rho_j}, u_j) $s are
non-trivial and   $ N_2(e^{\sqrt{-1}\lambda_{r_j}}, v_j)$s  are trivial basic normal
forms in \eqref{eq:4x4}; $\vartheta_j$, $\rho_j$, $\lambda_j\in(0,2\pi)$ with $\frac{\vartheta_j}{\pi}$, $\frac{\rho_j}{\pi}$, $\frac{\lambda_j}{\pi}\notin \mathbb{Q}$; and $r+2(r_*+r_0)=n-1$.
Then by Lemma \ref{lem:splt normal}, \eqref{eq:n-1} becomes
\begin{align*}
    n-1&=\sum_{1\le j\le r}(2\lceil\frac{(k_1+1)\vartheta_j}{2\pi}\rceil-2\lceil\frac{k_1\vartheta_j}{2\pi}\rceil-1)\\
    &\leq r\leq n-1,
\end{align*}
where the inequality holds by \eqref{eq:leq 1}. Consequently, $r=n-1$,  $r_*=0=r_0$,  and the basic normal form decomposition  \eqref{eq:basic decom} is precisely
\begin{equation}\label{eq:rotation decom}
    R(\vartheta_1)\diamond\cdots\diamond R(\vartheta_{n-1}), 
\end{equation}
with $\vartheta_j\in (0,2\pi)$, $\frac{\vartheta_j}{\pi}\notin\mathbb{Q}$ and  
\begin{equation}\label{eq:=1}
     \lceil\frac{(k_1+1)\vartheta_j}{2\pi}\rceil-\lceil\frac{k_1\vartheta_j}{2\pi}\rceil=1
\end{equation}for each $1\leq j\leq n-1$. 

As a direct consequence of \eqref{eq:theta 2pi-theta} and \eqref{eq:=1}, we have the following claim:

{\bf Claim 1.} {\it If $R(\theta)$ appears in the basic normal form decomposition
\eqref{eq:rotation decom}, then $R(2\pi-\theta)$ does not  appear.}

Finally, we finish the proof by the following assertion:

{\bf Claim 2.}  {\it There exists $Q_1\in \mathrm{Sp}(2n-2)$ such that
\begin{equation}\label{4.46}
Q_1^{-1}\Phi Q_1=R(\theta_1)\diamond\cdots\diamond R(\theta_{n-1})\end{equation}
with $\theta_i=\vartheta_i$ or $2\pi-\vartheta_i$ for $1\le i\le n-1$.}

 In fact, suppose $\omega^{\pm 1}_1, \ldots, \omega^{\pm 1}_{n-1}$ with $\omega_j = e^{\sqrt{-1}\vartheta_j}$ are the eigenvalues of $\Phi$. By \cite[Theorem 1.6.11]{Long02}, there exists $X_1 \in \mathrm{Sp}(2n-2)$ such that
\begin{equation}\label{4.47} 
X_1\Phi X_1^{-1} = S_1 \diamond \cdots \diamond S_{m_1} \diamond S_0,
\end{equation}
where $S_0 \in \mathrm{Sp}(2k_0)$ with $k_0 \geq 0$ and $\omega_1 \notin \sigma(S_0)$, while for $1 \leq i \leq m_1$, $k_i \geq 1$ and $S_i \in \mathrm{Sp}(2k_i)$ is in the normal form $N_{k_i}(\lambda_i, b_i)$ with $\lambda_i = \omega_1$ or $\omega_1^{-1}$ as defined in \cite[Section 1.6]{Long02}.

Applying Cases 3 and 4 from \cite[Section 1.8]{Long02}, if $k_i \geq 3$ for some $1 \leq i \leq m_1$, then $S_i$ can be connected within $\Omega^0(S_i)$ to some $\widetilde{S_i}$ whose total multiplicity of elliptic eigenvalues is strictly less than that of $S_i$. This contradicts \eqref{eq:rotation decom}. Hence $S_i = N_{l_i}(\lambda_i, b_i)$ with $l_i = 1$ or $2$ for $1 \leq i \leq m_1$.

We now show $l_i = 1$ for all $1 \leq i \leq m_1$. Suppose $l_i = 2$ for some $i$. Then by \eqref{eq:rotation decom}, $S_i$ is not a basic normal form, implying
\[
S_i \in \mathcal{M}_{\omega_1}^2(4) := \{ M \in \mathrm{Sp}(4) : \dim_{\mathbb{C}} \ker_{\mathbb{C}}(M - \omega_1 \mathrm{Id}) = 2 \}.
\]
By Case 4 in \cite[Section 1.8]{Long02}, $S_i$ can be connected within $\Omega^0(S_i)$ to $R(\omega_1) \diamond R(2\pi - \omega_1)$, contradicting Claim 1. Thus (\ref{4.47}) simplifies to
\begin{equation}\label{4.48} 
X_1\Phi X_1^{-1} = R(\theta_1)^{\diamond m_1} \diamond S_0,
\end{equation}
with $\theta_1 = \vartheta_1$ or $2\pi - \vartheta_1$. Iterating this argument at most $n-1$ times yields (\ref{4.46}).
\end{proof}

\quad

\addtocontents{toc}{\protect\begingroup}
\addtocontents{toc}{\protect\setcounter{tocdepth}{-10}}
\subsection*{Acknowledgements}
\addtocontents{toc}{\protect\endgroup}
H. Liu is partially supported by NSFC (Nos. 12371195, 12022111) and W. Wang is partially supported by NSFC (No. 12025101).

%\newpage
\addtocontents{toc}{\protect\setcounter{tocdepth}{1}}
%%%%%%%%%%%%%%%%%%%%%%%%%%%%%%%%%%%%%%%%%%%%%%%%%%%%%%%%%%%%
%%%%%%%%%%                APPENDIX                %%%%%%%%%%
%%%%%%%%%%%%%%%%%%%%%%%%%%%%%%%%%%%%%%%%%%%%%%%%%%%%%%%%%%%%

%\appendix

%%%%%%%%%%%%%%%%%%%%%%%%%%%%%%%%%%%%%%%%%%%%%%%%%%%%%%%%%%%%
%%%%%%%%%%                END PAGES               %%%%%%%%%%
%%%%%%%%%%%%%%%%%%%%%%%%%%%%%%%%%%%%%%%%%%%%%%%%%%%%%%%%%%%%

%%%%%%%%%%% Bibliography
\vspace{20pt}
\printbibliography

@article {Arn99,
    AUTHOR = {Arnaud, Marie-Claude},
     TITLE = {Existence d'orbites p\'eriodiques compl\`etement elliptiques
              des hamiltoniens convexes pr\'esentant certaines sym\'etries},
   JOURNAL = {C. R. Acad. Sci. Paris S\'er. I Math.},
  FJOURNAL = {Comptes Rendus de l'Acad\'emie des Sciences. S\'erie I.
              Math\'ematique},
    VOLUME = {328},
      YEAR = {1999},
    NUMBER = {11},
     PAGES = {1035--1038},
      %ISSN = {0764-4442},
   MRCLASS = {37J45 (34C25)},
  MRNUMBER = {1696202},
MRREVIEWER = {Gabriella\ Tarantello},
       DOI = {10.1016/S0764-4442(99)80320-6},
       URL = {https://doi.org/10.1016/S0764-4442(99)80320-6},
}

@article {AM22,
    AUTHOR = {Abreu, Miguel and Macarini, Leonardo},
     TITLE = {Dynamical implications of convexity beyond dynamical
              convexity},
   JOURNAL = {Calc. Var. Partial Differential Equations},
  FJOURNAL = {Calculus of Variations and Partial Differential Equations},
    VOLUME = {61},
      YEAR = {2022},
    NUMBER = {3},
     PAGES = {Paper No. 116, 47},
      %ISSN = {0944-2669,1432-0835},
   MRCLASS = {53D40 (37J55)},
  MRNUMBER = {4411705},
MRREVIEWER = {Vittorio\ Coti Zelati},
       DOI = {10.1007/s00526-022-02228-1},
      URL = {https://doi.org/10.1007/s00526-022-02228-1},
}

@article {AM17,
    AUTHOR = {Abreu, Miguel and Macarini, Leonardo},
     TITLE = {Dynamical convexity and elliptic periodic orbits for {R}eeb
              flows},
   JOURNAL = {Math. Ann.},
  FJOURNAL = {Mathematische Annalen},
    VOLUME = {369},
      YEAR = {2017},
    NUMBER = {1-2},
     PAGES = {331--386},
      %ISSN = {0025-5831,1432-1807},
   MRCLASS = {37J45 (53D10)},
  MRNUMBER = {3694649},
MRREVIEWER = {Gabriele\ Benedetti},
       DOI = {10.1007/s00208-017-1532-4},
       URL = {https://doi.org/10.1007/s00208-017-1532-4},
}

@article {BG92,
    AUTHOR = {Barge, J. and Ghys, \'E.},
     TITLE = {Cocycles d'{E}uler et de {M}aslov},
   JOURNAL = {Math. Ann.},
  FJOURNAL = {Mathematische Annalen},
    VOLUME = {294},
      YEAR = {1992},
    NUMBER = {2},
     PAGES = {235--265},
      %ISSN = {0025-5831,1432-1807},
   MRCLASS = {55U99 (11F20 57R20)},
  MRNUMBER = {1183404},
MRREVIEWER = {V.\ P.\ Snaith},
       DOI = {10.1007/BF01934324},
       URL = {https://doi.org/10.1007/BF01934324},
}

@article {BO09mb,
    AUTHOR = {Bourgeois, Fr\'ed\'eric and Oancea, Alexandru},
     TITLE = {Symplectic homology, autonomous {H}amiltonians, and
              {M}orse-{B}ott moduli spaces},
   JOURNAL = {Duke Math. J.},
  FJOURNAL = {Duke Mathematical Journal},
    VOLUME = {146},
      YEAR = {2009},
    NUMBER = {1},
     PAGES = {71--174},
      %ISSN = {0012-7094,1547-7398},
   MRCLASS = {53D40},
  MRNUMBER = {2475400},
MRREVIEWER = {Tobias\ Ekholm},
       DOI = {10.1215/00127094-2008-062},
       URL = {https://doi.org/10.1215/00127094-2008-062},
}

@article {BO13gysin,
    AUTHOR = {Bourgeois, Fr\'{e}d\'{e}ric and Oancea, Alexandru},
     TITLE = {The {G}ysin exact sequence for {$S^1$}-equivariant symplectic
              homology},
   JOURNAL = {J. Topol. Anal.},
  FJOURNAL = {Journal of Topology and Analysis},
    VOLUME = {5},
      YEAR = {2013},
    NUMBER = {4},
     PAGES = {361--407},
      %ISSN = {1793-5253},
   MRCLASS = {55N91 (53D40 57R58)},
  MRNUMBER = {3152208},
MRREVIEWER = {Michael J. Usher},
       DOI = {10.1142/S1793525313500210},
       URL = {https://doi.org/10.1142/S1793525313500210},
}

@article {BLMR85,
    AUTHOR = {Berestycki, Henri and Lasry, Jean-Michel and Mancini, Giovanni
              and Ruf, Bernhard},
     TITLE = {Existence of multiple periodic orbits on star-shaped
              {H}amiltonian surfaces},
   JOURNAL = {Comm. Pure Appl. Math.},
  FJOURNAL = {Communications on Pure and Applied Mathematics},
    VOLUME = {38},
      YEAR = {1985},
    NUMBER = {3},
     PAGES = {253--289},
      %ISSN = {0010-3640,1097-0312},
   MRCLASS = {58F05 (58E05)},
  MRNUMBER = {784474},
MRREVIEWER = {A.\ Vanderbauwhede},
       DOI = {10.1002/cpa.3160380302},
       URL = {https://doi.org/10.1002/cpa.3160380302},
}

@ARTICLE{CGGM23,
       author = {\c{C}ineli, Erman and Ginzburg, Viktor L. and G\"urel, Ba\c sak Z. and {Mazzucchelli}, Marco},
        title = "{Invariant Sets and Hyperbolic Closed Reeb Orbits}",
      journal = {arXiv e-prints},
     keywords = {Mathematics - Symplectic Geometry, Mathematics - Dynamical Systems, 53D40, 37J11, 37J46},
         year = 2023,
        month = sep,
          %eid = {arXiv:2309.04576},
        %pages = {arXiv:2309.04576},
          %doi = {10.48550/arXiv.2309.04576},
archivePrefix = {arXiv},
       eprint = {2309.04576},
 primaryClass = {math.SG},
       adsurl = {https://ui.adsabs.harvard.edu/abs/2023arXiv230904576C},
      adsnote = {Provided by the SAO/NASA Astrophysics Data System}
}

@ARTICLE{CGG24,
       author = {\c{C}ineli, Erman and Ginzburg, Viktor L. and G\"{u}rel, Ba\c{s}ak Z.},
        title = "{Closed Orbits of Dynamically Convex Reeb Flows: Towards the HZ- and Multiplicity Conjectures}",
      journal = {arXiv e-prints},
     keywords = {Mathematics - Symplectic Geometry, Mathematics - Dynamical Systems, 53D40, 37J11, 37J46},
         year = 2024,
        month = oct,
          %eid = {arXiv:2410.13093},
        %pages = {arXiv:2410.13093},
          %doi = {10.48550/arXiv.2410.13093},
archivePrefix = {arXiv},
       eprint = {2410.13093},
 primaryClass = {math.SG},
       adsurl = {https://ui.adsabs.harvard.edu/abs/2024arXiv241013093C},
      adsnote = {Provided by the SAO/NASA Astrophysics Data System}
}

@article {CZ84,
    AUTHOR = {Conley, Charles and Zehnder, Eduard},
     TITLE = {Morse-type index theory for flows and periodic solutions for
              {H}amiltonian equations},
   JOURNAL = {Comm. Pure Appl. Math.},
  FJOURNAL = {Communications on Pure and Applied Mathematics},
    VOLUME = {37},
      YEAR = {1984},
    NUMBER = {2},
     PAGES = {207--253},
     %ISSN = {0010-3640,1097-0312},
   MRCLASS = {58E05 (58F05 58F22)},
  MRNUMBER = {733717},
MRREVIEWER = {A.\ Vanderbauwhede},
       DOI = {10.1002/cpa.3160370204},
      URL = {https://doi.org/10.1002/cpa.3160370204},
}

@article {CFHW96,
    AUTHOR = {Cieliebak, K. and Floer, A. and Hofer, H. and Wysocki, K.},
     TITLE = {Applications of symplectic homology. {II}. {S}tability of the
              action spectrum},
   JOURNAL = {Math. Z.},
  FJOURNAL = {Mathematische Zeitschrift},
    VOLUME = {223},
      YEAR = {1996},
    NUMBER = {1},
     PAGES = {27--45},
      %ISSN = {0025-5874,1432-1823},
   MRCLASS = {58F05 (57R99 58E99)},
  MRNUMBER = {1408861},
MRREVIEWER = {Karl\ Friedrich\ Siburg},
       DOI = {10.1007/PL00004267},
       URL = {https://doi.org/10.1007/PL00004267},
}

@article {CGH16,
    AUTHOR = {Cristofaro-Gardiner, Daniel and Hutchings, Michael},
     TITLE = {From one {R}eeb orbit to two},
   JOURNAL = {J. Differential Geom.},
  FJOURNAL = {Journal of Differential Geometry},
    VOLUME = {102},
      YEAR = {2016},
    NUMBER = {1},
     PAGES = {25--36},
      %ISSN = {0022-040X,1945-743X},
   MRCLASS = {53D10 (53D42)},
  MRNUMBER = {3447085},
MRREVIEWER = {Cecilia\ Karlsson},
       URL = {http://projecteuclid.org/euclid.jdg/1452002876},
}

@article {CGHHL23,
    AUTHOR = {Cristofaro-Gardiner, Daniel and Hryniewicz, Umberto and
              Hutchings, Michael and Liu, Hui},
     TITLE = {Contact three-manifolds with exactly two simple {R}eeb orbits},
   JOURNAL = {Geom. Topol.},
  FJOURNAL = {Geometry \& Topology},
    VOLUME = {27},
      YEAR = {2023},
    NUMBER = {9},
     PAGES = {3801--3831},
      ISSN = {1465-3060,1364-0380},
   MRCLASS = {53D42 (37J55 53E50)},
  MRNUMBER = {4674840},
MRREVIEWER = {Alexander\ Fel\cprime shtyn},
       DOI = {10.2140/gt.2023.27.3801},
       URL = {https://doi.org/10.2140/gt.2023.27.3801},
}

@ARTICLE{CGHHL23b,
       author = {Cristofaro-Gardiner, Daniel and Hryniewicz, Umberto and Hutchings, Michael and Liu, Hui},
        title = "{Proof of Hofer-Wysocki-Zehnder's two or infinity conjecture}",
      journal = {arXiv e-prints},
     keywords = {Mathematics - Symplectic Geometry, Mathematics - Dynamical Systems},
         year = 2023,
        month = oct,
          %eid = {arXiv:2310.07636},
        %pages = {arXiv:2310.07636},
          %doi = {10.48550/arXiv.2310.07636},
archivePrefix = {arXiv},
       eprint = {2310.07636},
 primaryClass = {math.SG},
       adsurl = {https://ui.adsabs.harvard.edu/abs/2023arXiv231007636C},
      adsnote = {Provided by the SAO/NASA Astrophysics Data System}
}

@article {DLW16,
    AUTHOR = {Duan, Huagui and Long, Yiming and Wang, Wei},
     TITLE = {The enhanced common index jump theorem for symplectic paths
              and non-hyperbolic closed geodesics on {F}insler manifolds},
   JOURNAL = {Calc. Var. Partial Differential Equations},
  FJOURNAL = {Calculus of Variations and Partial Differential Equations},
    VOLUME = {55},
      YEAR = {2016},
    NUMBER = {6},
     PAGES = {Art. 145, 28},
      %ISSN = {0944-2669,1432-0835},
   MRCLASS = {53C22 (53B40 58E05 58E10)},
  MRNUMBER = {3568050},
MRREVIEWER = {Mircea\ Crasmareanu},
       DOI = {10.1007/s00526-016-1075-7},
       URL = {https://doi.org/10.1007/s00526-016-1075-7},
}

@article {DLLW24,
    AUTHOR = {Duan, Huagui and Liu, Hui and Long, Yiming and Wang, Wei},
     TITLE = {Generalized common index jump theorem with applications to
              closed characteristics on star-shaped hypersurfaces and
              beyond},
   JOURNAL = {J. Funct. Anal.},
  FJOURNAL = {Journal of Functional Analysis},
    VOLUME = {286},
      YEAR = {2024},
    NUMBER = {7},
     PAGES = {Paper No. 110352, 41},
      %ISSN = {0022-1236,1096-0783},
   MRCLASS = {58E05 (34C25 37J46)},
  MRNUMBER = {4701782},
MRREVIEWER = {Maria\ Letizia\ Bertotti},
       DOI = {10.1016/j.jfa.2024.110352},
       URL = {https://doi.org/10.1016/j.jfa.2024.110352},
}

@article {DL17,
    AUTHOR = {Duan, Huagui and Liu, Hui},
     TITLE = {Multiplicity and ellipticity of closed characteristics on
              compact star-shaped hypersurfaces in {$ \mathbf{R}^{2n}$}},
   JOURNAL = {Calc. Var. Partial Differential Equations},
  FJOURNAL = {Calculus of Variations and Partial Differential Equations},
    VOLUME = {56},
      YEAR = {2017},
    NUMBER = {3},
     PAGES = {Paper No. 65, 30},
      %ISSN = {0944-2669,1432-0835},
   MRCLASS = {58E05 (34A26 37J45)},
  MRNUMBER = {3640028},
MRREVIEWER = {Ying\ Lv},
       DOI = {10.1007/s00526-017-1173-1},
       URL = {https://doi.org/10.1007/s00526-017-1173-1},
}

@article {DLLW18,
    AUTHOR = {Duan, Huagui and Liu, Hui and Long, Yiming and Wang, Wei},
     TITLE = {Non-hyperbolic closed characteristics on non-degenerate
              star-shaped hypersurfaces in {$\mathbb{R}^{2n}$}},
   JOURNAL = {Acta Math. Sin. (Engl. Ser.)},
  FJOURNAL = {Acta Mathematica Sinica (English Series)},
    VOLUME = {34},
      YEAR = {2018},
    NUMBER = {1},
     PAGES = {1--18},
      %ISSN = {1439-8516,1439-7617},
   MRCLASS = {58E05 (34C25 37J45 53D12)},
  MRNUMBER = {3735829},
MRREVIEWER = {Xijun\ Hu},
       DOI = {10.1007/s10114-016-6019-9},
      URL = {https://doi.org/10.1007/s10114-016-6019-9},
}

@article {DDE92,
    AUTHOR = {Dell'Antonio, Gianfausto and D'Onofrio, Biancamaria and
              Ekeland, Ivar},
     TITLE = {Les syst\`emes hamiltoniens convexes et pairs ne sont pas
              ergodiques en g\'en\'eral},
   JOURNAL = {C. R. Acad. Sci. Paris S\'er. I Math.},
  FJOURNAL = {Comptes Rendus de l'Acad\'emie des Sciences. S\'erie I.
              Math\'ematique},
    VOLUME = {315},
      YEAR = {1992},
    NUMBER = {13},
     PAGES = {1413--1415},
      %ISSN = {0764-4442},
   MRCLASS = {58F22 (58F05 70H05)},
  MRNUMBER = {1199013},
MRREVIEWER = {Shi\ Qing\ Zhang},
}

@article {EL87,
    AUTHOR = {Ekeland, Ivar and Lassoued, L.},
     TITLE = {Multiplicit\'e{} des trajectoires ferm\'ees de syst\`emes
              hamiltoniens convexes},
   JOURNAL = {Ann. Inst. H. Poincar\'e{} Anal. Non Lin\'eaire},
  FJOURNAL = {Annales de l'Institut Henri Poincar\'e. Analyse Non
              Lin\'eaire},
    VOLUME = {4},
      YEAR = {1987},
    NUMBER = {4},
     PAGES = {307--335},
      %ISSN = {0294-1449},
   MRCLASS = {58F05 (34C25 58E05 58F22)},
  MRNUMBER = {917740},
MRREVIEWER = {Michele\ Matzeu},
       URL = {http://www.numdam.org/item?id=AIHPC_1987__4_4_307_0},
}

@article {EH87,
    AUTHOR = {Ekeland, Ivar and Hofer, H.},
     TITLE = {Convex {H}amiltonian energy surfaces and their periodic
              trajectories},
   JOURNAL = {Comm. Math. Phys.},
  FJOURNAL = {Communications in Mathematical Physics},
    VOLUME = {113},
      YEAR = {1987},
    NUMBER = {3},
     PAGES = {419--469},
      %ISSN = {0010-3616,1432-0916},
   MRCLASS = {58F05 (58E05)},
  MRNUMBER = {925924},
MRREVIEWER = {Shi\ Tao\ Deng},
       URL = {http://projecteuclid.org/euclid.cmp/1104160288},
}

@ARTICLE{Fe20,
       author = {Fender, Elijah},
        title = "{Local Symplectic Homology of Reeb Orbits}",
      journal = {arXiv e-prints},
     keywords = {Mathematics - Symplectic Geometry, 53D40},
         year = 2020,
        month = oct,
          %eid = {arXiv:2010.01438},
        %pages = {arXiv:2010.01438},
          %doi = {10.48550/arXiv.2010.01438},
archivePrefix = {arXiv},
       eprint = {2010.01438},
 primaryClass = {math.SG},
       adsurl = {https://ui.adsabs.harvard.edu/abs/2020arXiv201001438F},
      adsnote = {Provided by the SAO/NASA Astrophysics Data System}
}

@article {GM69a,
    AUTHOR = {Gromoll, Detlef and Meyer, Wolfgang},
     TITLE = {On differentiable functions with isolated critical points},
   JOURNAL = {Topology},
  FJOURNAL = {Topology. An International Journal of Mathematics},
    VOLUME = {8},
      YEAR = {1969},
     PAGES = {361--369},
      %ISSN = {0040-9383},
   MRCLASS = {57.55},
  MRNUMBER = {246329},
MRREVIEWER = {M.\ Klingmann},
       DOI = {10.1016/0040-9383(69)90022-6},
       URL = {https://doi.org/10.1016/0040-9383(69)90022-6},
}

@article {GM69b,
    AUTHOR = {Gromoll, Detlef and Meyer, Wolfgang},
     TITLE = {Periodic geodesics on compact riemannian manifolds},
   JOURNAL = {J. Differential Geometry},
  FJOURNAL = {Journal of Differential Geometry},
    VOLUME = {3},
      YEAR = {1969},
     PAGES = {493--510},
      %ISSN = {0022-040X,1945-743X},
   MRCLASS = {53.72},
  MRNUMBER = {264551},
MRREVIEWER = {W.\ Klingenberg},
       URL = {http://projecteuclid.org/euclid.jdg/1214429070},
}

@article {Gi10,
    AUTHOR = {Ginzburg, Viktor L.},
     TITLE = {The {C}onley conjecture},
   JOURNAL = {Ann. of Math. (2)},
  FJOURNAL = {Annals of Mathematics. Second Series},
    VOLUME = {172},
      YEAR = {2010},
    NUMBER = {2},
     PAGES = {1127--1180},
      %ISSN = {0003-486X,1939-8980},
   MRCLASS = {53D40 (37J05 53D35)},
  MRNUMBER = {2680488},
MRREVIEWER = {Hai-Long\ Her},
       DOI = {10.4007/annals.2010.172.1129},
       URL = {https://doi.org/10.4007/annals.2010.172.1129},
}

@article {GG10,
    AUTHOR = {Ginzburg, Viktor L. and G\"urel, Ba\c sak Z.},
     TITLE = {Local {F}loer homology and the action gap},
   JOURNAL = {J. Symplectic Geom.},
  FJOURNAL = {The Journal of Symplectic Geometry},
    VOLUME = {8},
      YEAR = {2010},
    NUMBER = {3},
     PAGES = {323--357},
      %ISSN = {1527-5256,1540-2347},
   MRCLASS = {53D40},
  MRNUMBER = {2684510},
MRREVIEWER = {Vincent\ Humili\`ere},
       DOI = {10.4310/jsg.2010.v8.n3.a4},
       URL = {https://doi.org/10.4310/jsg.2010.v8.n3.a4},
}

@article {GG20,
    AUTHOR = {Ginzburg, Viktor L. and G\"{u}rel, Ba\c{s}ak Z.},
     TITLE = {Lusternik-{S}chnirelmann theory and closed {R}eeb orbits},
   JOURNAL = {Math. Z.},
  FJOURNAL = {Mathematische Zeitschrift},
    VOLUME = {295},
      YEAR = {2020},
    NUMBER = {1-2},
     PAGES = {515--582},
      %ISSN = {0025-5874},
   MRCLASS = {53D40 (37J55 58E05)},
  MRNUMBER = {4100023},
MRREVIEWER = {Vittorio Coti Zelati},
       DOI = {10.1007/s00209-019-02361-2},
       URL = {https://doi.org/10.1007/s00209-019-02361-2},
}

@article {GGM18,
    AUTHOR = {Ginzburg, Viktor L. and G\"urel, Ba\c sak Z. and Macarini,
              Leonardo},
     TITLE = {Multiplicity of closed {R}eeb orbits on prequantization
              bundles},
   JOURNAL = {Israel J. Math.},
  FJOURNAL = {Israel Journal of Mathematics},
    VOLUME = {228},
      YEAR = {2018},
    NUMBER = {1},
     PAGES = {407--453},
      %ISSN = {0021-2172,1565-8511},
   MRCLASS = {53D10 (34C25 37D40 53D50)},
  MRNUMBER = {3874849},
MRREVIEWER = {Mircea\ Crasmareanu},
       DOI = {10.1007/s11856-018-1769-y},
       URL = {https://doi.org/10.1007/s11856-018-1769-y},
}

@article {GG15,
    AUTHOR = {Ginzburg, Viktor L. and G\"oren, Yusuf},
     TITLE = {Iterated index and the mean {E}uler characteristic},
   JOURNAL = {J. Topol. Anal.},
  FJOURNAL = {Journal of Topology and Analysis},
    VOLUME = {7},
      YEAR = {2015},
    NUMBER = {3},
     PAGES = {453--481},
      %ISSN = {1793-5253,1793-7167},
   MRCLASS = {53D42 (37C25 37J45 70H12)},
  MRNUMBER = {3346929},
MRREVIEWER = {Alexander\ Fel\cprime shtyn},
       DOI = {10.1142/S179352531550017X},
       URL = {https://doi.org/10.1142/S179352531550017X},
}

@article {GHHM13,
    AUTHOR = {Ginzburg, Viktor L. and Hein, Doris and Hryniewicz, Umberto L.
              and Macarini, Leonardo},
     TITLE = {Closed {R}eeb orbits on the sphere and symplectically
              degenerate maxima},
   JOURNAL = {Acta Math. Vietnam.},
  FJOURNAL = {Acta Mathematica Vietnamica},
    VOLUME = {38},
      YEAR = {2013},
    NUMBER = {1},
     PAGES = {55--78},
      %ISSN = {0251-4184,2315-4144},
   MRCLASS = {53D42 (37J10 53D40 57R17)},
  MRNUMBER = {3089878},
MRREVIEWER = {Rostislav\ Matveyev},
       DOI = {10.1007/s40306-012-0002-z},
       URL = {https://doi.org/10.1007/s40306-012-0002-z},
}

@article {GK16,
    AUTHOR = {Gutt, Jean and Kang, Jungsoo},
     TITLE = {On the minimal number of periodic orbits on some hypersurfaces
              in {$\mathbb{R}^{2n}$}},
   JOURNAL = {Ann. Inst. Fourier (Grenoble)},
  FJOURNAL = {Universit\'e{} de Grenoble. Annales de l'Institut Fourier},
    VOLUME = {66},
      YEAR = {2016},
    NUMBER = {6},
     PAGES = {2485--2505},
      %ISSN = {0373-0956,1777-5310},
   MRCLASS = {53D10 (37J55)},
  MRNUMBER = {3580178},
MRREVIEWER = {Gabriele\ Benedetti},
       DOI = {10.5802/aif.3069},
       URL = {https://doi.org/10.5802/aif.3069},
}

@article {Gir84,
    AUTHOR = {Girardi, Mario},
     TITLE = {Multiple orbits for {H}amiltonian systems on starshaped
              surfaces with symmetries},
   JOURNAL = {Ann. Inst. H. Poincar\'e{} Anal. Non Lin\'eaire},
  FJOURNAL = {Annales de l'Institut Henri Poincar\'e. Analyse Non
              Lin\'eaire},
    VOLUME = {1},
      YEAR = {1984},
    NUMBER = {4},
     PAGES = {285--294},
      %ISSN = {0294-1449},
   MRCLASS = {58F22 (58F05 70H05)},
  MRNUMBER = {778975},
MRREVIEWER = {Wei\ Yue\ Ding},
       URL = {http://www.numdam.org/item?id=AIHPC_1984__1_4_285_0},
}

@article {HO17,
    AUTHOR = {Hu, Xijun and Ou, Yuwei},
     TITLE = {Stability of closed characteristics on compact convex
              hypersurfaces in {$\mathbf{R}^{2n}$}},
   JOURNAL = {J. Fixed Point Theory Appl.},
  FJOURNAL = {Journal of Fixed Point Theory and Applications},
    VOLUME = {19},
      YEAR = {2017},
    NUMBER = {1},
     PAGES = {585--600},
      %ISSN = {1661-7738,1661-7746},
   MRCLASS = {58E05 (34C25 37J45 53D12)},
  MRNUMBER = {3625085},
MRREVIEWER = {Maria\ Letizia\ Bertotti},
       DOI = {10.1007/s11784-016-0366-0},
       URL = {https://doi.org/10.1007/s11784-016-0366-0},
}

@article {Hin09,
    AUTHOR = {Hingston, Nancy},
     TITLE = {Subharmonic solutions of {H}amiltonian equations on tori},
   JOURNAL = {Ann. of Math. (2)},
  FJOURNAL = {Annals of Mathematics. Second Series},
    VOLUME = {170},
      YEAR = {2009},
    NUMBER = {2},
     PAGES = {529--560},
      %ISSN = {0003-486X,1939-8980},
   MRCLASS = {53D35 (37J45)},
  MRNUMBER = {2552101},
MRREVIEWER = {Tobias\ Ekholm},
       DOI = {10.4007/annals.2009.170.529},
       URL = {https://doi.org/10.4007/annals.2009.170.529},
}

@article {HL02,
    AUTHOR = {Hu, Xijun and Long, Yiming},
     TITLE = {Closed characteristics on non-degenerate star-shaped
              hypersurfaces in {$\mathbb{R}^{2n}$}},
   JOURNAL = {Sci. China Ser. A},
  FJOURNAL = {Science in China. Series A. Mathematics},
    VOLUME = {45},
      YEAR = {2002},
    NUMBER = {8},
     PAGES = {1038--1052},
      %ISSN = {1006-9283,1862-2763},
   MRCLASS = {37J45 (34C25 53D12 58E05)},
  MRNUMBER = {1942918},
MRREVIEWER = {James\ F.\ Reineck},
       DOI = {10.1007/BF02879987},
       URL = {https://doi.org/10.1007/BF02879987},
}

@article {HWZ98,
    AUTHOR = {Hofer, H. and Wysocki, K. and Zehnder, E.},
     TITLE = {The dynamics on three-dimensional strictly convex energy
              surfaces},
   JOURNAL = {Ann. of Math. (2)},
  FJOURNAL = {Annals of Mathematics. Second Series},
    VOLUME = {148},
      YEAR = {1998},
    NUMBER = {1},
     PAGES = {197--289},
      %ISSN = {0003-486X,1939-8980},
   MRCLASS = {58F05 (58D10 58F22)},
  MRNUMBER = {1652928},
MRREVIEWER = {Matthias\ Schwarz},
       DOI = {10.2307/120994},
       URL = {https://doi.org/10.2307/120994},
}

@article {HWZ03,
    AUTHOR = {Hofer, H. and Wysocki, K. and Zehnder, E.},
     TITLE = {Finite energy foliations of tight three-spheres and
              {H}amiltonian dynamics},
   JOURNAL = {Ann. of Math. (2)},
  FJOURNAL = {Annals of Mathematics. Second Series},
    VOLUME = {157},
      YEAR = {2003},
    NUMBER = {1},
     PAGES = {125--255},
      %ISSN = {0003-486X,1939-8980},
   MRCLASS = {53D35 (37C27 37C85 37J05 57R30)},
  MRNUMBER = {1954266},
MRREVIEWER = {Kai\ Cieliebak},
       DOI = {10.4007/annals.2003.157.125},
       URL = {https://doi.org/10.4007/annals.2003.157.125},
}

@article {HM15,
    AUTHOR = {Hryniewicz, Umberto L. and Macarini, Leonardo},
     TITLE = {Local contact homology and applications},
   JOURNAL = {J. Topol. Anal.},
  FJOURNAL = {Journal of Topology and Analysis},
    VOLUME = {7},
      YEAR = {2015},
    NUMBER = {2},
     PAGES = {167--238},
      %ISSN = {1793-5253,1793-7167},
   MRCLASS = {53D10 (37J45 53D42)},
  MRNUMBER = {3326300},
MRREVIEWER = {Nick\ Sheridan},
       DOI = {10.1142/S1793525315500119},
       URL = {https://doi.org/10.1142/S1793525315500119},
}

@article {LL16,
    AUTHOR = {Liu, Hui and Long, Yiming},
     TITLE = {The existence of two closed characteristics on every compact
              star-shaped hypersurface in {$\mathbb{R}^4$}},
   JOURNAL = {Acta Math. Sin. (Engl. Ser.)},
  FJOURNAL = {Acta Mathematica Sinica (English Series)},
    VOLUME = {32},
      YEAR = {2016},
    NUMBER = {1},
     PAGES = {40--53},
      %ISSN = {1439-8516,1439-7617},
   MRCLASS = {58E05 (34C25 37J45)},
  MRNUMBER = {3431159},
MRREVIEWER = {Addolorata\ Salvatore},
       DOI = {10.1007/s10114-014-4108-1},
       URL = {https://doi.org/10.1007/s10114-014-4108-1},
}

@article {LL17,
    AUTHOR = {Liu, Hui and Long, Yiming},
     TITLE = {Irrationally elliptic closed characteristics on symmetric
              compact star-shaped hypersurfaces in {$\mathbf{R}^4$}},
   JOURNAL = {J. Fixed Point Theory Appl.},
  FJOURNAL = {Journal of Fixed Point Theory and Applications},
    VOLUME = {19},
      YEAR = {2017},
    NUMBER = {1},
     PAGES = {263--280},
      %ISSN = {1661-7738,1661-7746},
   MRCLASS = {53D42 (34C25 37J45 58E05)},
  MRNUMBER = {3625071},
MRREVIEWER = {Addolorata\ Salvatore},
       DOI = {10.1007/s11784-016-0352-6},
       URL = {https://doi.org/10.1007/s11784-016-0352-6},
}

@article {LWZ20,
    AUTHOR = {Liu, Hui and Wang, Chongzhi and Zhang, Duanzhi},
     TITLE = {Elliptic and non-hyperbolic closed characteristics on compact
              convex {P}-cyclic symmetric hypersurfaces in {$\mathbf R^{2n}$}},
   JOURNAL = {Calc. Var. Partial Differential Equations},
  FJOURNAL = {Calculus of Variations and Partial Differential Equations},
    VOLUME = {59},
      YEAR = {2020},
    NUMBER = {1},
     PAGES = {Paper No. 24, 20},
      %ISSN = {0944-2669,1432-0835},
   MRCLASS = {58E05 (34C25 37J12)},
  MRNUMBER = {4048333},
MRREVIEWER = {Huagui\ Duan},
       DOI = {10.1007/s00526-019-1681-2},
       URL = {https://doi.org/10.1007/s00526-019-1681-2},
}

@article {Long90,
    AUTHOR = {Long, Yiming},
     TITLE = {Maslov-type index, degenerate critical points, and
              asymptotically linear {H}amiltonian systems},
   JOURNAL = {Sci. China Ser. A},
  FJOURNAL = {Science in China (Scientia Sinica). Series A. Mathematics,
              Physics, Astronomy},
    VOLUME = {33},
      YEAR = {1990},
    NUMBER = {12},
     PAGES = {1409--1419},
      %ISSN = {1001-6511},
   MRCLASS = {58F22 (34C25 58E05 58F05 70H05)},
  MRNUMBER = {1090484},
MRREVIEWER = {Norman\ Dancer},
}

@article {Long97,
    AUTHOR = {Long, Yiming},
     TITLE = {A {M}aslov-type index theory for symplectic paths},
   JOURNAL = {Topol. Methods Nonlinear Anal.},
  FJOURNAL = {Topological Methods in Nonlinear Analysis},
    VOLUME = {10},
      YEAR = {1997},
    NUMBER = {1},
     PAGES = {47--78},
      %ISSN = {1230-3429},
   MRCLASS = {58E05 (34B05 34C99 58F05 58F22)},
  MRNUMBER = {1646611},
MRREVIEWER = {Liviu\ I.\ Nicolaescu},
       DOI = {10.12775/TMNA.1997.021},
       URL = {https://doi.org/10.12775/TMNA.1997.021},
}

@article {Long99,
    AUTHOR = {Long, Yiming},
     TITLE = {Bott formula of the {M}aslov-type index theory},
   JOURNAL = {Pacific J. Math.},
  FJOURNAL = {Pacific Journal of Mathematics},
    VOLUME = {187},
      YEAR = {1999},
    NUMBER = {1},
     PAGES = {113--149},
      %ISSN = {0030-8730,1945-5844},
   MRCLASS = {37J45 (34C25 53D12 58E05)},
  MRNUMBER = {1674313},
MRREVIEWER = {James\ F.\ Reineck},
       DOI = {10.2140/pjm.1999.187.113},
       URL = {https://doi.org/10.2140/pjm.1999.187.113},
}

@article {Long00,
    AUTHOR = {Long, Yiming},
     TITLE = {Precise iteration formulae of the {M}aslov-type index theory
              and ellipticity of closed characteristics},
   JOURNAL = {Adv. Math.},
  FJOURNAL = {Advances in Mathematics},
    VOLUME = {154},
      YEAR = {2000},
    NUMBER = {1},
     PAGES = {76--131},
      %ISSN = {0001-8708,1090-2082},
   MRCLASS = {37J45 (34C25 53D12)},
  MRNUMBER = {1780096},
MRREVIEWER = {James\ F.\ Reineck},
       DOI = {10.1006/aima.2000.1914},
       URL = {https://doi.org/10.1006/aima.2000.1914},
}

@article {LZ02,
    AUTHOR = {Long, Yiming and Zhu, Chaofeng},
     TITLE = {Closed characteristics on compact convex hypersurfaces in
              {$\mathbf R^{2n}$}},
   JOURNAL = {Ann. of Math. (2)},
  FJOURNAL = {Annals of Mathematics. Second Series},
    VOLUME = {155},
      YEAR = {2002},
    NUMBER = {2},
     PAGES = {317--368},
      %ISSN = {0003-486X,1939-8980},
   MRCLASS = {37J45 (34C25 53D12)},
  MRNUMBER = {1906590},
MRREVIEWER = {Karl\ Friedrich\ Siburg},
       DOI = {10.2307/3062120},
       URL = {https://doi.org/10.2307/3062120},
}

@article {Mcl16,
    AUTHOR = {McLean, Mark},
     TITLE = {Reeb orbits and the minimal discrepancy of an isolated
              singularity},
   JOURNAL = {Invent. Math.},
  FJOURNAL = {Inventiones Mathematicae},
    VOLUME = {204},
      YEAR = {2016},
    NUMBER = {2},
     PAGES = {505--594},
      %ISSN = {0020-9910,1432-1297},
   MRCLASS = {57R17 (14N35)},
  MRNUMBER = {3489704},
MRREVIEWER = {Yildiray\ Ozan},
       DOI = {10.1007/s00222-015-0620-x},
       URL = {https://doi.org/10.1007/s00222-015-0620-x},
}

@article {Mcl12,
    AUTHOR = {McLean, Mark},
     TITLE = {Local {F}loer homology and infinitely many simple {R}eeb
              orbits},
   JOURNAL = {Algebr. Geom. Topol.},
  FJOURNAL = {Algebraic \& Geometric Topology},
    VOLUME = {12},
      YEAR = {2012},
    NUMBER = {4},
     PAGES = {1901--1923},
      %ISSN = {1472-2747,1472-2739},
   MRCLASS = {53D25 (53C22)},
  MRNUMBER = {2994824},
MRREVIEWER = {Umberto\ Leone\ Hryniewicz},
       DOI = {10.2140/agt.2012.12.1901},
      URL = {https://doi.org/10.2140/agt.2012.12.1901},
}

@article {Rab78,
    AUTHOR = {Rabinowitz, Paul H.},
     TITLE = {Periodic solutions of {H}amiltonian systems},
   JOURNAL = {Comm. Pure Appl. Math.},
  FJOURNAL = {Communications on Pure and Applied Mathematics},
    VOLUME = {31},
      YEAR = {1978},
    NUMBER = {2},
     PAGES = {157--184},
      %ISSN = {0010-3640},
   MRCLASS = {58F05 (34C25)},
  MRNUMBER = {467823},
MRREVIEWER = {A. D. Bruno},
       DOI = {10.1002/cpa.3160310203},
       URL = {https://doi.org/10.1002/cpa.3160310203},
}

@article {SZ92,
    AUTHOR = {Salamon, Dietmar and Zehnder, Eduard},
     TITLE = {Morse theory for periodic solutions of {H}amiltonian systems
              and the {M}aslov index},
   JOURNAL = {Comm. Pure Appl. Math.},
  FJOURNAL = {Communications on Pure and Applied Mathematics},
    VOLUME = {45},
      YEAR = {1992},
    NUMBER = {10},
     PAGES = {1303--1360},
      ISSN = {0010-3640,1097-0312},
   MRCLASS = {58E05 (34C25 58F22 70H05)},
  MRNUMBER = {1181727},
MRREVIEWER = {Yong-Geun\ Oh},
       DOI = {10.1002/cpa.3160451004},
       URL = {https://doi.org/10.1002/cpa.3160451004},
}

@article {Szu88,
    AUTHOR = {Szulkin, Andrzej},
     TITLE = {Morse theory and existence of periodic solutions of convex
              {H}amiltonian systems},
   JOURNAL = {Bull. Soc. Math. France},
  FJOURNAL = {Bulletin de la Soci\'et\'e{} Math\'ematique de France},
    VOLUME = {116},
      YEAR = {1988},
    NUMBER = {2},
     PAGES = {171--197},
      %ISSN = {0037-9484},
   MRCLASS = {58F05 (34C25 58E05 58F22 70H05)},
  MRNUMBER = {971559},
MRREVIEWER = {Michele\ Matzeu},
       URL = {http://www.numdam.org/item?id=BSMF_1988__116_2_171_0},
}

@article {Viterbo90,
    AUTHOR = {Viterbo, Claude},
     TITLE = {A new obstruction to embedding {L}agrangian tori},
   JOURNAL = {Invent. Math.},
  FJOURNAL = {Inventiones Mathematicae},
    VOLUME = {100},
      YEAR = {1990},
    NUMBER = {2},
     PAGES = {301--320},
      %ISSN = {0020-9910,1432-1297},
   MRCLASS = {58F05 (53C23 58E05 58F22)},
  MRNUMBER = {1047136},
MRREVIEWER = {Michel\ Willem},
       DOI = {10.1007/BF01231188},
       URL = {https://doi.org/10.1007/BF01231188},
}

@article {Vit89,
    AUTHOR = {Viterbo, Claude},
     TITLE = {Equivariant {M}orse theory for starshaped {H}amiltonian
              systems},
   JOURNAL = {Trans. Amer. Math. Soc.},
  FJOURNAL = {Transactions of the American Mathematical Society},
    VOLUME = {311},
      YEAR = {1989},
    NUMBER = {2},
     PAGES = {621--655},
      %ISSN = {0002-9947,1088-6850},
   MRCLASS = {58F05 (58E05 58F35)},
  MRNUMBER = {978370},
MRREVIEWER = {Wei\ Yue\ Ding},
       DOI = {10.2307/2001144},
       URL = {https://doi.org/10.2307/2001144},
}

@article {Wang14,
    AUTHOR = {Wang, Wei},
     TITLE = {Irrationally elliptic closed characteristics on compact convex
              hypersurfaces in {$\mathbf{R}^6$}},
   JOURNAL = {J. Funct. Anal.},
  FJOURNAL = {Journal of Functional Analysis},
    VOLUME = {267},
      YEAR = {2014},
    NUMBER = {3},
     PAGES = {799--841},
      %ISSN = {0022-1236,1096-0783},
   MRCLASS = {58E05 (37C75 37J45)},
  MRNUMBER = {3212724},
MRREVIEWER = {Xijun\ Hu},
       DOI = {10.1016/j.jfa.2014.05.014},
       URL = {https://doi.org/10.1016/j.jfa.2014.05.014},
}

@article {Wang22,
    AUTHOR = {Wang, Wei},
     TITLE = {Irrationally elliptic closed characteristics on compact convex
              hypersurfaces in {$\mathbf{R}^{2n}$}},
   JOURNAL = {J. Funct. Anal.},
  FJOURNAL = {Journal of Functional Analysis},
    VOLUME = {282},
      YEAR = {2022},
    NUMBER = {1},
     PAGES = {Paper No. 109269, 29},
      %ISSN = {0022-1236,1096-0783},
   MRCLASS = {58E05 (34C25 37J46)},
  MRNUMBER = {4323511},
MRREVIEWER = {Huagui\ Duan},
       DOI = {10.1016/j.jfa.2021.109269},
       URL = {https://doi.org/10.1016/j.jfa.2021.109269},
}

@article {Wang09,
    AUTHOR = {Wang, Wei},
     TITLE = {Stability of closed characteristics on compact convex
              hypersurfaces in {$\mathbb{R}^6$}},
   JOURNAL = {J. Eur. Math. Soc. (JEMS)},
  FJOURNAL = {Journal of the European Mathematical Society (JEMS)},
    VOLUME = {11},
      YEAR = {2009},
    NUMBER = {3},
     PAGES = {575--596},
      %ISSN = {1435-9855,1435-9863},
   MRCLASS = {37J45 (37J25 58E05)},
  MRNUMBER = {2505442},
MRREVIEWER = {Maria\ Letizia\ Bertotti},
       DOI = {10.4171/JEMS/161},
       URL = {https://doi.org/10.4171/JEMS/161},
}

@article {Wei71,
    AUTHOR = {Weinstein, Alan},
     TITLE = {Symplectic manifolds and their {L}agrangian submanifolds},
   JOURNAL = {Advances in Math.},
  FJOURNAL = {Advances in Mathematics},
    VOLUME = {6},
      YEAR = {1971},
     PAGES = {329--346},
      %ISSN = {0001-8708},
   MRCLASS = {57.50},
  MRNUMBER = {286137},
MRREVIEWER = {D.\ G.\ Ebin},
       DOI = {10.1016/0001-8708(71)90020-X},
       URL = {https://doi.org/10.1016/0001-8708(71)90020-X},
}

@article {WHL07,
    AUTHOR = {Wang, Wei and Hu, Xijun and Long, Yiming},
     TITLE = {Resonance identity, stability, and multiplicity of closed
              characteristics on compact convex hypersurfaces},
   JOURNAL = {Duke Math. J.},
  FJOURNAL = {Duke Mathematical Journal},
    VOLUME = {139},
      YEAR = {2007},
    NUMBER = {3},
     PAGES = {411--462},
      %ISSN = {0012-7094,1547-7398},
   MRCLASS = {37J45 (34C25 58E05)},
  MRNUMBER = {2350849},
MRREVIEWER = {Karl\ Friedrich\ Siburg},
       DOI = {10.1215/S0012-7094-07-13931-0},
       URL = {https://doi.org/10.1215/S0012-7094-07-13931-0},
}

@article {Wang16b,
    AUTHOR = {Wang, Wei},
     TITLE = {Closed characteristics on compact convex hypersurfaces in
              {${\bf R}^8$}},
   JOURNAL = {Adv. Math.},
  FJOURNAL = {Advances in Mathematics},
    VOLUME = {297},
      YEAR = {2016},
     PAGES = {93--148},
      %ISSN = {0001-8708,1090-2082},
   MRCLASS = {53A07 (34C25 37J45 58E05)},
  MRNUMBER = {3498795},
MRREVIEWER = {Thomas\ Bartsch},
       DOI = {10.1016/j.aim.2016.03.044},
       URL = {https://doi.org/10.1016/j.aim.2016.03.044},
}

@article {Wang16a,
    AUTHOR = {Wang, Wei},
     TITLE = {Existence of closed characteristics on compact convex
              hypersurfaces in {$\mathbf{R}^{2n}$}},
   JOURNAL = {Calc. Var. Partial Differential Equations},
  FJOURNAL = {Calculus of Variations and Partial Differential Equations},
    VOLUME = {55},
      YEAR = {2016},
    NUMBER = {1},
     PAGES = {Art. 2, 25},
      %ISSN = {0944-2669,1432-0835},
   MRCLASS = {58E05 (34C25 37J45)},
  MRNUMBER = {3441279},
MRREVIEWER = {Xijun\ Hu},
       DOI = {10.1007/s00526-015-0945-8},
       URL = {https://doi.org/10.1007/s00526-015-0945-8},
}

@article {Wei78,
    AUTHOR = {Weinstein, Alan},
     TITLE = {Periodic orbits for convex {H}amiltonian systems},
   JOURNAL = {Ann. of Math. (2)},
  FJOURNAL = {Annals of Mathematics. Second Series},
    VOLUME = {108},
      YEAR = {1978},
    NUMBER = {3},
     PAGES = {507--518},
      %ISSN = {0003-486X},
   MRCLASS = {58F05},
  MRNUMBER = {512430},
MRREVIEWER = {J.\ Moser},
       DOI = {10.2307/1971185},
       URL = {https://doi.org/10.2307/1971185},
}

@book {AD14,
    AUTHOR = {Audin, Mich\`ele and Damian, Mihai},
     TITLE = {Morse theory and {F}loer homology},
    SERIES = {Universitext},
      NOTE = {Translated from the 2010 French original by Reinie Ern\'e},
 PUBLISHER = {Springer, London; EDP Sciences, Les Ulis},
      YEAR = {2014},
     %PAGES = {xiv+596},
      %ISBN = {978-1-4471-5495-2; 978-1-4471-5496-9; 978-2-7598-0704-8},
   MRCLASS = {53-02 (53D40 58E05)},
  MRNUMBER = {3155456},
MRREVIEWER = {Sonja\ Hohloch},
       DOI = {10.1007/978-1-4471-5496-9},
       URL = {https://doi.org/10.1007/978-1-4471-5496-9},
}

@incollection {Eke86,
    AUTHOR = {Ekeland, Ivar},
     TITLE = {An index theory for periodic solutions of convex {H}amiltonian
              systems},
 BOOKTITLE = {Nonlinear functional analysis and its applications, {P}art 1
              ({B}erkeley, {C}alif., 1983)},
    SERIES = {Proc. Sympos. Pure Math.},
    VOLUME = {45, Part 1},
     PAGES = {395--423},
 PUBLISHER = {Amer. Math. Soc., Providence, RI},
      YEAR = {1986},
      %ISBN = {0-8218-1467-2},
   MRCLASS = {58E05 (58F05)},
  MRNUMBER = {843575},
MRREVIEWER = {Michele\ Matzeu},
       DOI = {10.1090/pspum/045.1/843575},
       URL = {https://doi.org/10.1090/pspum/045.1/843575},
}

@book {Eke90,
    AUTHOR = {Ekeland, Ivar},
     TITLE = {Convexity methods in {H}amiltonian mechanics},
    SERIES = {Ergebnisse der Mathematik und ihrer Grenzgebiete (3) [Results
              in Mathematics and Related Areas (3)]},
    VOLUME = {19},
 PUBLISHER = {Springer-Verlag, Berlin},
      YEAR = {1990},
     %PAGES = {x+247},
      %ISBN = {3-540-50613-6},
   MRCLASS = {58F05 (34C25 58E05 58E40 58F22 70H05)},
  MRNUMBER = {1051888},
MRREVIEWER = {D.\ Pascali},
       DOI = {10.1007/978-3-642-74331-3},
       URL = {https://doi.org/10.1007/978-3-642-74331-3},
}

@book {Long02,
    AUTHOR = {Long, Yiming},
     TITLE = {Index theory for symplectic paths with applications},
    SERIES = {Progress in Mathematics},
    VOLUME = {207},
 PUBLISHER = {Birkh\"auser Verlag, Basel},
      YEAR = {2002},
     %PAGES = {xxiv+380},
      %ISBN = {3-7643-6647-8},
   MRCLASS = {37J45 (34C25 37-02 53D40 58-02 58E05 70H05)},
  MRNUMBER = {1898560},
MRREVIEWER = {Thomas\ Bartsch},
       DOI = {10.1007/978-3-0348-8175-3},
       URL = {https://doi.org/10.1007/978-3-0348-8175-3},
}

@incollection {LZ90,
    AUTHOR = {Long, Yiming and Zehnder, Eduard},
     TITLE = {Morse-theory for forced oscillations of asymptotically linear
              {H}amiltonian systems},
 BOOKTITLE = {Stochastic processes, physics and geometry ({A}scona and
              {L}ocarno, 1988)},
     PAGES = {528--563},
 PUBLISHER = {World Sci. Publ., Teaneck, NJ},
      YEAR = {1990},
   MRCLASS = {58E05 (34C25 58F05 58F22 70H05)},
  MRNUMBER = {1124230},
MRREVIEWER = {Vittorio Coti Zelati},
}

@book {MS17,
    AUTHOR = {McDuff, Dusa and Salamon, Dietmar},
     TITLE = {Introduction to symplectic topology},
    SERIES = {Oxford Graduate Texts in Mathematics},
   EDITION = {Third},
 PUBLISHER = {Oxford University Press, Oxford},
      YEAR = {2017},
    % PAGES = {xi+623},
      %ISBN = {978-0-19-879490-5; 978-0-19-879489-9},
   MRCLASS = {53D35 (53D40 57R17 57R57 57R58)},
  MRNUMBER = {3674984},
MRREVIEWER = {Hansj\"org\ Geiges},
       DOI = {10.1093/oso/9780198794899.001.0001},
       URL = {https://doi.org/10.1093/oso/9780198794899.001.0001},
}

@incollection {Sal99,
    AUTHOR = {Salamon, Dietmar},
     TITLE = {Lectures on {F}loer homology},
 BOOKTITLE = {Symplectic geometry and topology ({P}ark {C}ity, {UT}, 1997)},
    SERIES = {IAS/Park City Math. Ser.},
    VOLUME = {7},
     PAGES = {143--229},
 PUBLISHER = {Amer. Math. Soc., Providence, RI},
      YEAR = {1999},
      %ISBN = {0-8218-0838-9},
   MRCLASS = {53D40 (37J45 53D45 57R17 57R58)},
  MRNUMBER = {1702944},
MRREVIEWER = {David\ E.\ Hurtubise},
       DOI = {10.1016/S0165-2427(99)00127-0},
       URL = {https://doi.org/10.1016/S0165-2427(99)00127-0},
}

@book {Sch93,
    AUTHOR = {Schwarz, Matthias},
     TITLE = {Morse homology},
    SERIES = {Progress in Mathematics},
    VOLUME = {111},
 PUBLISHER = {Birkh\"auser Verlag, Basel},
      YEAR = {1993},
     %PAGES = {x+235},
      %ISBN = {3-7643-2904-1},
   MRCLASS = {58E05 (55N35 57R70)},
  MRNUMBER = {1239174},
MRREVIEWER = {Daniel\ M.\ Burns, Jr.},
       DOI = {10.1007/978-3-0348-8577-5},
       URL = {https://doi.org/10.1007/978-3-0348-8577-5},
}
%%%%%%%%%%% Bibliography

%%%%%%%%%%% Affiliation
%\newpage
\vspace{4mm}
{\footnotesize

\address{
\noindent\textsc{Xiaorui Li:}
\href{mailto:lixiaorui@mail.nankai.edu.cn}
{lixiaorui@mail.nankai.edu.cn}\\
School of Mathematics, Shandong University, China}

\vspace{1mm}
\address{
\noindent\textsc{Hui Liu:}
\href{mailto:huiliu00031514@whu.edu.cn.}
{huiliu00031514@whu.edu.cn.}\\
School of Mathematics and Statistics, Wuhan University, China}

\vspace{1mm}
\address{
\noindent\textsc{Wei Wang:}
\href{wangwei@math.pku.edu.cn}
{wangwei@math.pku.edu.cn}\\
Key Laboratory of Pure and Applied Mathematics\\
School of Mathematical Science, Peking University, China}
}

%%%%%%%%%%%%%%%%%%%%%%%%%%%%%%%%%%%%%%%%%%%%%%%%%%%%%%%%%%%%
%%%%%%%%%%%%%%%%%%%%%%%%%%%%%%%%%%%%%%%%%%%%%%%%%%%%%%%%%%%%
%%%%%%%%%%%%%%%%%%%%%%%%%%%%%%%%%%%%%%%%%%%%%%%%%%%%%%%%%%%%
\end{sloppypar}
\end{document}